\documentclass[a4paper,12pt,oneside, reqno]{amsart}

\usepackage{amsaddr}
\usepackage[bb=boondox]{mathalfa}
\usepackage{cite}
\usepackage{hyperref}
\usepackage{lmodern} 
\usepackage{bm}

\usepackage{tikz}

\numberwithin{equation}{section}
\numberwithin{figure}{section}

\allowdisplaybreaks[3]

\newcommand{\RR}{\mathbb{R}}
\newcommand{\CC}{\mathbb{C}}

\newcommand{\NN}{\mathbb{N}}

\newcommand{\cD}{\mathcal{D}}
\newcommand{\cE}{\mathcal{E}}
\newcommand{\cF}{\mathcal{F}}

\newcommand{\cS}{\mathcal{S}}

\newcommand{\eps}{\varepsilon}
\newcommand{\rmi}{\mathrm{i}}

\newcommand{\loc}{\mathrm{loc}}

\DeclareMathOperator{\supp}{\mathrm{supp}}
\DeclareMathOperator{\dist}{\mathrm{dist}}
\DeclareMathOperator{\ran}{\mathrm{ran}}

\DeclareMathOperator{\dom}{\mathrm{dom}}
\DeclareMathOperator{\spec}{\mathrm{spec}}

\DeclareMathOperator{\spece}{\mathrm{spec}_\mathrm{ess}}

\DeclareMathOperator{\vspan}{span}

\newtheorem{theorem}{Theorem}[section]
\newtheorem{corollary}[theorem]{Corollary}
\newtheorem{lemma}[theorem]{Lemma}
\newtheorem{proposition}[theorem]{Proposition}
\theoremstyle{definition}
\newtheorem{definition}[theorem]{Definition}

\newtheorem{remark}[theorem]{Remark}

\renewcommand{\Tilde}{\widetilde}

\newcommand{\dd}{\mathrm{d}}
\newcommand{\dxy}{\dd x\,\dd y}

\begin{document}
	\title[]{\large Non-resonant angles\\[\smallskipamount]
		in~a~Robin eigenvalue problem}

	\author[]{Konstantin Pankrashkin}
	\address{\footnotesize Carl von Ossietzky Universit\"at Oldenburg,
		Fakult\"at V -- Mathematik und Naturwissenschaften,  Institut f\"ur Mathematik, Ammerl\"ander Heerstr. 114--118, 26129~Oldenburg (Oldb), Germany}
	\email{konstantin.pankrashkin@uol.de}

	\date{}

\begin{abstract}
In an earlier study of the Laplacian eigenvalues with strongly attractive Robin boundary conditions in polygons it was found that if each corner is either concave or non-resonant, then the asymptotic behavior of the individual eigenvalues is determined by an effective one-dimensional operator on the boundary with Dirichlet boundary conditions at the corners. The non-resonance condition was formulated in terms of an eigenvalue bound on truncated sectors, and it was only shown to be satisfied by the right and obtuse opening angles. In the present work, we determine the full range of non-resonant angles: it turns out that the non-resonance condition fails only for the opening angles $\frac{\pi}{2n+1}$ with $n=1,2,3,\dots$, and these critical values are equivalently characterized in terms of the jumps of the eigenvalue counting function and the existence of bounded generalized eigenfunctions. This resolves the numerically supported conjecture stated by Dauge and Lafranche. 
\end{abstract}
%

	\maketitle


%
%
%
%
%

	\setcounter{tocdepth}{2}
	
	\tableofcontents

\section{Introduction}

Over the past few decades, considerable attention was paid to the asymptotic behavior of the eigenvalues
of Laplacians with strongly attractive Robin boundary conditions. Namely, for a bounded Lipschitz domain $\Omega\subset\RR^n$, $n\ge 2$, and $\alpha\in\RR$ denote by $R^\Omega_\alpha$ the self-adjoint Laplacian in $L^2(\Omega)$ acting as $u\mapsto-\Delta u$ on the functions $u$ satisfying the Robin boundary condition $\partial_\nu u=\alpha u$ on $\partial\Omega$, where $\partial_\nu$ is the outer normal derivative, and we denote by $E_j(R^\Omega_\alpha)$
its $j$-th eigenvalue (numbered in non-decreasing order and counting multiplicities).
 One is interested in the asymptotic behavior of $E_j(R^\Omega_\alpha)$ with a fixed $j\in \NN$ for $\alpha\to\infty$. The expression for the associated bilinear form $r^\Omega_\alpha$,
\[
r^\Omega_\alpha(u,u)=\int_\Omega |\nabla u|^2\dd x-\alpha \int_{\partial\Omega}\ u^2\dd s,\quad u\in H^1(\Omega),
\]
easily shows that for each fixed $j$ one has $E_j(R^\Omega_\alpha)\xrightarrow{\alpha\to\infty}-\infty$ and the associated eigenfunctions
concentrate near $\partial\Omega$. Several works \cite{HK, KKR, pp16} found that in the case of $\Omega$ with smooth boundary
important spectral properties of $R^\Omega_\alpha$ for large $\alpha$ are governed by the effective operator
$L^\alpha:=-\Delta_{\partial\Omega}-\alpha H$ on $\partial\Omega$, where $\Delta_{\partial\Omega}$ is the Laplace-Beltrami operator on $\partial\Omega$ and $H$ is the mean curvature of $\partial\Omega$ (defined as the sum of the principal curvatures, with the convention that $H\ge 0$ for convex domains). In particular,
for each fixed $j\in\NN$ one has
\[
E_j(R^\Omega_\alpha)=-\alpha^2+E_j(L^\alpha)+O(1),\quad \alpha\to\infty,
\]
see \cite{pp16}. However, the analysis of these works is based on the existence of tubular coordinates near $\partial\Omega$.
Even the initial steps of this approach fail for non-smooth $\partial\Omega$, and the asymptotic behavior
for the non-smooth case is generally different, see e.g. the reviews in \cite{DFK,bp,LP}.

Progress in the reduction to effective operators in non-smooth settings was achie\-ved in~\cite{kobp}. To simplify the discussion, we neglect the role of the curvature
and look at the role of boundary singularities. Namely, assume that $n=2$ and $\Omega\subset\RR^2$ is a polygon with $M$  vertices  $A_1,\dots, A_M$ (with the cyclic convention $A_{0}:=A_M$ and $A_{M+1}:=A_1$). The boundary
of $\Omega$ is composed of the line segments $\Gamma_j:=[A_j,A_{j+1}]$ having the length $\ell_j:=|\Gamma_j|=|A_{j+1}-A_j|$, $j\in\{1,\dots,M\}$ (with the same cyclic convention). 
We further denote by $2\theta_j\in (0,2\pi)\setminus\{\pi\}$ the angle between $\Gamma_{j-1}$ and $\Gamma_j$ measured inside $\Omega$ (the condition $2\theta_j\ne\pi$ excludes artificial vertices). A vertex $A_j$ is called convex 
for $\theta_j<\frac{\pi}{2}$ and concave for $\theta_j>\frac{\pi}{2}$. We refer to Fig.~\ref{fig-domain} for an illustration.

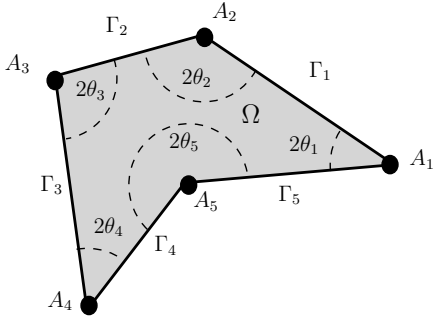
\begin{figure}
	
	\centering
	
	\tikzset{every picture/.style={line width=0.75pt}} 
	
	\begin{minipage}{0.45\textwidth}\centering
	\scalebox{0.7}{\begin{tikzpicture}[x=0.75pt,y=0.75pt,yscale=-1,xscale=1]
		
		\draw  [fill={rgb, 255:red, 155; green, 155; blue, 155 }  ,fill opacity=0.42 ][line width=1.5]  (163.3,38.19) -- (300.46,130.98) -- (150.44,144) -- (80.44,235.16) -- (57.58,67.49) -- cycle ;
		\draw  [fill={rgb, 255:red, 0; green, 0; blue, 0 }  ,fill opacity=1 ] (291.89,130.98) .. controls (291.89,127.38) and (294.45,124.47) .. (297.6,124.47) .. controls (300.76,124.47) and (303.32,127.38) .. (303.32,130.98) .. controls (303.32,134.57) and (300.76,137.49) .. (297.6,137.49) .. controls (294.45,137.49) and (291.89,134.57) .. (291.89,130.98) -- cycle ;
		\draw  [fill={rgb, 255:red, 0; green, 0; blue, 0 }  ,fill opacity=1 ] (147.59,145.63) .. controls (147.59,142.03) and (150.15,139.12) .. (153.3,139.12) .. controls (156.46,139.12) and (159.02,142.03) .. (159.02,145.63) .. controls (159.02,149.22) and (156.46,152.14) .. (153.3,152.14) .. controls (150.15,152.14) and (147.59,149.22) .. (147.59,145.63) -- cycle ;
		\draw  [fill={rgb, 255:red, 0; green, 0; blue, 0 }  ,fill opacity=1 ] (159.02,39.81) .. controls (159.02,36.22) and (161.58,33.3) .. (164.73,33.3) .. controls (167.89,33.3) and (170.45,36.22) .. (170.45,39.81) .. controls (170.45,43.41) and (167.89,46.33) .. (164.73,46.33) .. controls (161.58,46.33) and (159.02,43.41) .. (159.02,39.81) -- cycle ;
		\draw  [fill={rgb, 255:red, 0; green, 0; blue, 0 }  ,fill opacity=1 ] (51.86,70.74) .. controls (51.86,67.15) and (54.42,64.23) .. (57.58,64.23) .. controls (60.73,64.23) and (63.29,67.15) .. (63.29,70.74) .. controls (63.29,74.34) and (60.73,77.26) .. (57.58,77.26) .. controls (54.42,77.26) and (51.86,74.34) .. (51.86,70.74) -- cycle ;
		\draw  [fill={rgb, 255:red, 0; green, 0; blue, 0 }  ,fill opacity=1 ] (76.15,231.91) .. controls (76.15,228.31) and (78.71,225.4) .. (81.86,225.4) .. controls (85.02,225.4) and (87.58,228.31) .. (87.58,231.91) .. controls (87.58,235.5) and (85.02,238.42) .. (81.86,238.42) .. controls (78.71,238.42) and (76.15,235.5) .. (76.15,231.91) -- cycle ;
		\draw  [draw opacity=0][dash pattern={on 4.5pt off 4.5pt}] (72.94,191.61) .. controls (75.37,191.21) and (77.88,191) .. (80.44,191) .. controls (90.17,191) and (99.14,194.03) .. (106.33,199.12) -- (80.44,231) -- cycle ; \draw  [dash pattern={on 4.5pt off 4.5pt}] (72.94,191.61) .. controls (75.37,191.21) and (77.88,191) .. (80.44,191) .. controls (90.17,191) and (99.14,194.03) .. (106.33,199.12) ;  
		\draw  [draw opacity=0][dash pattern={on 4.5pt off 4.5pt}] (99.71,55.27) .. controls (101.11,59.62) and (101.87,64.26) .. (101.87,69.08) .. controls (101.87,91.1) and (86.05,109.34) .. (65.43,112.51) -- (59.01,69.08) -- cycle ; \draw  [dash pattern={on 4.5pt off 4.5pt}] (99.71,55.27) .. controls (101.11,59.62) and (101.87,64.26) .. (101.87,69.08) .. controls (101.87,91.1) and (86.05,109.34) .. (65.43,112.51) ;  
		\draw  [draw opacity=0][dash pattern={on 4.5pt off 4.5pt}] (201.62,63.15) .. controls (194.34,76.77) and (180.23,86) .. (164.01,86) .. controls (142.48,86) and (124.66,69.74) .. (121.6,48.55) -- (164.01,42.08) -- cycle ; \draw  [dash pattern={on 4.5pt off 4.5pt}] (201.62,63.15) .. controls (194.34,76.77) and (180.23,86) .. (164.01,86) .. controls (142.48,86) and (124.66,69.74) .. (121.6,48.55) ;  
		\draw  [draw opacity=0][dash pattern={on 4.5pt off 4.5pt}] (254.98,135.62) .. controls (254.82,134.1) and (254.74,132.55) .. (254.74,130.98) .. controls (254.74,122.09) and (257.32,113.82) .. (261.75,106.9) -- (297.6,130.98) -- cycle ; \draw  [dash pattern={on 4.5pt off 4.5pt}] (254.98,135.62) .. controls (254.82,134.1) and (254.74,132.55) .. (254.74,130.98) .. controls (254.74,122.09) and (257.32,113.82) .. (261.75,106.9) ;  
		\draw  [draw opacity=0][dash pattern={on 4.5pt off 4.5pt}] (123.73,177.42) .. controls (115.54,169.42) and (110.44,158.14) .. (110.44,145.63) .. controls (110.44,121.37) and (129.63,101.71) .. (153.3,101.71) .. controls (174.91,101.71) and (192.79,118.1) .. (195.74,139.41) -- (153.3,145.63) -- cycle ; \draw  [dash pattern={on 4.5pt off 4.5pt}] (123.73,177.42) .. controls (115.54,169.42) and (110.44,158.14) .. (110.44,145.63) .. controls (110.44,121.37) and (129.63,101.71) .. (153.3,101.71) .. controls (174.91,101.71) and (192.79,118.1) .. (195.74,139.41) ;  
		
		\draw (310.18,119.12) node [anchor=north west][inner sep=0.75pt]    {$A_{1}$};
		\draw (168.59,13.31) node [anchor=north west][inner sep=0.75pt]    {$A_{2}$};
		\draw (20.14,50.75) node [anchor=north west][inner sep=0.75pt]    {$A_{3}$};
		\draw (50.72,220.59) node [anchor=north west][inner sep=0.75pt]    {$A_{4}$};
		\draw (155.3,149.03) node [anchor=north west][inner sep=0.75pt]    {$A_{5}$};
		\draw (189.64,86.07) node [anchor=north west][inner sep=0.75pt]  [font=\large]  {$\Omega $};
		\draw (238.1,55.63) node [anchor=north west][inner sep=0.75pt]    {$\Gamma _{1}$};
		\draw (92.37,23.07) node [anchor=north west][inner sep=0.75pt]    {$\Gamma _{2}$};
		\draw (45.5,136.96) node [anchor=north west][inner sep=0.75pt]    {$\Gamma _{3}$};
		\draw (127.37,180.52) node [anchor=north west][inner sep=0.75pt]    {$\Gamma _{4}$};
		\draw (216.24,144.42) node [anchor=north west][inner sep=0.75pt]    {$\Gamma _{5}$};
		\draw (83.72,167.19) node [anchor=north west][inner sep=0.75pt]    {$2\theta _{4}$};
		\draw (70.87,69.56) node [anchor=north west][inner sep=0.75pt]    {$2\theta _{3}$};
		\draw (146.87,60.56) node [anchor=north west][inner sep=0.75pt]    {$2\theta _{2}$};
		\draw (223.87,108.56) node [anchor=north west][inner sep=0.75pt]    {$2\theta _{1}$};
		\draw (137.87,106.56) node [anchor=north west][inner sep=0.75pt]    {$2\theta _{5}$};

	\end{tikzpicture}}	
\end{minipage}
	\begin{minipage}{0.5\textwidth}
		\caption{An example of a polygon $\Omega$ with $M=5$. The vertex $A_5$ is concave, the other vertices are convex.}\label{fig-domain}
\end{minipage}
\end{figure}

A special role in the analysis of $R^\Omega_\alpha$ is played by the Robin Laplacians on infinite sectors. Namely, for $\theta\in(0,\pi)$  denote by $\cS_\theta$ the following infinite sector of opening angle~$2\theta$:
\[
\cS_\theta=\big\{ (x_1,x_2)\in\mathbb R^2: -\theta<\arg(x_1+\rmi x_2)<\theta\big\},
\]
see Figure~\ref{fig-sector} for an illustration. Clearly, $\cS_\theta$ is convex if and only if $\theta\le\frac{\pi}{2}$, while for $\theta=\frac{\pi}{2}$ one obtains simply the half-plane $(0,\infty)\times\RR$. 	

\begin{figure}[b]
	
	\centering

	\tikzset{every picture/.style={line width=0.75pt}} 

	\begin{minipage}{0.35\textwidth}\centering
		\scalebox{0.7}{\begin{tikzpicture}[x=0.75pt,y=0.75pt,yscale=-1,xscale=1]
		
		\draw  (77.5,151) -- (295.2,151)(91.5,64.6) -- (91.5,229) (288.2,146) -- (295.2,151) -- (288.2,156) (86.5,71.6) -- (91.5,64.6) -- (96.5,71.6)  ;
		\draw  [draw opacity=0][fill={rgb, 255:red, 155; green, 155; blue, 155 }  ,fill opacity=0.55 ] (250.8,81) -- (250.8,221) -- (90.4,151) -- cycle ;
		\draw [line width=2.25]    (90.4,151) -- (265.6,74.2) ;
		\draw [line width=2.25]    (90.4,151) -- (264.4,226.2) ;
		\draw  [draw opacity=0] (148.87,125.94) .. controls (152.17,133.63) and (154,142.1) .. (154,151) .. controls (154,159.69) and (152.26,167.98) .. (149.1,175.53) -- (90.4,151) -- cycle ; \draw  [color={rgb, 255:red, 74; green, 74; blue, 74 }  ,draw opacity=1 ] (148.87,125.94) .. controls (152.17,133.63) and (154,142.1) .. (154,151) .. controls (154,159.69) and (152.26,167.98) .. (149.1,175.53) ;  
		
		\draw (265.4,131) node [anchor=north west][inner sep=0.75pt]    {$x_{1}$};
		\draw (98,74) node [anchor=north west][inner sep=0.75pt]    {$x_{2}$};
		\draw (157.2,126.2) node [anchor=north west][inner sep=0.75pt]    {$\theta $};
		\draw (157.2,155.8) node [anchor=north west][inner sep=0.75pt]    {$\theta $};
		\draw (210,114.4) node [anchor=north west][inner sep=0.75pt]    {$\cS_\theta$};

\end{tikzpicture}}\end{minipage}\begin{minipage}{0.6\textwidth}
\caption{The infinite sector $\cS_\theta$. Here $\theta\in(0,\frac{\pi}{2})$.}\label{fig-sector}	
\end{minipage}

\end{figure}
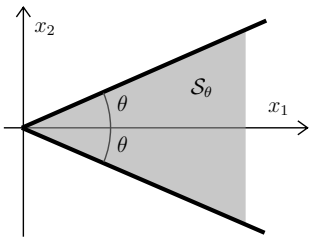

Let us denote by $T_\theta$ the Laplacian $u\mapsto-\Delta u$ on $\cS_\theta$ with the Robin boundary condition $\partial_\nu u=u$
on $\partial\cS_\theta$, realized as a self-adjoint operator in $L^2(\cS_\theta)$. It is known that the essential spectrum of $T_\theta$ is always $[-1,\infty)$, i.e. it is independent of $\theta$, and it has $\kappa(\theta)<\infty$ discrete
eigenvalues $\cE_1(\theta),\dots\cE_{\kappa(\theta)}(\theta)$, while $\kappa(\theta)=0$ for $\theta\ge \frac{\pi}{2}$
(i.e. there are no discrete eigenvalues at all if the sector $\cS_\theta$ is concave),
and $\kappa(\theta)\ge 1$ for $\theta<\pi/2$. Based on these numbers we associate with $\Omega$
the following natural quantities:
\begin{align*}
	K&:=\kappa(\theta_1)+\dots +\kappa(\theta_M),\\
	\cE&:=\text{ the disjoint union } \bigsqcup_{j=1}^M\big\{ \cE_n(\theta_j):\ n=1,\dots,\kappa(\theta_j)\big\},\\
	\cE_n&:=\text{ the $n$-th element of $\cE$ when numbered in the non-decreasing order.}
\end{align*}

\begin{remark}
The finiteness of the discrete spectrum of $T_\theta$ and some estimates for the counting function $\kappa$
and the eigenvalues $\cE_n(\theta)$ were established in \cite{khp} using abstract techniques. These results were then applied
in subsequent works \cite{dl,khalile,kobp}  to study domains with corners. The authors of all these works
were not aware of the earlier paper~\cite{kmz} that established explicit analytic expressions for the eigenvalues $\cE_n(\theta)$
and the associated eigenfunctions and showed that they exhaust the whole negative point spectrum: a detailed overview is given in Subsection~\ref{ssec-sectors}. For the current understanding we just note that the function $\theta\mapsto \kappa(\theta)$ is monotonically decreasing and right-continuous, with jumps at the critical values
\begin{equation}
	\label{theta-critical}
\theta_n:=\dfrac{\pi}{2(2n+1)},\quad n=0,1,2,\dots,\quad \kappa(\theta_n^-)-\kappa(\theta_n)=1.
\end{equation}
Note that the same values were independently predicted numerically in \cite{dl}.
\end{remark}

It was shown in \cite{khalile} that for each fixed $j\in\{1,\dots,K\}$ one has, with a suitable $c>0$,
\[
E_j(R^\Omega_\alpha)=\cE_j \alpha^2+O(e^{-c\alpha}),\quad \alpha\to\infty,
\]
and the respective eigenfunctions are strongly localized near the corners. Due to these observations, the first
$K$ eigenvalues of $R^\Omega_\alpha$ are referred to as \emph{corner-induced}. It is also straightforward to see that
$E_{K+n}(R^\Omega_\alpha)=-\alpha^2+o(\alpha^2)$ for any fixed $n\in\NN$, see \cite{khp}, but one has the natural question
whether a more detailed asymptotic expansion can be obtained. An important step in this direction was made in \cite{kobp}, and it is based on the notion of a non-resonant corner.

\begin{figure}
	\centering
	
	\begin{minipage}{0.4\textwidth}\centering
		\scalebox{0.75}{\begin{tikzpicture}[x=0.75pt,y=0.75pt,yscale=-1,xscale=1]
			
			\draw  (97.5,171) -- (315.2,171)(111.5,84.6) -- (111.5,249) (308.2,166) -- (315.2,171) -- (308.2,176) (106.5,91.6) -- (111.5,84.6) -- (116.5,91.6)  ;
			\draw [color={rgb, 255:red, 155; green, 155; blue, 155 }  ,draw opacity=1 ][line width=1.5]    (110.4,171) -- (285.6,94.2) ;
			\draw [color={rgb, 255:red, 155; green, 155; blue, 155 }  ,draw opacity=1 ][line width=1.5]    (110.4,171) -- (284.4,246.2) ;
			\draw  [color={rgb, 255:red, 0; green, 0; blue, 0 }  ,draw opacity=1 ][fill={rgb, 255:red, 155; green, 155; blue, 155 }  ,fill opacity=0.55 ][line width=2.25]  (270,170.4) -- (242,227.6) -- (110.4,171) -- (243.6,113.2) -- cycle ;
			\draw  [draw opacity=0] (168.87,145.94) .. controls (172.17,153.63) and (174,162.1) .. (174,171) .. controls (174,179.69) and (172.26,187.98) .. (169.1,195.53) -- (110.4,171) -- cycle ; \draw  [color={rgb, 255:red, 74; green, 74; blue, 74 }  ,draw opacity=1 ] (168.87,145.94) .. controls (172.17,153.63) and (174,162.1) .. (174,171) .. controls (174,179.69) and (172.26,187.98) .. (169.1,195.53) ;  
			\draw  [fill={rgb, 255:red, 0; green, 0; blue, 0 }  ,fill opacity=1 ] (267,170.4) .. controls (267,168.74) and (268.34,167.4) .. (270,167.4) .. controls (271.66,167.4) and (273,168.74) .. (273,170.4) .. controls (273,172.06) and (271.66,173.4) .. (270,173.4) .. controls (268.34,173.4) and (267,172.06) .. (267,170.4) -- cycle ;
			\draw  [fill={rgb, 255:red, 0; green, 0; blue, 0 }  ,fill opacity=1 ] (108.4,171) .. controls (108.4,169.34) and (109.74,168) .. (111.4,168) .. controls (113.06,168) and (114.4,169.34) .. (114.4,171) .. controls (114.4,172.66) and (113.06,174) .. (111.4,174) .. controls (109.74,174) and (108.4,172.66) .. (108.4,171) -- cycle ;
			\draw  [fill={rgb, 255:red, 0; green, 0; blue, 0 }  ,fill opacity=1 ] (240.6,113.2) .. controls (240.6,111.54) and (241.94,110.2) .. (243.6,110.2) .. controls (245.26,110.2) and (246.6,111.54) .. (246.6,113.2) .. controls (246.6,114.86) and (245.26,116.2) .. (243.6,116.2) .. controls (241.94,116.2) and (240.6,114.86) .. (240.6,113.2) -- cycle ;
			\draw  [fill={rgb, 255:red, 0; green, 0; blue, 0 }  ,fill opacity=1 ] (239,227.6) .. controls (239,225.94) and (240.34,224.6) .. (242,224.6) .. controls (243.66,224.6) and (245,225.94) .. (245,227.6) .. controls (245,229.26) and (243.66,230.6) .. (242,230.6) .. controls (240.34,230.6) and (239,229.26) .. (239,227.6) -- cycle ;
			\draw  [dash pattern={on 0.75pt off 0.75pt}] (236.97,213.78) -- (246.4,218.18) -- (242,227.6) -- (232.58,223.2) -- cycle ;
			\draw  [dash pattern={on 0.75pt off 0.75pt}] (243.37,112.57) -- (247.81,121.97) -- (238.4,126.41) -- (233.97,117) -- cycle ;
			
			\draw (285.4,152) node [anchor=north west][inner sep=0.75pt]    {$x_{1}$};
			\draw (117,90.8) node [anchor=north west][inner sep=0.75pt]    {$x_{2}$};
			\draw (90,152) node [anchor=north west][inner sep=0.75pt]    {$O$};
			\draw (177.2,146.2) node [anchor=north west][inner sep=0.75pt]    {$\theta $};
			\draw (177.2,175.8) node [anchor=north west][inner sep=0.75pt]    {$\theta $};
			\draw (222,92) node [anchor=north west][inner sep=0.75pt]    {$A_{r}^{+}$};
			\draw (222,230) node [anchor=north west][inner sep=0.75pt]    {$A_{r}^{-}$};
			\draw (272,173.8) node [anchor=north west][inner sep=0.75pt]    {$B_{r}$};
			\draw (150.18,128.33) node [anchor=north west][inner sep=0.75pt]  [rotate=-336.39]  {$\partial _{\nu } u=u$};
			\draw (150.23,191.55) node [anchor=north west][inner sep=0.75pt]  [rotate=-23.49]  {$\partial _{\nu } u=u$};
			\draw (259.6,118.8) node [anchor=north west][inner sep=0.75pt]    {$\partial _{\nu } u=0$};
			\draw (260,202) node [anchor=north west][inner sep=0.75pt]    {$\partial _{\nu } u=0$};
			\draw (224.8,142.4) node [anchor=north west][inner sep=0.75pt]    {$\cS_{\theta }^{r}$};

		\end{tikzpicture}}
	\end{minipage}
	\begin{minipage}{0.55\textwidth}\raggedright \caption{\small The truncated sector $\cS^r_\theta$. One has $|OA^\pm_r|=r$.}\label{figure-intro}\end{minipage}
\end{figure}
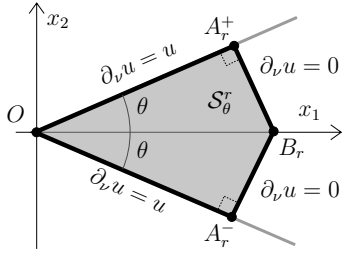

For a chosen $\theta\in(0,\frac{\pi}{2})$, for each $r>0$ denote by $A^\pm_r$ the two points lying on the two boundary rays of the sector $\cS_\theta$ at the distance $r>0$ from the origin $O$, and further denote
by $B_r$ the intersection point of the straight lines passing through $A^\pm_r$ perpendicularly to the boundary, see Figure~\ref{figure-intro}. Denote by $\cS^r_\theta$ the bounded quadrangle $OA^+_r B_r A^-_r$, and further denote
by $\Tilde N^r_{\theta}$ the Laplacian $u\mapsto-\Delta u$ in $\cS^r_\theta$ with the Robin boundary condition $\partial_\nu u=u$ at the sides $OA^\pm_r$ and the Neumann boundary condition at the sides $A^\pm_rB_r$.
It is easy to see that the first $\kappa(\theta)$ eigenvalues of $\Tilde N^r_\theta$ converge to the respective eigenvalues
of the Robin Laplacian $T_\theta$ on the infinite sector $\cS_\theta$ for $r\to\infty$, and the non-resonance condition
concerns the asymptotics of the next eigenvalue. Namely, a half-angle $\theta\in(0,\frac{\pi}{2})$ is called \emph{non-resonant}, if for some $C>0$ one has
\begin{equation}
	\label{eq-nonres}
E_{\kappa(\theta)+1}(\Tilde N^r_\theta)\ge-1+ \dfrac{C}{r^2} \text{ for large $r$},
\end{equation}
and is called  \emph{resonant} otherwise. The following result holds true \cite[Thm.~1.1]{kobp}:
\begin{theorem}\label{thm0}
	Assume that each corner of $\Omega$ is either concave or non-resonant, i.e. for each $j\in\{1,\dots, M\}$
	one has either $\theta_j>\frac{\pi}{2}$ or \eqref{eq-nonres} holds for $\theta:=\theta_j$. Then
	for each $n\in\NN$ one has
	\[
	E_{K+n}(R^\Omega_\alpha)=-\alpha^2+E_n\Big(\bigoplus_{j=1}^M L_j\Big)+O\Big(\dfrac{\log\alpha}{\sqrt{\alpha}}\Big),\quad \alpha\to\infty,
	\]
	where $L_j$ is the Dirichlet Laplacian on $(0,\ell_j)$.	
\end{theorem}

Note that the effective operator $\bigoplus_{j=1}^M L_j$ can be interpreted as the Laplacian on $\partial\Omega$ with Dirichlet boundary conditions imposed at all the corners. The above result is useful only if one can practically check the non-resonance condition at least for some values of $\theta$. In \cite[Prop.~3.10]{kobp}
one shows (by combining a separation of variables and monotonicity) that \eqref{eq-nonres} is satisfied at least
by right and obtuse corners, i.e. for $\theta\in[\frac{\pi}{4},\frac{\pi}{2})$. At the same time, it was seen
in \cite[Sec.~6.1]{kobp} that it is \emph{not} satisfied for $\theta=\frac{\pi}{6}$: this was done in an indirect way
by showing that the asymptotic behavior of the eigenvalues of $R^\Omega_\alpha$ for the equilateral triangle differs from the one stated
in Theorem~\ref{thm0}. Note that the condition~\eqref{eq-nonres} is easily adapted to other operator types, see e.g. \cite{bkp}, and, as already mentioned in \cite[Sec.~6.3]{kobp}, it is motivated by similar problem settings in waveguides.

The condition \eqref{eq-nonres} was subsequently analyzed numerically  and popularized by Dauge and Lafranche, see e.g.  \cite{dl}, and they arrived at the following conjecture: \emph{all half-angles $\theta$ that differ from the critical values $\theta_n$ in \eqref{theta-critical} are non-resonant} (or equivalently: a half-angle $\theta$ is non-resonant provided it is \emph{not} a point of discontinuity of the eigenvalue counting function $\kappa$). In the present work we prove this conjecture.
The link between the two quite different characterizations of non-resonant half-angles
is made using an additional notion which seems new in the context of Robin Laplacians but is a well-established object in  the analysis of waveguides, see e.g. \cite{bakh,grieser,naz1,p17}.
\begin{definition}\label{defin-threshold}
	We say that the Robin Laplacian $T_\theta$ in the infinite sector $\cS_\theta$ has a \emph{threshold resonance} if there exists a weak solution $0\not\equiv F\in L^\infty(\cS_\theta)$ of the boundary value problem
	\begin{equation}
		\label{dff1}
	-\Delta F=-F \text{ in }\cS_\theta,\quad \partial_\nu F=F \text{ on }\partial\cS_\theta.		
	\end{equation}
	The problem \eqref{dff1} is understood as follows: $F\in H^1(U)$ for any bounded subdomain $U\subset\cS_\theta$, and for any $\varphi\in H^1(\cS_\theta)$ with bounded support it holds that
	\[
	\int_{\cS_\theta}\nabla F\cdot\nabla\varphi\,\dd x-\int_{\partial\cS_\theta}F\varphi\,\dd s
	=-\int_{\cS_\theta}F\varphi\,\dd x.
	\]
	While the problem~\eqref{dff1} has formally the same form  as the eigenvalue problem at the edge $-1$ of the essential spectrum of $T_\theta$,
	the ``generalized eigenfunction'' $F$ is only assumed to be bounded, but is \emph{not} assumed to belong to $L^2(\cS_\theta)$.
\end{definition}

Our main results can be summarized as follows.
\begin{theorem}\label{thm-main}
For each half-angle $\theta\in (0,\frac{\pi}{2})$ the following three statements
are equivalent:
\begin{itemize}
	\item[(A)] $\theta$ is non-resonant in the sense of \eqref{eq-nonres},
	\item[(B)] the Robin Laplacian in the sector $\cS_\theta$ has \underline{no} threshold resonance,
	\item[(C)] the eigenvalue counting function $\kappa$ is constant near $\theta$.
\end{itemize}
\end{theorem}
By taking into account the values \eqref{theta-critical} at which jumps of $\kappa$ occur, we deduce that the above Dauge-Lafranche conjecture holds true:
\begin{corollary}
	The non-resonance condition~\eqref{eq-nonres} fails precisely for the critical half-angles
	\[
	\theta\in\Big\{\dfrac{\pi}{2(2n+1)}:\, n\in\NN\Big\}.
	\]
\end{corollary}	

This shows that Theorem~\ref{thm0} covers a much larger class of polygons compared to what was explicitly established in~\cite{kobp}.

The proof of Theorem~\ref{thm-main} is quite involved and occupies the rest of the text. In Section~\ref{sec2}
we first summarize the necessary information about the Robin Laplacian in the infinite sectors, including 
Ursell solutions describing generalized eigenfunctions~\cite{kmz,ursell}. Then we recall some elliptic estimates that will allow us to control the boundedness of the solutions to \eqref{dff1}. The equivalence (B)$\iff$(C) is shown in Section~\ref{sec3}, and actually it covers the whole range $\theta\in(0,\pi)$.
It should be said that the implication (C)$\implies$(B) is quite standard: it uses the idea that
a non-trivial solution of \eqref{dff1} supplies an additional test function, and it only uses the sector geometry
and the monotonicity with respect to the angle, while the direction (B)$\implies$(C) is essentially a direct consequence
of the existence of Ursell solutions from Section~\ref{sec2}. The proof of the equivalence (A)$\iff$(B) is covered in Section~\ref{sec4}. The implication (A)$\implies$(B) is quite straightforward but uses the specific form of Ursell solutions
and uses the same idea as above (these solutions provide an additional test function for $\Tilde N^r_\theta$).
The implication (B)$\implies$(A) is the most difficult part, and it combines a truncation argument with some ideas
of the paper~\cite{bakh}  dealing with threshold resonances in waveguides. The key observation is that if the non-resonance
condition~\eqref{eq-nonres} fails, then the respective eigenfunctions converge (in the sense made precise below)
to a solution of \eqref{dff1}, and the main efforts are put in showing that this solution is not identically zero.
Overall, our proof of the implication (C)$\implies$(A) does not use the explicit expressions
of Robin eigenvalues and eigenfunctions in the sectors and is potentially adaptable to other classes of differential operators in sector-like geometries.

\section{Preparations}\label{sec2}

\subsection{Robin Laplacians in infinite sectors}\label{ssec-sectors}	

In this subsection we describe in greater detail some properties of the Robin Laplacians $T_\theta$
in the sectors
\[
\cS_\theta=\big\{ (x_1,x_2)\in\mathbb R^2: -\theta<\arg(x_1+\rmi x_2)<\theta\big\},\quad \theta\in(0,\pi),
\]
see Figure~\ref{fig-sector}. Recall that $T_\theta$ is defined as the self-adjoint operator in $L^2(\cS_\theta)$
acting as the Laplacian $u\mapsto-\Delta u$ with the Robin boundary condition $\partial_\nu u=u$ on $\partial\cS_\theta$, and the associated bilinear form $t_\theta$
is
\[
t_\theta(u,v)=\int_{\cS_\theta}\nabla u\cdot\nabla v\,\dd x-\int_{\partial\cS_\theta}uv\,\dd s,
\quad u,v\in H^1(\cS_\theta).
\]

For  $\theta\in(0,\frac{\pi}{2}]$ and $n\in\NN$ with $(2n-1)\theta\le \frac{\pi}{2}$
consider the following functions $\psi_{n,\theta}$ defined on $\cS_\theta$:
\begin{equation}
	\label{urs00}
\begin{aligned}
	\psi_{n,\theta}(x_1,x_2)&:=\sum_{m=0}^{n-1}A_{n,m,\theta}\varphi_{n,m,\theta}(x_1,x_2),\\
	\varphi_{n,m,\theta}(x_1,x_2)&:=\exp\Big(-\dfrac{1}{\sin\big((2n-1)\theta\big)}\big(x_1\cos(2m\theta)-x_2\sin(2 m \theta)\big)\Big)\\
		&\quad +\exp\Big(-\dfrac{1}{\sin\big((2n-1)\theta\big)}\big(x_1\cos(2m\theta)+x_2\sin(2 m \theta)\big)\Big),
\end{aligned}		
\end{equation}
with
\begin{align*}
	A_{n,0,\theta}&:=1,\\
	A_{n,m,\theta}&:=2(-1)^m\prod_{j=1}^m \dfrac{\tan\big((n-j)\theta\big)}{\tan\big((n+j-1)\theta\big)},\quad m\in\{1,\dots,n-1\},
\end{align*}
which are known as Ursell's solutions of the sloping beach problem~\cite{ursell}.
Namely, if one denotes
\begin{equation*}
	\begin{aligned}
\cS_\theta^+&:=\cS_\theta\cap \big\{(x_1,x_2): x_2>0\big\},\\
\Gamma_\theta^+&:=\partial\cS_\theta\cap \big\{(x_1,x_2): x_2>0\big\},
\end{aligned}
\end{equation*}
then
\begin{equation}
	\label{urseq}
	\left\{
\begin{aligned}
-\Delta \psi_{n,\theta}&=-\dfrac{1}{\sin^2\big((2n-1)\theta\big)}\psi_{n,\theta} \text{ in }\cS_\theta^+,\\
\partial_\nu \psi_{n,\theta}&=\psi_{n,\theta} \text{ on } \Gamma_\theta^+,\\
\partial_\nu \psi_{n,\theta}&=0 \text{ on } \partial\cS_\theta^+\setminus \Gamma_\theta^+,
\end{aligned}\right.
\end{equation}
see the original paper \cite{ursell} (which uses different coordinates)
or \cite{lyal} for a direct verification. 

	\begin{proposition}\label{prop-sector}
		For any $\theta\in(0,\pi)$ the following statements hold true:
		\begin{itemize}
			\item[(i)] $\spece T_{\theta}=[-1, +\infty)$  for any $\theta\in (0,\pi)$,
			\item[(ii)] the discrete spectrum of $T_{\theta}$ is at most finite.
	\end{itemize}
	More precisely, if one denotes
			\[
			\text{$\kappa(\theta):=\text{the number of discrete eigenvalues of }T_{\theta}$ (counted with multiplicities),}
			\]
			then
			\[
				\kappa(\theta)=\begin{cases}
		0,& \text{if }\theta\ge \frac{\pi}{2},\\[\medskipamount]
		\Big\lceil \dfrac{\pi}{4\theta}-\dfrac{1}{2}\Big\rceil, & \text{if }\theta<\frac{\pi}{2},
	\end{cases}
	\]
	and for  $\theta<\frac{\pi}{2}$ the discrete spectrum of $T_{\theta}$ consists of the simple eigenvalues
	\[
	\cE_n(\theta)\equiv E_n(T_\theta)=-\dfrac{1}{\sin^2 \big((2n-1)\theta\big)},\quad n=1,\ldots,\kappa(\theta),
	\]
	with eigenfunctions $\psi_{n,\theta}$ defined in \eqref{urs00}.
	\end{proposition}

	\begin{proof}
		The equality $\spece T_\theta=[-1,\infty)$ and the finiteness of the discrete spectrum are proved  in~\cite[App.~B]{khp}. It is also
		shown in \cite[Sec.~3.1]{khp} that the discrete spectrum of $T_\theta$ coincides with the discrete spectrum
		of the operator $T^N_\theta$ in $L^2(\cS^+_\theta)$
		defined as the Laplacian with the Robin boundary condition  $\partial_\nu u=u$ on $\Gamma^+_\theta$
		and Neumann boundary condition on the remaining horizontal part of the boundary. Each eigenfunction of $T^N_\theta$
		is the restriction of an eigenfunction of $T_\theta$ on $\cS^+_\theta$ (with the same eigenvalue) and, vice versa,
		each eigenfunction of $T_\theta$ is obtained from an eigenfunction of $T_\theta^N$ (with the same eigenvalue)
		by the $x_2$-parity. It was shown in \cite{kmz} that the operator $T^N_\theta$
		has exactly $\kappa(\theta)$ eigenvalues $\cE_n(\theta)$
		 with $n\in\{1,\dots,\kappa(\theta)\}$ in $(-\infty,-1)$ with associated eigenfunctions $\psi_{n,\theta}$.
		 As the functions $\psi_{n,\theta}$ are even in $x_2$, this concludes the proof.		
	\end{proof}

	\begin{remark}\label{rem26}
		Using the statements of Proposition~\ref{prop-sector}
		one easily sees that the piecewise constant eigenvalue counting function $\kappa$ has jumps only at the \emph{critical half-angles} $\theta_n$,
		\begin{equation}
			\label{theta-critical2}
		\theta_n:=\dfrac{\pi}{2(2n+1)},\quad n=0,1,2,\ldots,
		\end{equation}
and one has
		\[
		\kappa(\theta_n^-)=n+1,\quad \kappa(\theta_n^+)=\kappa(\theta_n)=n.		
		\]		
		
We further note that for each $n\in\NN_0$ the function
\begin{equation}
	\label{eq-psin}
\Psi_n:=\psi_{n+1,\theta_n}.
\end{equation}
is a bounded solution of the Robin problem
\begin{equation}
	\label{urs01}
-\Delta u=-u \text{ in }\cS_{\theta_n},
\quad 
\partial_\nu u=u \text{ on } \partial\cS_{\theta_n}
\end{equation}
at the edge of the essential spectrum of $T_{\theta_n}$. Namely, the conditions~\eqref{urs01} are satisfied due to~\eqref{urseq}
and $x_2$-parity. To see the boundedness we first use \eqref{urs00} to obtain
\begin{multline*}
\varphi_{n+1,m,\theta_n}(x_1,x_2)=\exp\Big(-\big(x_1\cos(2m\theta_n)-x_2\sin(2 m \theta_n)\big)\Big)\\
+\exp\Big(-\big(x_1\cos(2m\theta_n)+x_2\sin(2 m \theta_n)\big)\Big)
\end{multline*}
For $n=0$ one has $\theta_0=\frac{\pi}{2}$ and
\[
\Psi_0(x_1,x_2)=\psi_{1,\frac{\pi}{2}}(x_1,x_2)=\varphi_{1,0,\frac{\pi}{2}}(x_1,x_2)=2e^{-x_1},
\]
so $\Psi_0$ is clearly bounded on the half-space $\cS_\frac{\pi}{2}$. For $n\ge 1$ we note first that for any $m\in\{0,\dots,n\}$
one has
\[
(2m+1)\theta_n=\dfrac{2m+1}{2n+1}\cdot\dfrac{\pi}{2}\in\Big[0,\frac{\pi}{2}\Big],\qquad
\cos \big((2m+1)\theta_n\big)\ge 0.
\]
For $(x_1,x_2)\in\cS_{\theta_n}$	one has $x_1>0$ and $|x_2|< x_1\tan\theta_n$, therefore,
\begin{align*}
	-x_1\cos(2m\theta_n)+|x_2|&\sin(2 m \theta_n)\le -x_1\cos(2m\theta_n)+x_1\tan(\theta_n)\sin(2 m \theta_n)\\
	&=-\dfrac{x_1}{\cos (\theta_n)}\big( \cos (\theta_n) \cos(2m\theta_n)-\sin (\theta_n) \sin(2m\theta_n)\big)\\
	&=-\dfrac{x_1}{\cos (\theta_n)} \cos \big((2m+1)\theta_n\big)\le 0,
\end{align*}
	which shows the boundedness of all $\varphi_{n+1,m,\theta_n}$ with $m\in\{0,\dots,n\}$ and yields the boundedness of $\Psi_n$ on $\cS_{\theta_n}$.
	
	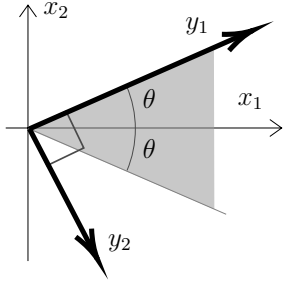
\begin{figure}\centering
		\begin{minipage}{0.35\textwidth}\centering
				\scalebox{0.85}{\begin{tikzpicture}[x=0.75pt,y=0.75pt,yscale=-1,xscale=1]
				
				\draw  (97.5,170.5) -- (260.38,170.5)(110.63,98) -- (110.63,249) (253.38,165.5) -- (260.38,170.5) -- (253.38,175.5) (105.63,105) -- (110.63,98) -- (115.63,105)  ;
				\draw  [draw opacity=0][fill={rgb, 255:red, 155; green, 155; blue, 155 }  ,fill opacity=0.55 ] (220.38,122.25) -- (220.63,218.5) -- (110.4,171) -- cycle ;
				\draw  [draw opacity=0] (168.87,145.94) .. controls (172.17,153.63) and (174,162.1) .. (174,171) .. controls (174,179.69) and (172.26,187.98) .. (169.1,195.53) -- (110.4,171) -- cycle ; \draw  [color={rgb, 255:red, 74; green, 74; blue, 74 }  ,draw opacity=1 ] (168.87,145.94) .. controls (172.17,153.63) and (174,162.1) .. (174,171) .. controls (174,179.69) and (172.26,187.98) .. (169.1,195.53) ;  
				\draw [color={rgb, 255:red, 128; green, 128; blue, 128 }  ,draw opacity=1 ]   (111.5,171) -- (227.6,221.4) ;
				\draw  [color={rgb, 255:red, 74; green, 74; blue, 74 }  ,draw opacity=1 ][line width=0.75]  (111.5,171) -- (133.17,160.61) -- (143.56,182.28) -- (121.89,192.67) -- cycle ;
				\draw [line width=2.25]    (111.5,171) -- (149.75,244.25) ;
				\draw [shift={(151.6,247.8)}, rotate = 242.43] [color={rgb, 255:red, 0; green, 0; blue, 0 }  ][line width=2.25]    (17.49,-5.26) .. controls (11.12,-2.23) and (5.29,-0.48) .. (0,0) .. controls (5.29,0.48) and (11.12,2.23) .. (17.49,5.26)   ;
				\draw [line width=2.25]    (111.5,171) -- (239.54,114.61) ;
				\draw [shift={(243.2,113)}, rotate = 156.23] [color={rgb, 255:red, 0; green, 0; blue, 0 }  ][line width=2.25]    (17.49,-5.26) .. controls (11.12,-2.23) and (5.29,-0.48) .. (0,0) .. controls (5.29,0.48) and (11.12,2.23) .. (17.49,5.26)   ;
				
				\draw (233.15,149) node [anchor=north west][inner sep=0.75pt]    {$x_{1}$};
				\draw (118.5,95.05) node [anchor=north west][inner sep=0.75pt]    {$x_{2}$};
				\draw (177.2,146.2) node [anchor=north west][inner sep=0.75pt]    {$\theta $};
				\draw (177.2,175.8) node [anchor=north west][inner sep=0.75pt]    {$\theta $};
				\draw (202.4,104) node [anchor=north west][inner sep=0.75pt]    {$y_{1}$};
				\draw (156.65,232) node [anchor=north west][inner sep=0.75pt]    {$y_{2}$};

			\end{tikzpicture}}
			
		\end{minipage}\begin{minipage}{0.6\textwidth}
		\caption{The coordinates $(y_1,y_2)$.}\label{fig-yy}
		\end{minipage}
	\end{figure}

	For subsequent use it will be helpful to have an expression of some of the Ursell solutions in the orthogonal coordinates
	\[
	\begin{pmatrix}
		y_1\\y_2
	\end{pmatrix}=\begin{pmatrix}
	\cos\theta&\sin\theta\\
	\sin\theta & -\cos\theta
	\end{pmatrix}
	\begin{pmatrix}
	x_1\\x_2
	\end{pmatrix},\quad \text{ i.e. }
	\begin{pmatrix}
		x_1\\x_2
	\end{pmatrix}=\begin{pmatrix}
		\cos\theta&\sin\theta\\
		\sin\theta & -\cos\theta
	\end{pmatrix}
	\begin{pmatrix}
		y_1\\y_2
	\end{pmatrix},
	\]
	in particular, the value $|y_2|$ is the distance between $(x_1,x_2)$ and the boundary line $x_1=x_2\cot\theta$. See Fig.~\ref{fig-yy} for an illustration. We would like to have a suitable representation for
	\[
	\Tilde\Psi_n(y_1,y_2):=\Psi_n(x_1,x_2),\quad \theta:=\theta_n.
	\]

	For $n=0$ we simply have $(y_1,y_2)=(x_2,x_1)$, hence,
	\[
	\Tilde \Psi_0(y_1,y_2)=2e^{-y_2}.
	\]
	
	From now on assume $\theta=\theta_n$ with $n\ge 1$	 and denote
	\begin{align*}
	\Tilde\varphi_{n+1,m,\theta_n}(y_1,y_2)&:=\varphi_{n+1,m,\theta_n}(x_1,x_2)\\
	&\equiv \varphi_{n+1,m,\theta_n}(y_1\cos\theta_n+y_2\sin\theta_n,y_1\sin\theta_n-y_2\cos\theta_n).
	\end{align*}
	It holds that
	\begin{align*}
	x_1\cos(2m\theta_n)\pm x_2\sin(2 m \theta_n)&=(y_1\cos\theta_n+y_2\sin\theta_n)\cos(2m\theta_n)\\
	&\qquad \pm (y_1\sin\theta_n-y_2\cos\theta_n)\sin(2 m \theta_n)\\
	&=y_1\big(\cos(\theta_n)\cos(2m\theta_n)\pm \sin(\theta_n)\sin(2m\theta_n)\big)\\
	&\quad +y_2\big( \sin(\theta_n)\cos(2m\theta_n)\mp \cos(\theta_n)\sin(2m\theta_n)\big)\\
	&=y_1 \cos\big((2m\mp 1)\theta_n\big)\mp y_2 \sin\big((2m\mp 1)\theta_n\big).
	\end{align*}
	In particular,
	\begin{multline*}
		\Tilde\varphi_{n+1,n,\theta_n}(y_1,y_2)=\exp(-y_2)+\exp\Big(-y_1 \cos\big((2n-1)\theta_n\big)+y_2 \sin\big((2n-1)\theta_n\big)\Big),
	\end{multline*}
	and for $\Tilde\Psi_n(y_1,y_2):=\Psi_n(x_1,x_2)$ we have
	\begin{align}
	\Tilde\Psi_n(y_1,y_2)&=a_ne^{-y_2}+\Tilde\Phi_n(y_1,y_2),\label{urssely}\\
	\Tilde \Phi_n(y_1,y_2)&=\sum_{m=0}^{n-1} \Big\{ b_{n,m} e^{-(y_1 \cos\alpha_{m,n}-y_2\sin\alpha_{m,n})}
	+c_{n,m} e^{-(y_1 \cos\alpha_{m,n}+y_2\sin\alpha_{m,n})}\Big\},\nonumber\\
	\alpha_{m,n}&:=(2m+1)\theta_n,\nonumber
	\end{align}
	with some $a_n,b_{n,m},c_{n,m}\in\RR$ ($a_n\ne 0$) whose precise values are not important for the subsequent analysis. We note that for $(y_1,y_2)$ corresponding to the upper half of the sector, i.e. for
	\[
	0< y_2< y_1 \tan \theta_n,
	\]
	and any $m\in\{0,\dots,n-1\}$ we have $-y_2\ge -y_1\tan\theta_n$,
	\begin{align*}
		y_1 \cos\alpha_{m,n}+y_2\sin\alpha_{m,n}
		&\ge y_1 \cos\alpha_{m,n}-y_2\sin\alpha_{m,n}\\
		&\ge y_1 \cos(\alpha_{m,n})-y_1\tan(\theta_n)\sin(\alpha_{m,n})\\
		&=\dfrac{y_1}{\cos\theta_n}\Big( \cos(\theta_n)\cos(\alpha_{m,n})-\sin(\theta_n)\sin(\alpha_{m,n})\Big)\\
		&=\dfrac{y_1}{\cos\theta_n}\cos(\theta_n+\alpha_{m,n})\equiv \dfrac{y_1}{\cos\theta_n}\cos\big((2m+2)\theta_n\big)\\
		&=d_{m,n} y_1,\quad d_{m,n}:=\dfrac{\cos\big(\frac{m+1}{2n+1}\pi\big)}{\cos\theta_n}>0.
	\end{align*}
	It follows that for suitable $d_n,d'_n>0$ in \eqref{urssely} one has
	\begin{equation}
		\label{urssel-decay}
		\big|\Tilde\Phi_n(y_1,y_2)\big|+\big|\nabla \Tilde\Phi_n(y_1,y_2)\big|\le d'_n e^{-d_n y_1}, \quad 0< y_2< y_1 \tan \theta_n.
	\end{equation}	
Note that the representation \eqref{urssely} and the estimate \eqref{urssel-decay} show $\Psi_n\notin L^2(\cS_{\theta_n})$.
\end{remark}
	
	\subsection{Elliptic estimates}\label{ssec22}
	
	The following assertion is Theorem~9.11 in \cite{gt} for the specific choice $L=-\Delta$, $n=2$ and $p=2$:
	\begin{proposition}[Local Calder\'on-Zygmund estimate]\label{calzyg}
		Let $\Omega\subset\RR^2$ be an open set and $\Omega_0\subset\Omega$ be a bounded open subset with $\overline{\Omega_0}\subset \Omega$. Then there is $C>0$ such that for any
		$u\in H^2_\loc(\Omega)\cap L^2(\Omega)$ with $\Delta u\in L^2(\Omega)$ one has
		\[
		\|u\|_{H^2(\Omega_0)}\le C\big(\|u\|_{L^2(\Omega)}+\|\Delta u\|_{L^2(\Omega)}\big).
		\]

	\end{proposition}
	
	The next statement is Theorem~3.14(ii) in \cite{nittka} in which we take
	\[
	L=-\Delta, \quad N=2, \quad q=1, \quad \eps=1, \quad f_j=0, \quad g=0
	\]
	and use the trivial inequality $\|u\|_{L^\infty(\Omega)}\le \|u\|_{C^{0,\gamma}(\Omega)}$:
	
	\begin{proposition}[$L^2$--$L^\infty$ estimate]	\label{prop-nittka}
		Let $\Omega\subset\RR^2$ be a bounded Lipschitz domain with outer unit normal $\nu$.
		Let $\beta\in L^\infty(\partial\Omega)$, then there exists a constant $C>0$ with the following property: for any $f\in L^2(\Omega)$ and any weak solution $u\in H^1(\Omega)$ of
		\[
		-\Delta u=f \text{ in }\Omega,\quad \partial_\nu u+\beta u=0 \text{ on }\partial\Omega
		\]
		it holds that
		\[
		\|u\|_{L^\infty(\Omega)}\le C\big(\|u\|_{L^2(\Omega)}+\|f\|_{L^2(\Omega)}\big).
		\]
	\end{proposition}

\section{Eigenvalue counting function and threshold resonances}\label{sec3}
	
The aim of this section is to discuss the existence of threshold resonances for the Robin Laplacians in the infinite sectors $\cS_\theta$ (see Subsection~\ref{ssec-sectors}).

For $\theta\in(0,\pi)$ denote by $\cF_\theta$ the vector space of the \emph{bounded} weak solutions of the boundary value problem
\begin{equation*}
	-\Delta F=-F \text{ in }\cS_\theta,\quad \partial_\nu F=F \text{ on }\partial\cS_\theta
\end{equation*}
(understood as in Definition~\ref{defin-threshold}). Our main result is:

\begin{theorem}\label{thm-sec3}
	For any $\theta\in (0,\pi)$ one has
	\[
	\dim \cF_\theta=\begin{cases}
		1, & \text{if $\theta$ is one of the critical half-angles }\theta_n \text{ in }\eqref{theta-critical2},\\
		0, & \text{otherwise.}
	\end{cases}
	\]
	In particular, the Robin Laplacian $T_\theta$ has a threshold resonance if and only if $\theta$ is a critical half-angle at which the eigenvalue counting function $\kappa$ has a jump.
\end{theorem}	
This establishes the equivalence (B)$\iff$(C) in Theorem~\ref{thm-main}.
By combining Theorem~\ref{thm-sec3} with the computations in Remark~\ref{rem26} we also obtain:
\begin{corollary}
	For each critical half-angle $\theta_n$ in \eqref{theta-critical2} one has
	\[
	\cF_{\theta_n}=\RR \Psi_n,
	\]
	where $\Psi_n$ is Ursell's solution given by \eqref{eq-psin}.
\end{corollary}

\subsection{Reduction by parity and cut-off functions}

As a preparation for the proof let us make a couple of reductions and choose suitable coordinates.
Consider the upper half $\cS^+_\theta$ of the sector $\cS_\theta$ and the respective part $\Gamma^+_\theta$
of the Robin boundary,
\begin{equation*}
	\begin{aligned}
		\cS_\theta^+&:=\cS_\theta\cap \big\{(x_1,x_2): x_2>0\big\},\\
		\Gamma_\theta^+&:=\partial\cS_\theta\cap \big\{(x_1,x_2): x_2>0\big\},\\
		\Gamma&:=\partial\cS_\theta^+\setminus\Gamma^+_\theta\equiv\{(x_1,0):\ x_1\ge 0\},
	\end{aligned}
\end{equation*}
see Fig.~\ref{fig-sec2}.

	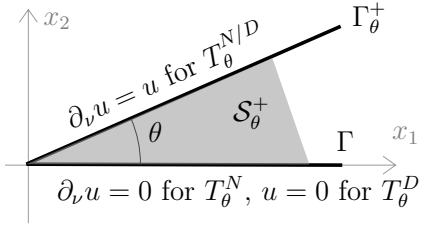
\begin{figure}\centering
	\begin{minipage}{0.4\textwidth}\centering
		\scalebox{0.9}{\begin{tikzpicture}[x=0.75pt,y=0.75pt,yscale=-1,xscale=1]
				
				\draw [color={rgb, 255:red, 155; green, 155; blue, 155 }  ,draw opacity=1 ] (74.5,108.33) -- (292.2,108.33)(87.5,21.73) -- (87.5,140.67) (285.2,103.33) -- (292.2,108.33) -- (285.2,113.33) (82.5,28.73) -- (87.5,21.73) -- (92.5,28.73)  ;
				\draw [color={rgb, 255:red, 0; green, 0; blue, 0 }  ,draw opacity=1 ][line width=1.5]    (86.4,107.8) -- (261.6,31) ;
				\draw [color={rgb, 255:red, 0; green, 0; blue, 0 }  ,draw opacity=1 ][line width=1.5]    (86.4,107.8) -- (262.5,108.33) ;
				\draw  [draw opacity=0][fill={rgb, 255:red, 155; green, 155; blue, 155 }  ,fill opacity=0.55 ][line width=2.25]  (244.17,108) -- (86.4,107.8) -- (223.5,48.67) -- cycle ;
				\draw  [draw opacity=0] (144.87,82.74) .. controls (148.17,90.43) and (150,98.9) .. (150,107.8) .. controls (150,107.87) and (150,107.93) .. (150,108) -- (86.4,107.8) -- cycle ; \draw  [color={rgb, 255:red, 74; green, 74; blue, 74 }  ,draw opacity=1 ] (144.87,82.74) .. controls (148.17,90.43) and (150,98.9) .. (150,107.8) .. controls (150,107.87) and (150,107.93) .. (150,108) ;  
				
				\draw (289.07,85.8) node [anchor=north west][inner sep=0.75pt]  [color={rgb, 255:red, 128; green, 128; blue, 128 }  ,opacity=1 ]  {$x_{1}$};
				\draw (94.67,21.93) node [anchor=north west][inner sep=0.75pt]  [color={rgb, 255:red, 128; green, 128; blue, 128 }  ,opacity=1 ]  {$x_{2}$};
				\draw (153.2,83) node [anchor=north west][inner sep=0.75pt]    {$\theta $};
				\draw (103.52,72.8) node [anchor=north west][inner sep=0.75pt]  [rotate=-336.39]  {$\partial _{\nu } u=u$ for $T_{\theta }^{N/D}$};
				\draw (102.27,113.27) node [anchor=north west][inner sep=0.75pt]    {$\partial _{\nu } u=0$ for $T_{\theta }^{N}$, $u=0$ for $T_{\theta }^{D}$};
				\draw (200.47,69.53) node [anchor=north west][inner sep=0.75pt]    {$\cS_{\theta }^{+}$};
				\draw (266.67,16.4) node [anchor=north west][inner sep=0.75pt]    {$\Gamma _{\theta }^{+}$};
				\draw (258,88.07) node [anchor=north west][inner sep=0.75pt]    {$\Gamma $};

			\end{tikzpicture}
		}
		
	\end{minipage}\begin{minipage}{0.55\textwidth}
		\caption{\small The upper half $\cS_\theta^+$ of the sector $\cS_\theta$.}\label{fig-sec2}
	\end{minipage}
\end{figure}

By applying the usual decomposition by parity \cite[Sec.~3.1]{khp} one shows the unitary equivalence $T_\theta\simeq T^N_\theta\oplus T^D_\theta$, where $T^{N/D}_\theta$ are the Laplacians $u\mapsto-\Delta u$ in $L^2(\cS^+_\theta)$
with the boundary conditions
\begin{align*}
	\partial_\nu u&=u \text{ on }\Gamma^+_\theta \text{ for both }T^{N/D}_\theta,\\
	\partial_\nu u&=0 \text{ on }\Gamma \text{ for }T^{N}_\theta,\\
	u&=0 \text{ on }\Gamma \text{ for }T^{D}_\theta,
\end{align*}		
and a simple extension argument shows that $T^D_\theta\ge -1$, hence, the operator $T^N_\theta$
has the same $\kappa(\theta)$ eigenvalues in $(-\infty,-1)$ as $T_\theta$.

Let us denote by $\Tilde \cF_\theta^{N}$ the space of the bounded weak solutions of
\begin{equation}
	\label{tnt1}
	-\Delta F=-F \text{ in }\cS^+_\theta,\quad \partial_\nu F=F \text{ on }\Gamma^+_\theta,
	\quad \partial_\nu F=0 \text{ on }\Gamma,
\end{equation}
and by $\Tilde\cF_\theta^{D}$ the vector space of the bounded weak solutions of
\begin{equation}
	\label{tdt1}
	-\Delta F=-F \text{ in }\cS^+_\theta,\quad \partial_\nu F=F \text{ on }\Gamma^+_\theta,
	\quad F=0 \text{ on }\Gamma,
\end{equation}
then the standard parity argument shows
\begin{equation}
	\label{fsum1}
\dim \cF_\theta=\dim \Tilde \cF_\theta^N+\dim \Tilde\cF_\theta^D.
\end{equation}

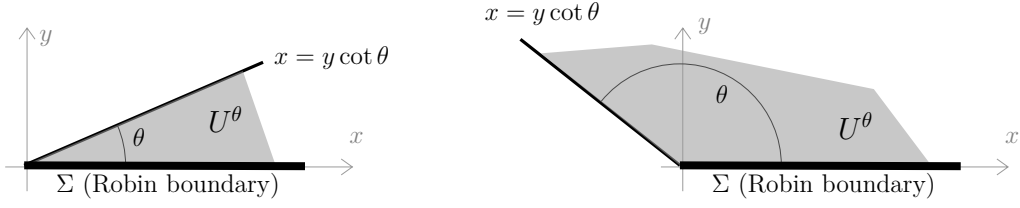
\begin{figure}
	\centering
	
	\scalebox{0.8}{\begin{tikzpicture}[x=0.75pt,y=0.75pt,yscale=-1,xscale=1]
			
			\draw [color={rgb, 255:red, 155; green, 155; blue, 155 }  ,draw opacity=1 ] (19.5,139.8) -- (237.2,139.8)(33.2,52.73) -- (33.2,155) (230.2,134.8) -- (237.2,139.8) -- (230.2,144.8) (28.2,59.73) -- (33.2,52.73) -- (38.2,59.73)  ;
			\draw [color={rgb, 255:red, 0; green, 0; blue, 0 }  ,draw opacity=1 ][line width=1.5]    (31.4,138.8) -- (181.2,74.2) ;
			\draw  [draw opacity=0][fill={rgb, 255:red, 155; green, 155; blue, 155 }  ,fill opacity=0.55 ][line width=2.25]  (189.17,139) -- (31.4,138.8) -- (168.5,79.67) -- cycle ;
			\draw [color={rgb, 255:red, 155; green, 155; blue, 155 }  ,draw opacity=1 ] (430.7,139.8) -- (648.4,139.8)(444.4,52.73) -- (444.4,155) (641.4,134.8) -- (648.4,139.8) -- (641.4,144.8) (439.4,59.73) -- (444.4,52.73) -- (449.4,59.73)  ;
			\draw  [draw opacity=0] (89.87,113.74) .. controls (93.17,121.43) and (95,129.9) .. (95,138.8) .. controls (95,138.87) and (95,138.93) .. (95,139) -- (31.4,138.8) -- cycle ; \draw  [color={rgb, 255:red, 74; green, 74; blue, 74 }  ,draw opacity=1 ] (89.87,113.74) .. controls (93.17,121.43) and (95,129.9) .. (95,138.8) .. controls (95,138.87) and (95,138.93) .. (95,139) ;  
			\draw [color={rgb, 255:red, 0; green, 0; blue, 0 }  ,draw opacity=1 ][line width=1.5]    (442.6,138.8) -- (342.6,59.8) ;
			\draw  [draw opacity=0][fill={rgb, 255:red, 155; green, 155; blue, 155 }  ,fill opacity=0.55 ][line width=2.25]  (600.37,139) -- (442.6,138.8) -- (355.4,68.6) -- (425,63) -- (564.2,91) -- cycle ;
			\draw  [draw opacity=0] (393.21,98.73) .. controls (404.87,84.37) and (422.66,75.2) .. (442.6,75.2) .. controls (477.73,75.2) and (506.2,103.67) .. (506.2,138.8) .. controls (506.2,138.87) and (506.2,138.93) .. (506.2,139) -- (442.6,138.8) -- cycle ; \draw  [color={rgb, 255:red, 74; green, 74; blue, 74 }  ,draw opacity=1 ] (393.21,98.73) .. controls (404.87,84.37) and (422.66,75.2) .. (442.6,75.2) .. controls (477.73,75.2) and (506.2,103.67) .. (506.2,138.8) .. controls (506.2,138.87) and (506.2,138.93) .. (506.2,139) ;  
			\draw [color={rgb, 255:red, 0; green, 0; blue, 0 }  ,draw opacity=1 ][line width=3.75]    (442.6,138.8) -- (618.7,139.33) ;
			\draw [color={rgb, 255:red, 0; green, 0; blue, 0 }  ,draw opacity=1 ][line width=3.75]    (31.4,138.8) -- (207.5,139.33) ;
			
			\draw (234.07,116.8) node [anchor=north west][inner sep=0.75pt]  [color={rgb, 255:red, 128; green, 128; blue, 128 }  ,opacity=1 ]  {$x$};
			\draw (39.67,52.93) node [anchor=north west][inner sep=0.75pt]  [color={rgb, 255:red, 128; green, 128; blue, 128 }  ,opacity=1 ]  {$y$};
			\draw (98.2,114) node [anchor=north west][inner sep=0.75pt]    {$\theta $};
			\draw (145.47,100.53) node [anchor=north west][inner sep=0.75pt]    {\large $U^{\theta }$};
			\draw (50.2,142.47) node [anchor=north west][inner sep=0.75pt]    {$\Sigma$ (Robin boundary)};
			\draw (187,63.2) node [anchor=north west][inner sep=0.75pt]    {$x=y\cot \theta $};
			\draw (645.27,116.8) node [anchor=north west][inner sep=0.75pt]  [color={rgb, 255:red, 128; green, 128; blue, 128 }  ,opacity=1 ]  {$x$};
			\draw (452.87,44.93) node [anchor=north west][inner sep=0.75pt]  [color={rgb, 255:red, 128; green, 128; blue, 128 }  ,opacity=1 ]  {$y$};
			\draw (463.8,87.6) node [anchor=north west][inner sep=0.75pt]    {$\theta $};
			\draw (540.67,104.53) node [anchor=north west][inner sep=0.75pt]    {\large $U^{\theta }$};
			\draw (461.4,142.47) node [anchor=north west][inner sep=0.75pt]    {$\Sigma$ (Robin boundary)};
			\draw (319,35.2) node [anchor=north west][inner sep=0.75pt]    {$x=y\cot \theta $};

		\end{tikzpicture}
	}
	
	\caption{The sector $U^\theta$ for $\theta<\frac{\pi}{2}$ (on the left) and $\theta\ge \frac{\pi}{2}$ (on the right).}\label{fig-uth}
\end{figure}

For the subsequent analysis it will be useful to pass to the coordinates $(y_1,y_2)$ as~in Remark~\ref{rem26} (see Fig.~\ref{fig-yy}) in order to put the Robin condition on a $\theta$-independent part of the boundary.
Namely, for any $\theta\in(0,\pi)$ we will work with the sets
\begin{align*}
	U^\theta&:=\big\{(x,y):\ 0<\arg (x+\rmi y)<\theta\big\}\\
	&\phantom{:}\equiv\big\{(x,y):\ y>0,\ x> y\cot \theta\big\},\\
	\Sigma&:=\big\{(x,y):\ x\ge 0,\ y=0\big\}\subset\partial U^\theta,
\end{align*}	
see Fig.~\ref{fig-uth},
and denote by $N^\theta$ the Laplacian  $u\mapsto-\Delta u$ on $U^\theta$ with the Robin boundary condition $\partial_\nu u=u$ on $\Sigma$
and Neumann boundary condition on the oblique part $\{x=y\cot\theta\}$.
The bilinear form $n^\theta$ for $N^\theta$ is
\[
n^\theta(u,v)=\iint_{U^\theta}\nabla u\cdot\nabla v\,\dxy-\int_\Sigma uv\,\dd s,\quad u,v\in H^1(U^\theta),
\]
and $N^\theta$ is unitarily equivalent to the above $T^N_\theta$, in particular, it has exactly $\kappa(\theta)$
eigenvalues in $(-\infty,-1)$.

The Robin-Neumann threshold resonance problem \eqref{tnt1} takes the form
\begin{equation}
	\label{FFF-aa}
	\left\{
	\begin{aligned}
		-\Delta F&=-F \text{ in } U^\theta,\\
		\partial_\nu F&=F \text{ on } \Sigma,\\
		\partial_\nu F&=0 \text{ on } \{y>0,\ x=y\cot\theta\},
	\end{aligned}\right.
\end{equation}
which is understood as $F\in H^1(U')$ for any bounded subdomain $U'\subset U^\theta$, with 
\begin{equation}
	\label{FFF-aa2}
\iint_{U^\theta}\nabla F\cdot\nabla \varphi\,\dxy-\int_\Sigma F\varphi\,\dd s=-\iint_{U^\theta}F\varphi\,\dxy
\end{equation}
for any $\varphi\in H^1(U^\theta)$ with bounded support, and we denote by $\cF^N_\theta$ the space of the bounded solutions
to \eqref{FFF-aa}. Similarly, the Robin-Dirichlet threshold resonance problem \eqref{tdt1} takes the form
\begin{equation}
	\label{FFF-dir1}
\left\{	\begin{aligned}
		-\Delta F&=-F \text{ in } U^\theta,\\
		\partial_\nu F&=F \text{ on } \Sigma,\\
		F&=0 \text{ on } \{y>0,\ x=y\cot\theta\},
	\end{aligned}\right.
\end{equation}
which is understood as $F\in H^1(U')$ for any bounded subdomain $U'\subset U^\theta$ with $F=0$ on $\{x=y\cot \theta\}$,
and
\begin{equation}
		\label{FFF-dir2}
\iint_{U^\theta}\nabla F\cdot\nabla \varphi\,\dxy-\int_\Sigma F\varphi\,\dd s=-\iint_{U^\theta}F\varphi\,\dxy
\end{equation}
for any $\varphi\in H^1(U^\theta)$ with bounded support and $\varphi=0$ on $\{x=y\cot \theta\}$. We denote by $\cF^D_\theta$ the space of the bounded solutions
to \eqref{FFF-dir1}.
By construction and due to \eqref{fsum1} we have
\begin{equation}
	\label{fsum2}
	\begin{aligned}
	\dim\Tilde \cF^{N/D}_\theta&=\dim\cF^{N/D}_\theta,\\
	\dim \cF_\theta&=\dim \cF_\theta^N+\dim \cF_\theta^D.
	\end{aligned}
\end{equation}

Recall that the function $\theta\mapsto\kappa(\theta)$, which counts the eigenvalues of the operator $N^\theta$ in $(-\infty,-1)$,
 has jumps at the critical half-angles
\begin{equation}
	\label{kjump}
\theta_n:=\dfrac{\pi}{2(2n+1)},\quad n=0,1,2\ldots,\quad
	\kappa(\theta_n^-)-\kappa(\theta_n^+)\equiv \kappa(\theta_n^-)-\kappa(\theta_n)=1,
\end{equation}	
and is continuous at the other values of $\theta$.

As a preparation for the proofs below, for $R>1$ introduce the cutoff functions
\[
\chi_R:\ (x,y)\mapsto \begin{cases}
	1, & \sqrt{x^2+y^2}\le R,\\[\medskipamount]
	\dfrac{2\log R-\log\sqrt{x^2+y^2}}{\log R}, & R<\sqrt{x^2+y^2}<R^2,\\
	0, & \sqrt{x^2+y^2}\ge R^2.
\end{cases}
\]
Note that we have
\[
0\le \chi_R\le 1,\quad \supp \nabla\chi_R\subset \Big\{ (x,y):\ R\le \sqrt{x^2+y^2}\le R^2\Big\},
\]
and for $R< \sqrt{x^2+y^2}< R^2$ it holds that
\[
\nabla \chi_R(x,y)=-\dfrac{1}{\log R} \cdot \dfrac{(x,y)}{x^2+y^2},\quad
\big|\nabla \chi_R(x,y)\big|^2=\dfrac{1}{(\log R)^2(x^2+y^2)}.
\]
Hence, by passing to polar coordinates one computes
\begin{equation}
	\label{chinabla}
	\begin{aligned}
		\iint_{U^\theta}&\big|\nabla \chi_R(x,y)\big|^2\dxy=
		\dfrac{1}{(\log R)^2}\iint_{\substack{ R\le \sqrt{x^2+y^2}\le R^2\\ 0<\arg(x+\rmi y)<\theta}}\dfrac{1}{x^2+y^2}\dxy\\
		&=\dfrac{1}{(\log R)^2}\int_0^\theta \int_R^{R^2} \dfrac{1}{r^2} r\,\dd r\,\dd\varphi
		=\dfrac{\theta}{(\log R)^2}\int_R^{R^2} \dfrac{1}{r} \,\dd r\\
		&=\dfrac{\theta}{(\log R)^2} \big(\log (R^2)-\log R\big)=\dfrac{\theta}{\log R}\xrightarrow{R\to\infty}0.
	\end{aligned}
\end{equation}

Further note that for any $F\in \cF^N_\theta$ and any $R>1$ the function $\chi_R^2 F$ belongs to $H^1(U^\theta)$
and has bounded support, so one can use \eqref{FFF-aa2} with $\varphi:=\chi_R^2 F$. With the help of
\begin{equation*}
	\begin{aligned}
		\nabla F&\cdot \nabla(\chi_R^2 F)=\nabla F\cdot(2F\chi_R\nabla\chi_R+\chi_R^2\nabla F)\\
		&=2\chi_R F \nabla F\cdot\nabla\chi_R+\chi_R^2|\nabla F|^2
		=|\chi_R \nabla F+ F\nabla\chi_R|^2-|\nabla \chi_R|^2 F^2\\
		&=\big|\nabla(\chi_R F)\big|^2-|\nabla \chi_R|^2 F^2
	\end{aligned}
\end{equation*}
we arrive at
\begin{equation}
	\label{uuff}
	\iint_{U^\theta}\Big(\big|\nabla(\chi_R F)\big|^2+(\chi_R F)^2\Big)\,\dxy-\int_\Sigma (\chi_R F)^2\dd s
	=\iint_{U^\theta} |\nabla \chi_R|^2 F^2\dxy.
\end{equation}
Similarly note that for any $F\in \cF^D_\theta$ and any $R>1$ the function $\chi_R^2 F$ belongs to $H^1(U^\theta)$, has bounded support and vanishes on $\{x=y\cot\theta\}$, hence, it can be used as $\varphi$ in \eqref{FFF-dir2}, and the same computation shows that \eqref{uuff} holds in this case too.

\subsection{Exclusion of threshold resonances in easy cases}

\begin{lemma}\label{lemdd}
For any $\theta\in(0,\pi)$ one has $\cF^D_\theta=\{0\}$.
\end{lemma}

\begin{proof}
	Let $F\in\cF^D_\theta$, then the identity \eqref{uuff} holds. 
	Denote by $\Tilde F$  the extension of $F$ by zero to the half-plane $\RR\times(0,\infty)$.
	Due to $F=0$ on $\{x=y\cot\theta\}$, one has $\Tilde F\in H^1(U')$ for any bounded subdomain $U'\subset\RR\times(0,\infty)$, with
	$\nabla \Tilde F=0$ outside $U^\theta$, so \eqref{uuff} yields
\begin{multline*}
\int_\RR\bigg[\int_0^\infty \Big(\big|\partial_x(\chi_R \Tilde F)(x,y)\big|^2+\big|\partial_y(\chi_R \Tilde F)(x,y)^2+(\chi_R \Tilde F)(x,y)\big|^2\Big)\,\dd y\\
-(\chi_R \Tilde F)(x,0)^2\bigg]\,\dd x=\iint_{U^\theta} |\nabla \chi_R|^2 F^2\,\dxy.
\end{multline*}
For any $x$, the function $y\mapsto (\chi_R \Tilde F)(x,y)$ vanishes identically at least for all $y>R^2$, hence, 
\[
-(\chi_R \Tilde F)^2(x,0)=\int_0^\infty \partial_y\Big( (\chi_R \Tilde F)^2\Big)(x,y)\,\dd y=2\int_0^\infty (\chi_R \tilde F)(x,y) \partial_y(\chi_R \Tilde F)(x,y)\,\dd y,
\]
which allows one to rewrite the preceding identity as
\begin{equation}
	\label{uuff-aa}
\iint_{\RR\times(0,\infty)}\Big(\big|\partial_x(\chi_R \Tilde F)\big|^2+\big|\partial_y(\chi_R \Tilde F) +(\chi_R \Tilde F)\big|^2\Big)\,
	\dxy=\iint_{U^\theta} |\nabla \chi_R|^2 F^2\,\dxy.
\end{equation}

For any fixed $\rho>0$ and all $R>\max\{1,\rho\}$ we have $\chi_R=1$ on $B_\rho(0)$. As all integrands in \eqref{uuff-aa} are non-negative, by sending $R\to \infty$ we obtain
\begin{align*}
	\iint_{B_\rho(0)\cap (\RR\times(0,\infty))}&\Big((\partial_x \Tilde F)^2+(\partial_y\Tilde F +\Tilde F)^2\Big)\dxy\\
	&\le \iint_{\RR\times(0,\infty)}\Big(\big|\partial_x(\chi_R \Tilde F)\big|^2+\big|\partial_y(\chi_R \Tilde F) +(\chi_R \Tilde F)\big|^2\Big)
	\dxy\\
	&\stackrel{\eqref{uuff-aa}}{=} \iint_{U^\theta} |\nabla \chi_R|^2 F^2\dxy
	\le \|F\|_\infty^2\iint_{U^\theta} |\nabla \chi_R|^2\,\dxy\xrightarrow[\text{by }\eqref{chinabla}]{R\to\infty}0.
\end{align*}
Due to the arbitrariness of $\rho$ we arrive at
\[
\partial_x\Tilde F=\partial_y\Tilde F+\Tilde F=0 \text{ in }\RR\times(0,\infty),
\]
hence, $\Tilde F(x,y)=ce^{-y}$ with some $c\in\RR$. As $\Tilde F$ vanishes outside $U^\theta$, one obtains $c=0$
and then $\Tilde F=0$, i.e. $F=0$.
\end{proof}

By using the result of Lemma~\ref{lemdd} in~\eqref{fsum2} we arrive at
\begin{equation}
	\label{fsum3}
	\dim \cF_\theta=\dim \cF_\theta^N,
\end{equation}
and Theorem~\ref{thm-sec3} becomes a direct consequence of the following statement:

\begin{lemma}\label{thm31}
	For any $\theta\in (0,\pi)$ one has
	\[
	\dim \cF_\theta^N=\begin{cases}
		1, & \text{ if $\theta$ is one of the $\theta_n$ in \eqref{kjump},}\\
		0, & \text{otherwise.}
	\end{cases}
	\]
\end{lemma}	

For the proof, we will distinguish between $\theta\ge \frac{\pi}{2}$ and $\theta<\frac{\pi}{2}$. One deals with the first
case by proceeding very similarly to Lemma~\ref{lemdd}.

\begin{lemma}\label{lem32a}
The assertion of Lemma~\ref{thm31} holds for all $\theta\ge\frac{\pi}{2}$.
\end{lemma}

\begin{proof}
Let $\theta \in[\frac{\pi}{2},\pi\big)$ and $F\in \cF^N_\theta$, then the identity \eqref{uuff} takes the form
\begin{multline}
	\iint_{U^\theta\setminus U^\frac{\pi}{2}}\Big(\big|\nabla(\chi_R F)\big|^2+(\chi_R F)^2\Big)\,\dxy\\
	+\iint_{U^\frac{\pi}{2}}\Big(\big|\nabla(\chi_R F)\big|^2+(\chi_R F)^2\Big)\,\dxy-\int_0^\infty (\chi_R F)^2(x,0)\,\dd x\\
	=\iint_{U^\theta} \big|\nabla \chi_R(x,y)\big|^2 F(x,y)^2\,\dd y\,\dd x,
	\label{uuff2}
\end{multline}
while
\begin{multline*}
	\iint_{U^\frac{\pi}{2}}\Big(\big|\nabla(\chi_R F)\big|^2+(\chi_R F)^2\Big)\,\dxy-\int_0^\infty (\chi_R F)^2(x,0)\,\dd x\\
	=\int_0^\infty\bigg[\int_0^\infty \Big(\big|\partial_x(\chi_R F)(x,y)\big|^2
	+\big|\partial_y(\chi_R F)(x,y)\big|^2\\
	+(\chi_R F)(x,y)^2\Big)\,\dd y-(\chi_R F)^2(x,0)\bigg]\,\dd x.
\end{multline*}	
For $x>0$, the function $y\mapsto (\chi_R F)(x,y)$ vanishes identically for all $y>R^2$, hence, 
\[
-(\chi_R F)^2(x,0)=\int_0^\infty \partial_y\Big( (\chi_R F)^2\Big)(x,y)\,\dd y=2\int_0^\infty (\chi_R F)(x,y) \partial_y(\chi_R F)(x,y)\,\dd y,
\]
which allows one to rewrite the preceding identity as
\begin{multline*}
	\iint_{U^\frac{\pi}{2}}\Big(\big|\nabla(\chi_R F)\big|^2+(\chi_R F)^2\Big)\,\dxy-\int_0^\infty (\chi_R F)^2(x,0)\dd x\\
	=\iint_{U^\frac{\pi}{2}}\Big(\big|\partial_x(\chi_R F)\big|^2+\big|\partial_y(\chi_R F) +(\chi_R F)\big|^2\Big)
	\dxy,
\end{multline*}
and the substitution into \eqref{uuff2} yields
\begin{multline}
	\iint_{U^\theta\setminus U^\frac{\pi}{2}}\Big(\big|\nabla(\chi_R F)\big|^2+(\chi_R F)^2\Big)\,\dxy\\
	+\iint_{U^\frac{\pi}{2}}\Big(\big|\partial_x(\chi_R F)\big|^2+\big|\partial_y(\chi_R F) +(\chi_R F)\big|^2\Big)
	\dxy\\
	=\iint_{U^\theta} \big|\nabla \chi_R(x,y)\big|^2 F(x,y)^2\,\dxy.
	\label{uuff3}
\end{multline}

For any fixed $\rho>0$ and all $R>\max\{1,\rho\}$ we have $\chi_R=1$ on $B_\rho(0)$. As all integrands in \eqref{uuff3} are non-negative,
by sending $R\to \infty$ we obtain
\begin{align*}
	\iint_{B_\rho(0)\cap U^\theta\setminus U^\frac{\pi}{2}}&\Big(|\nabla F|^2+F^2\Big)\,\dxy
	+\iint_{B_\rho(0)\cap U^\frac{\pi}{2}}\Big((\partial_x F)^2+(\partial_yF +F)^2\Big)
	\dxy\\
	&\le
	\iint_{U^\theta\setminus U^\frac{\pi}{2}}\Big(\big|\nabla(\chi_R F)\big|^2+(\chi_R F)^2\Big)\,\dxy\\
	&\quad +\iint_{U^\frac{\pi}{2}}\Big(\big|\partial_x(\chi_R F)\big|^2+\big|\partial_y(\chi_R F) +(\chi_R F)\big|^2\Big)
	\dxy\\
	&\stackrel{\eqref{uuff3}}{=}\iint_{U^\theta} |\nabla \chi_R|^2 F^2\,\dxy\le \|F\|_\infty^2\iint_{U^\theta} |\nabla \chi_R|^2\,\dxy
	\xrightarrow[\text{by \eqref{chinabla}}]{R\to\infty}0,
\end{align*}
and due to the arbitrariness of $\rho$ we arrive at
\begin{align}
	F&=0 \text{ in }U^\theta\setminus U^\frac{\pi}{2}, \label{f0-a}\\
	\partial_xF=\partial_y F+F&=0 \text{ in }U^\frac{\pi}{2}. \label{f0-b}
\end{align}

If $\theta>\frac{\pi}{2}$, then $\theta$ is non-critical. The condition \eqref{f0-a} shows that $F$ vanishes on a non-empty open subset of $U^\theta$, so $F=0$ in $U^\theta$ due to the unique continuation principle, thus $\cF^N_\theta=\{0\}$.

The value $\theta=\frac{\pi}{2}$ is critical. The condition \eqref{f0-b} yields $F(x,y)=ce^{-y}$ with $c\in\RR$. On the other hand,
all such functions $F$ clearly belong to $\cF^N_\frac{\pi}{2}$, which gives $\dim\cF^N_\frac{\pi}{2}= 1$.	
\end{proof}

\subsection{Relating threshold resonances with the eigenvalue count}

To cover the case $\theta<\frac{\pi}{2}$ in Lemma~\ref{thm31} we need an intermediate step:

\begin{lemma}\label{lem32}
For any $\theta\in(0,\frac{\pi}{2})$ one has
	\begin{equation}
		\label{ftheta}
	 \dim\cF^N_\theta\le \kappa(\theta^-)- \kappa(\theta).
	\end{equation}
\end{lemma}

\begin{proof}
Let $0<\theta'<\theta< \frac{\pi}{2}$ and $m\in\NN$. Denote
\[
\kappa:=\kappa(\theta),\quad \cF:= \text{any $m$-dimensional subspace of }\cF^N_\theta.
\]	
We are going to show that $\kappa(\theta')\ge \kappa+m$, i.e. that the operator $N^{\theta'}$ has at least
$\kappa+m$ eigenvalues in $(-\infty,-1)$. This will be done
by constructing a $(\kappa+m)$-dimensional test subspace for the min-max principle
by using the suitable truncations  and dilation of the functions in $\cF$, which is similar to
the approach used for broken waveguides, see e.g. \cite[Prop.~11]{p17} and uses the monotonicity with respect to the angle
as in \cite[Sec.~3.3]{khp}.

The main idea of the construction is that
$U^{\theta'}$ is obtained from $U^\theta$ by applying the dilation in $x$-direction
\[
(x,y)\mapsto (\lambda x,y),\quad \lambda:=\dfrac{\tan \theta}{\tan\theta'}>1.
\]
Consider the respective unitary transform
\begin{equation}
	\label{eq-dilat}
\Theta: L^2(U^{\theta})\to L^2(U^{\theta'}),\quad (\Theta u) (x,y):=\dfrac{1}{\sqrt{\lambda}}u\big(\dfrac{x}{\lambda},y\big),
\end{equation}
then we have obviously $H^1(U^{\theta'})=\Theta H^1(U^\theta)$, and for any $u\in H^1(U^{\theta})$ it holds that
\begin{align*}
n^{\theta'}(\Theta u,\Theta u)&=\dfrac{1}{\lambda}\iint_{\substack{0<y<x\tan \theta'}} \bigg[\Big(\partial_x\big[ u\big(\dfrac{x}{\lambda},y\big)\big]\Big)^2+\Big(\partial_y \big[u\big(\dfrac{x}{\lambda},y\big)\big]\Big)^2\bigg]\dxy\\
&\qquad -\dfrac{1}{\lambda}\int_0^\infty u\big(\dfrac{x}{\lambda},0\big)^2\dd x\\
&=\dfrac{1}{\lambda}\iint_{0<y<\dfrac{x}{\lambda}\tan \theta} \bigg[\dfrac{1}{\lambda^2} (\partial_x u)\big(\dfrac{x}{\lambda},y\big)^2+(\partial_y u)\big(\dfrac{x}{\lambda},y\big)^2\bigg]\dxy\\
&\qquad -\dfrac{1}{\lambda}\int_0^\infty u\big(\dfrac{x}{\lambda},0\big)^2\dd x\\
(x:=\lambda z):\quad &=\iint_{0<y<z\tan \theta} \bigg[\dfrac{1}{\lambda^2} (\partial_x u)(z,y)^2+(\partial_y u)(z,y)^2\bigg]\dd z\,\dd y-\int_0^\infty u (z,0)^2\dd z,
\end{align*}
therefore,
\begin{equation}
	\label{nti1}
	n^{\theta'}(\Theta u,\Theta u)=n^\theta(u,u)-\delta \iint_{U^\theta}(\partial_x u)^2\dxy,\quad \delta:=1-\dfrac{1}{\lambda^2}>0.
\end{equation}

Let $\mu\in \big(0,-1-\cE_\kappa(\theta)\big)$, then $E_\kappa(N^\theta+I)\equiv \cE_\kappa(\theta)+1<-\mu$, and by the min-max principle and the density of functions with bounded supports in $H^1(U^\theta)$, see e.g. \cite[Thm.~1.4.2.1]{grisvard}, we find a $\kappa$-dimensional subspace
$V\subset H^1(U^\theta)$ and $\rho>0$
such that
\[
	(n^\theta+I)(v,v)\le -\mu \|v\|^2_{L^2(U^\theta)},\quad \supp v\subset U^\theta\cap B_\rho(0),\quad v\in V.
\]

Note that for any $R>0$ the map
\[
\cF\ni F\mapsto \chi_R F\in C^0(U^\theta)
\]
is injective: If $\chi_R F=0$, then $F=0$ in $B_R(0)$ and the unique continuation yields $F\equiv 0$.
It follows that for any $R>0$ we have
\[
\dim (\chi_R \cF)=\dim \cF=m,\quad
\chi_R \cF:=\big\{\chi_R F: \ F\in \cF\}.
\]

Choose $R_0>\max\{1,\rho\}$, and for $R>R_0$ consider the
subspace
\[
W_R:=V+\chi_R\cF\equiv \big\{v+\chi_R F:\ v\in V,\ F\in \cF\big\}.
\]
Let $v\in V$ and $F\in\cF$. By construction $\chi_R F=F$ on $\supp v$ for any $v\in V$, and the weak formulation~\eqref{FFF-aa2} with $\varphi:=v$ reads as $(n^\theta+I)(v,\chi_R F)=0$, 
which shows, in particular, that $V$ and $\chi_R \cF$ are linearly independent. Namely, if $\chi_R F\in V$ for some $F\in\cF$, then
taking $v:=\chi_R F$ gives $0=(n^\theta+I)(\chi_R F,\chi_R F)\le -\mu\|\chi_R F\|^2_{L^2(U^\theta)}$ and implies first $\chi_R F=0$ and then $F=0$. Therefore,
\begin{gather*}
V\times\cF\ni (v,F)\mapsto v+\chi_R F\in W_R \text{ is injective for } R>R_0,\\
\dim W_R=\dim V+\dim \cF=\kappa+m.
\end{gather*}
The identity \eqref{uuff} gives
\begin{align*}
	(n^\theta+I)(\chi_R F,\chi_R F)&=\iint_{U^\theta}|\nabla\chi_R|^2 F^2\,\dxy\le \eps_R \|F\|^2_{L^\infty(U^\theta)},\\
	\eps_R&:=\iint_{U^\theta}|\nabla\chi_R|^2\,\dxy\xrightarrow[\text{by \eqref{chinabla}}]{R\to\infty}0,
\end{align*}
To summarize,
\begin{equation}
	\label{nti2}
\begin{aligned}
(n^\theta+I)(v+\chi_R F,v+\chi_R F)&=(n^\theta+I)(v,v)+(n^\theta+I)(\chi_R F,\chi_R F)\\
&\le -\mu\|v\|_{L^2(U^\theta)}^2+\eps_R \|F\|^2_{L^\infty(U^\theta)}.
\end{aligned}
\end{equation}

Now use the dilation $\Theta$ from \eqref{eq-dilat} and consider the $(\kappa+m)$-dimensional subspace
\[
W'_R:=\Theta W_R\subset H^1(U^{\theta'}).
\]
Represent any function $w'\in W'_R$ as
\[
w'=\Theta w \in W'_R,\quad  w=v+\chi_R F\in W_R,\quad v\in V,\quad F\in\cF,
\]
then one obtains from \eqref{nti1} and \eqref{nti2}:
\begin{equation}
	\label{nti3}
\begin{aligned}
	(n^{\theta'}+&I)(w',w')\stackrel{\eqref{nti1}}{=}(n^{\theta}+I)(w,w)-\delta\iint_{U^\theta}|\partial_x w|^2\dxy\\
	&=
	(n^\theta+I)(v+\chi_R F,v+\chi_RF)-\delta\iint_{U^{\theta}}\big|\partial_x(v+\chi_RF)\big|^2\dxy\\
	&\le -\mu\|v\|_{L^2(U^\theta)}^2+\eps_R \|F\|^2_{L^\infty(U^\theta)}-\delta\iint_{U^{\theta}}\big|\partial_x(v+\chi_R F)\big|^2\dxy.
\end{aligned}
\end{equation}

We claim:
	\begin{equation}
		\label{jmap}
	\text{The map }\cF\ni F\mapsto \partial_x F \in C^0(U^\theta) \text{ is injective}.
	\end{equation}
Namely, let $F\in\cF$ with $\partial_x F=0$, then $F(x,y)=f(y)$ with some function $f$. From $-\Delta F=-F$	one obtains $-f''=-f$, and $F\in L^\infty(U^\theta)$ gives $f(y)=ce^{-y}$, $c\in\RR$. The resulting function $F:\,(x,y)\mapsto ce^{-y}$ clearly satisfies the Robin boundary condition on $\Sigma$, while Neumann boundary condition $\partial_\nu F=0$ on
the half-line $y=x\tan\theta$ takes the form
\[
0=(-\sin\theta,\cos\theta)\cdot \nabla F(x,y)\equiv (-\sin\theta,\cos\theta)\cdot \big(0,f'(y)\big)=-c\cos(\theta)e^{-y},
\]
and due to $\theta<\frac{\pi}{2}$ we have $c=0$ and then $f\equiv 0$ and, subsequently, $F=0$. Hence, \eqref{jmap} holds true.

In view of \eqref{jmap} and the unique continuation, for
\[
 K:=U^\theta\cap B_{R_0}(0)
\]
we also have
\begin{equation}
	\label{jmap2}
	\text{The maps }\cF\ni F\mapsto \partial_x F \in L^2(K),\ \cF\ni F\mapsto F \in L^2(K)
	 \text{ are injective}.
\end{equation}
As $\cF$ is finite-dimensional, one can find $c_1>0$ with 
\begin{equation*}
\|F\|^2_{L^\infty(U^\theta)}\le c_1 \|F\|^2_{L^2(K)},\quad F\in \cF.
\end{equation*}
Further note that the quadratic forms
\begin{align}
V\times\cF\ni (v,F)&\mapsto \|v\|^2_{L^2(U^\theta)}+\|F\|^2_{L^2(K)},\nonumber\\
V\times\cF\ni (v,F)&\mapsto \mu \|v\|^2_{L^2(U^\theta)}+\delta \iint_{K}|\partial_x(v+F)|^2\dxy \label{qf1}
\end{align}
are positive-definite: For the first one it is obvious, and if the right-hand side of \eqref{qf1} is zero, then $v=0$, thus $\|\partial_x F\|^2_{L^2(K)}=0$, which implies $F=0$ by \eqref{jmap2}. As both $V$ and $\cF$ are finite-dimensional,
one can find a constant $c_2>0$ such that for all $(v,F)\in V\times \cF$ one has
\[
\mu \|v\|^2_{L^2(U^\theta)}+\delta\iint_{K}|\partial_x(v+F)|^2\dxy\ge c_2 \Big(\|v\|^2_{L^2(U^\theta)}+\|F\|^2_{L^2(K)}
\Big).
\]

Substitute the last inequality into \eqref{nti3} and pick $R>R_0$ with $c_1\eps_R<c_2$, then
\[
c_3:=c_2-c_1\eps_R>0,\quad
	(n^{\theta'}+I)(w',w')\le -c_3\big( \|v\|^2_{L^2(U^\theta)}+\|F\|^2_{L^2(K)}\big).
\]
We use the finite-dimensionality of $W_R$ (with the chosen $R$) to find $c_4>0$ with
\[
\|v\|^2_{L^2(U^\theta)}+\|F\|^2_{L^2(K)}\ge c_4\|v+\chi_R F\|^2_{L^2(U^\theta)},\quad (v,F)\in V\times\cF,
\]
which results in
\[
	(n^{\theta'}+I)(w',w')\le -c_3c_4\|v+\chi_R F\|^2_{L^2(U^\theta)}\equiv -c_3c_4\|w\|^2_{L^2(U^\theta)}\equiv
	-c_3c_4\|w'\|^2_{L^2(U^{\theta'})}.	
\]
The min-max principle shows that the operator $N^{\theta'}+I$ has at least $\kappa+m$ negative eigenvalues, i.e. $\kappa(\theta')\ge \kappa+m$.
\end{proof}

\begin{proof}[\bf Proof of Lemma~\ref{thm31}]
The case $\theta\ge \frac{\pi}{2}$ is already covered by Lemma~\ref{lem32a}. Now let $\theta<\frac{\pi}{2}$.	
 If $\kappa$ is continuous at $\theta$ (and $\theta$ is non-critical), then the right-hand side of \eqref{ftheta} is zero, hence,
$\cF^N_\theta=\{0\}$. The jump of $\kappa$ at each critical $\theta_n$ is $1$, see~\eqref{kjump},
and Lemma~\ref{lem32} yields
$\dim\cF^N_{\theta_n}\le 1$. On the other hand, $\dim\cF^N_{\theta_n}\ge 1$ by the explicit computations
in Remark~\ref{rem26}. 
\end{proof}

\section{Non-resonant angles and threshold resonances}\label{sec4}

In this section we establish the equivalence (A)$\iff$(B) in Theorem~\ref{thm-main} by proving:
\begin{theorem}\label{thm-sec4}
A half-angle $\theta\in(0,\frac{\pi}{2})$ is non-resonant in the sense of \eqref{eq-nonres} if and only if the 
Robin Laplacian in the infinite sector $\cS_\theta$ has \emph{no} threshold resonance.
\end{theorem}

\subsection{Reduction by parity}

Recall that the non-resonance condition is an eigenvalue bound for the Robin-Neumann Laplacian $\Tilde N^r_\theta$
on the quadrangle $OA^+_rB_rA^-_r$ shown in Fig.~\ref{figure-oab}. Namely, the operator $\Tilde N^r_\theta$
acts as $u\mapsto-\Delta u$ on the functions $u$ satisfying the Robin boundary condition $\partial_\nu u=u$ on the sides $OA^\pm_r$ and Neumann boundary condition on the sides $A^\pm_rB_r$, and the half-angle $\theta$ is non-resonant
if for some $c>0$ one has
\begin{equation}
	\label{eq-nonres2}
	E_{\kappa(\theta)+1}(\Tilde N^r_\theta)\ge -1+\dfrac{c}{r^2} \text{ for all sufficiently large }r.
\end{equation}
We further recall that by \cite[Lem.~3.6]{kobp} one has
\begin{equation}
	\label{eq-low}
	E_j(\Tilde N^r_\theta)\xrightarrow{r\to\infty}\cE_j(\theta)<-1,\quad j=1,\dots,\kappa(\theta).
\end{equation}

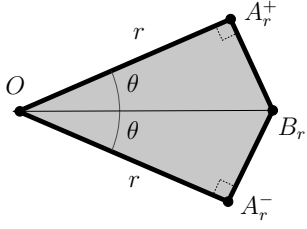
\begin{figure}
	\centering
	
	\begin{minipage}{0.4\textwidth}\centering
		\scalebox{0.8}{\begin{tikzpicture}[x=0.75pt,y=0.75pt,yscale=-1,xscale=1]
				
				\draw  [color={rgb, 255:red, 0; green, 0; blue, 0 }  ,draw opacity=1 ][fill={rgb, 255:red, 155; green, 155; blue, 155 }  ,fill opacity=0.55 ][line width=2.25]  (240.4,119.2) -- (212.4,176.4) -- (80.8,119.8) -- (214,62) -- cycle ;
				\draw  [draw opacity=0] (139.27,94.74) .. controls (142.57,102.43) and (144.4,110.9) .. (144.4,119.8) .. controls (144.4,128.49) and (142.66,136.78) .. (139.5,144.33) -- (80.8,119.8) -- cycle ; \draw  [color={rgb, 255:red, 74; green, 74; blue, 74 }  ,draw opacity=1 ] (139.27,94.74) .. controls (142.57,102.43) and (144.4,110.9) .. (144.4,119.8) .. controls (144.4,128.49) and (142.66,136.78) .. (139.5,144.33) ;  
				\draw  [fill={rgb, 255:red, 0; green, 0; blue, 0 }  ,fill opacity=1 ] (237.4,119.2) .. controls (237.4,117.54) and (238.74,116.2) .. (240.4,116.2) .. controls (242.06,116.2) and (243.4,117.54) .. (243.4,119.2) .. controls (243.4,120.86) and (242.06,122.2) .. (240.4,122.2) .. controls (238.74,122.2) and (237.4,120.86) .. (237.4,119.2) -- cycle ;
				\draw  [fill={rgb, 255:red, 0; green, 0; blue, 0 }  ,fill opacity=1 ] (78.8,119.8) .. controls (78.8,118.14) and (80.14,116.8) .. (81.8,116.8) .. controls (83.46,116.8) and (84.8,118.14) .. (84.8,119.8) .. controls (84.8,121.46) and (83.46,122.8) .. (81.8,122.8) .. controls (80.14,122.8) and (78.8,121.46) .. (78.8,119.8) -- cycle ;
				\draw  [fill={rgb, 255:red, 0; green, 0; blue, 0 }  ,fill opacity=1 ] (211,62) .. controls (211,60.34) and (212.34,59) .. (214,59) .. controls (215.66,59) and (217,60.34) .. (217,62) .. controls (217,63.66) and (215.66,65) .. (214,65) .. controls (212.34,65) and (211,63.66) .. (211,62) -- cycle ;
				\draw  [fill={rgb, 255:red, 0; green, 0; blue, 0 }  ,fill opacity=1 ] (209.4,176.4) .. controls (209.4,174.74) and (210.74,173.4) .. (212.4,173.4) .. controls (214.06,173.4) and (215.4,174.74) .. (215.4,176.4) .. controls (215.4,178.06) and (214.06,179.4) .. (212.4,179.4) .. controls (210.74,179.4) and (209.4,178.06) .. (209.4,176.4) -- cycle ;
				\draw  [dash pattern={on 0.75pt off 0.75pt}] (207.37,162.58) -- (216.8,166.98) -- (212.4,176.4) -- (202.98,172) -- cycle ;
				\draw  [dash pattern={on 0.75pt off 0.75pt}] (213.77,61.37) -- (218.21,70.77) -- (208.8,75.21) -- (204.37,65.8) -- cycle ;
				\draw    (81.8,119.8) -- (240.4,119.2) ;
				
				\draw (147.6,95) node [anchor=north west][inner sep=0.75pt]    {$\theta $};
				\draw (147.6,124.6) node [anchor=north west][inner sep=0.75pt]    {$\theta $};
				\draw (221.2,48.8) node [anchor=north west][inner sep=0.75pt]    {$A_{r}^{+}$};
				\draw (218.8,170.38) node [anchor=north west][inner sep=0.75pt]    {$A_{r}^{-}$};
				\draw (242.4,122.6) node [anchor=north west][inner sep=0.75pt]    {$B_{r}$};
				\draw (151.6,66.4) node [anchor=north west][inner sep=0.75pt]    {$r$};
				\draw (148.6,159) node [anchor=north west][inner sep=0.75pt]    {$r$};
				\draw (71.6,96.8) node [anchor=north west][inner sep=0.75pt]    {$O$};

			\end{tikzpicture}
		}
	\end{minipage}
	\begin{minipage}{0.55\textwidth}\caption{The quadrangle $OA^+_rB_rA^-_r$.}\label{figure-oab}\end{minipage}
\end{figure}

Using the reduction by parity (with respect to the diagonal $OB_r$)
one easily shows
\begin{equation}
	\label{decomp1}
\Tilde N^r_\theta\simeq N^r_\theta\oplus D^r_\theta,
\end{equation}
where $N^r_\theta$ and $D^r_\theta$ are the operators in the triangle $OA^+_r B_r$
acting as $u\mapsto-\Delta u$ on the functions $u$ satisfying
\begin{align*}
	\partial_\nu u&=u \text{ on }OA^+_r \text{ for both $N^r_\theta$ and $D^r_\theta$},\\
	\partial_\nu u&=0 \text{ on }A^+_r B_r\text{ for both $N^r_\theta$ and $D^r_\theta$},\\
	\partial_\nu u&=0 \text{ on }OB_r \text{ for } N^r_\theta,\\
	u&=0 \text{ on }OB_r \text{ for } D^r_\theta.
\end{align*}
It is shown in \cite[Prop.~3.10]{kobp},
\begin{equation}
\text{for any }C_D\in\big(0,\frac{\pi^2}{4}\big):\quad
D^r_\theta\ge -1+\dfrac{C_D}{r^2} \text{ for all sufficiently large }r,
\label{eq-dr}
\end{equation}
and the decomposition \eqref{decomp1} together with \eqref{eq-low} show that the non-resonance condition \eqref{eq-nonres2}
is equivalent to the existence of $C>0$ with
\begin{equation}
	\label{eq-nonres3}
	E_{\kappa(\theta)+1}(N^r_\theta)\ge -1+\dfrac{C}{r^2} \text{ for all sufficiently large }r.
\end{equation}

Similarly to Section~\ref{sec3} it will be convenient to put the Robin  boundary
on one of the coordinate axes. We drop the dependence on the angle $\theta\in(0,\frac{\pi}{2})$ as it remains unchanged
from now on, and for $r>0$ consider the sets
\begin{align*}
U&:=\big\{(x,y):\ x>0,\ 0<y<x\tan\theta\big\},\\
\Sigma&:=\big\{(x,y):\ x>0,\ y=0\big\}\subset\partial U &&\text{ (Robin boundary of $U$),}\\
U_r&:=U\cap\{x<r\},\\
\Sigma_r&:=\big\{(x,y):\ x\in (0,r),\ y=0\big\}\subset\partial U_r &&\text{ (Robin boundary of $U_r$),}
\end{align*}	
and denote
\begin{equation}
	\label{eq-ar}
	A_r:=(r,r\tan\theta)
\end{equation}
the upper right corner of $U_r$, see Fig.~\ref{fig-usr}.
Further let $N_r$ be the Laplacian $u\mapsto -\Delta u$ in $L^2(U_r)$ with Robin boundary condition  $\partial_\nu u=u$ on  $\Sigma_r$ and Neumann boundary condition on the remaining parts of boundary i.e. on $\{y=x\tan\theta\}$ and $\{x=r\}$, and we denote
\[
\kappa:=\kappa(\theta)\ge 1.
\]
By construction, the operator $N_r$ is unitarily equivalent to the above $N^r_\theta$, so the non-resonance conditions
\eqref{eq-nonres2} and \eqref{eq-nonres3} are equivalent
to the existence of $C>0$ with
\begin{equation}
	\label{eknr}
	E_{\kappa+1}(N_r)\ge -1+\dfrac{C}{r^2} \text{ for all sufficiently large }r.
\end{equation}

\begin{figure}
	\centering
	
	\begin{minipage}{0.45\textwidth}\centering
		\scalebox{0.8}{\begin{tikzpicture}[x=0.75pt,y=0.75pt,yscale=-1,xscale=1]
				
				\draw [color={rgb, 255:red, 155; green, 155; blue, 155 }  ,draw opacity=1 ] (66.7,139.67) -- (340.83,139.67)(79.17,53) -- (79.17,155.27) (333.83,134.67) -- (340.83,139.67) -- (333.83,144.67) (74.17,60) -- (79.17,53) -- (84.17,60)  ;
				\draw  [draw opacity=0][fill={rgb, 255:red, 155; green, 155; blue, 155 }  ,fill opacity=0.42 ][line width=2.25]  (290.5,140) -- (78.6,139.07) -- (272.5,55.67) -- cycle ;
				\draw    (230.17,73.33) -- (229.67,140.67) ;
				\draw  [draw opacity=0][fill={rgb, 255:red, 155; green, 155; blue, 155 }  ,fill opacity=0.72 ][line width=2.25]  (230.5,139) -- (79.17,139.67) -- (230.17,73.33) -- cycle ;
				\draw [color={rgb, 255:red, 0; green, 0; blue, 0 }  ,draw opacity=1 ][line width=1.5]    (78.6,139.07) -- (293.17,46) ;
				\draw  [draw opacity=0] (137.07,114.01) .. controls (140.37,121.7) and (142.2,130.17) .. (142.2,139.07) .. controls (142.2,139.13) and (142.2,139.2) .. (142.2,139.27) -- (78.6,139.07) -- cycle ; \draw  [color={rgb, 255:red, 74; green, 74; blue, 74 }  ,draw opacity=1 ] (137.07,114.01) .. controls (140.37,121.7) and (142.2,130.17) .. (142.2,139.07) .. controls (142.2,139.13) and (142.2,139.2) .. (142.2,139.27) ;  
				\draw [color={rgb, 255:red, 0; green, 0; blue, 0 }  ,draw opacity=1 ][line width=1.5]    (78.6,139.07) -- (319.5,139.33) ;
				\draw  [color={rgb, 255:red, 255; green, 255; blue, 255 }  ,draw opacity=1 ][fill={rgb, 255:red, 0; green, 0; blue, 0 }  ,fill opacity=1 ] (226.83,72.83) .. controls (226.83,71.27) and (228.1,70) .. (229.67,70) .. controls (231.23,70) and (232.5,71.27) .. (232.5,72.83) .. controls (232.5,74.4) and (231.23,75.67) .. (229.67,75.67) .. controls (228.1,75.67) and (226.83,74.4) .. (226.83,72.83) -- cycle ;
				\draw [color={rgb, 255:red, 0; green, 0; blue, 0 }  ,draw opacity=1 ][line width=3]    (79.17,139.67) -- (229.67,139.67) ;
				\draw  [color={rgb, 255:red, 255; green, 255; blue, 255 }  ,draw opacity=1 ][fill={rgb, 255:red, 0; green, 0; blue, 0 }  ,fill opacity=1 ] (226.83,139.83) .. controls (226.83,138.27) and (228.1,137) .. (229.67,137) .. controls (231.23,137) and (232.5,138.27) .. (232.5,139.83) .. controls (232.5,141.4) and (231.23,142.67) .. (229.67,142.67) .. controls (228.1,142.67) and (226.83,141.4) .. (226.83,139.83) -- cycle ;
				
				\draw (330.27,114.73) node [anchor=north west][inner sep=0.75pt]  [color={rgb, 255:red, 128; green, 128; blue, 128 }  ,opacity=1 ]  {$x$};
				\draw (86.87,53.2) node [anchor=north west][inner sep=0.75pt]  [color={rgb, 255:red, 128; green, 128; blue, 128 }  ,opacity=1 ]  {$y$};
				\draw (145.4,114.27) node [anchor=north west][inner sep=0.75pt]    {$\theta $};
				\draw (258.33,94.8) node [anchor=north west][inner sep=0.75pt]    {$U$};
				\draw (297.4,117.4) node [anchor=north west][inner sep=0.75pt]    {$\Sigma $};
				\draw (285.87,48.47) node [anchor=north west][inner sep=0.75pt]    {$y=x\tan \theta $};
				\draw (184,105.8) node [anchor=north west][inner sep=0.75pt]    {$U_{r}$};
				\draw (123.33,47.8) node [anchor=north west][inner sep=0.75pt]    {$( r,r\tan \theta ) =A_{r}$};
				\draw (225,147) node [anchor=north west][inner sep=0.75pt]    {$r$};
				\draw (156.83,146) node [anchor=north west][inner sep=0.75pt]    {$\Sigma _{r}$};

			\end{tikzpicture}
		}
	\end{minipage}
	\begin{minipage}{0.5\textwidth}\caption{\small The sets $U$, $\Sigma$, $U_r$ and $\Sigma_r$.}\label{fig-usr}\end{minipage}
\end{figure}
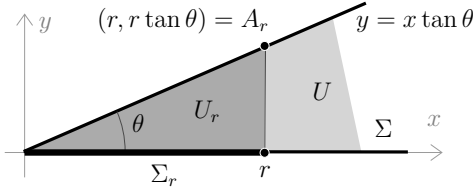

We further recall that due to the considerations of Section~\ref{sec3}, in particular, the equality \eqref{fsum3},
one has a threshold resonance in $\cS_\theta$ if and only if there exists
a \emph{bounded} weak solution $F\not\equiv 0$ of the boundary value problem
\begin{equation}
	\label{FFF}
	\left\{
	\begin{aligned}
		-\Delta F&=-F \text{ in } U,\\
		\partial_\nu F&=F \text{ on } \Sigma,\\
		\partial_\nu F&=0 \text{ on } \{x>0,\ y=x\tan\theta\}.
	\end{aligned}\right.
\end{equation}
These observations show that Theorem~\ref{thm-sec4} becomes a consequence
of the following statement:
\begin{lemma}\label{thm41}
The non-resonance condition \eqref{eknr} fails if and only if
 the problem \eqref{FFF} has a bounded solution $F\not\equiv 0$.
\end{lemma}

\subsection{Outline of the proof}

The direction $\Leftarrow$ in Lemma~\ref{thm41} will be proved in a quite straightforward way.
This step  uses the facts proved in Section~\ref{sec3}, in particular, the structure
of the possible bounded solutions $F$, and we employ the standard idea that these $F$ can be used as
additional test functions for $N_r$.

\begin{lemma}\label{prop41a}
	Assume that the problem \eqref{FFF} has a bounded solution $F\not\equiv 0$, then for some $c>0$ one has
	\[
	E_{\kappa+1}(N_r)\le -1+c^{-1}e^{-cr} \text{ for }r\to\infty,
	\]
	which violates the non-resonance condition \eqref{eknr}.
\end{lemma}

\begin{proof}
As discussed in Section~\ref{sec3}, the assumption implies that $\theta$ is a critical half-angle, and
then $F$ is a multiple of the Ursell solution  $\Tilde\Psi_n$ in adapted coordinates as described in \eqref{urssely} and \eqref{urssel-decay}. Hence, by adjusting the constant factor one assumes without loss of generality:
\begin{equation}
	\label{FFF3}
F(x,y)=e^{-y}+ G(x,y), 	\quad \big|G(x,y)\big|+\big|\nabla G(x,y)\big|\le d^{-1} e^{-d x}, \quad (x,y)\in U
\end{equation}
with some $d>0$.

Denote by $N$ the Laplacian $u\mapsto-\Delta u$ in the infinite sector $U$ with the Robin boundary condition $\partial_\nu u=u$ on $\Sigma$ and Neumann boundary condition on the oblique boundary $\{x>0,\ y=x\tan\theta\}$.
Further let $n$ and $n_r$ be the bilinear forms for $N$ and $N_r$, respectively,
\begin{align*}
	n(v,v)&=\iint_U |\nabla v|^2\,\dxy-\int_{\Sigma}v^2\,\dd s, & v&\in H^1(U),\\
	n_r(u,u)&=\iint_{U_r} |\nabla u|^2\,\dxy-\int_{\Sigma_r}u^2\,\dd s, & u&\in H^1(U_r).
\end{align*}

As the operator $N+I$ has $\kappa$ strictly negative eigenvalues $\cE_1(\theta)+1,\dots,\cE_\kappa(\theta)+1$
and the functions with bounded supports are dense in $H^1(U)$, see e.g. \cite[Thm.~1.4.2.1]{grisvard}, due to the
min-max principle we can find
a $\kappa$-dimensional subspace $V\subset H^1(U)$ such that
\[
(n+I)(v,v)\le 0 \text{ for all } v\in V
\]
and for a large $R>0$ one has $U\cap \supp v\subset U_{R}$ for all $v\in V$. For any $r>R$ the subspace $V$ can also be viewed as a subset of $H^1(U_r)=\dom n_r$, with
\begin{equation}
	\label{nri}
(n_r+I)(v,v)\le 0 \text{ for all } v\in V,\ r>R.
\end{equation}

By \eqref{FFF3} we may increase $R$ to have $\big|G(x,y)\big|<\frac{1}{4}$ for all $x>R$.
Then for a sufficiently small $\delta>0$ and all $r>R$ we have
\begin{gather*}
	(R,r)\times(0,\delta)\subset U_r,\quad F(x,y)\ge \frac{1}{2} \text{ for all }(x,y)\in (R,r)\times(0,\delta),\\
\|F\|^2_{L^2(U_r)}\ge\iint_{(R,r)\times(0,\delta)} |F|^2\dxy\ge \iint_{(R,r)\times(0,\delta)} \dfrac{1}{4}\,\dxy
= \dfrac{\delta}{4}(r-R),
\end{gather*}
therefore, for some $a_0>0$ we have
\begin{equation}
	\label{c0f}
\|F\|^2_{L^2(U_r)}\ge a_0 r,\quad r\to \infty.
\end{equation}

As $\supp F$ is unbounded (due to the unique continuation principle), one has $F\notin V$, and the subspace
\[
W:=V+\RR F
\]
is a $(\kappa+1)$-dimensional subspace of $H^1(U_r)$ for all $r>R$.
Let
\[
w=v+\xi F\in W, \quad \xi\in \RR.
\]
For any $v\in V$ we have
\[
\langle v,F\rangle_{L^2(U_r)}=\langle v,F_V\rangle_{L^2(U_r)},
\]
where $F_V\in V$ is the orthogonal projection of $F$ on $V$ in $L^2(U_r)$, which is independent of $r$
in view of the above support properties, i.e. the number
\[
a_1:=\|F_V\|_{L^2(U_r)}
\]
is independent of $r$.
Therefore,
\begin{equation}
	\label{wcf}
\begin{aligned}
\|w\|^2_{L^2(U_r)}&=\|v\|^2_{L^2(U_r)}+2\xi\langle v,F\rangle_{L^2(U_r)}+\xi^2\|F\|^2_{L^2(U_r)}\\
&=\|v\|^2_{L^2(U_r)}+2\xi\langle v,F_V\rangle_{L^2(U_r)}+\xi^2\|F\|^2_{L^2(U_r)}\\
&\ge \|v\|^2_{L^2(U_r)}-2|\xi| a_1\|v\|_{L^2(U_r)}+\xi^2\|F\|^2_{L^2(U_r)}\\
&= \big(\|v\|_{L^2(U_r)}-|\xi| a_1\big)^2+ \xi^2\big(\|F\|^2_{L^2(U_r)}-a_1^2\big)\\
&\stackrel{\eqref{c0f}}{\ge} \xi^2 (a_0 r -a_1^2)\ge \dfrac{a_0 r}{2} \xi^2.
\end{aligned}
\end{equation}

At the same time, we have
\begin{align*}
(n_r+I)(w,w)&=(n_r+I)(v,v)+2\xi(n_r+I)(v,F)+\xi^2(n_r+I)(F,F),
\end{align*}
and the weak formulation \eqref{FFF-aa2} of~\eqref{FFF} with $\varphi:=v$ gives
\[
(n_r+I)(v,F)=0.
\]
Further, taking into account the properties \eqref{FFF} of $F$ we have, using the integration by parts,
\begin{align*}
	(n_r+I)(F,F)&=\iint_{U_r} |\nabla F|^2\,\dxy-\int_{\Sigma_r}F^2\,\dd s +\iint_{U_r}F^2\,\dxy\\
	&=\iint_{U_r} F(-\Delta F)\,\dxy+\int_{\partial U_r} F\partial_\nu F\,\dd s
	-\int_{\Sigma_r}F^2\,\dd s +\iint_{U_r}F^2\,\dxy\\
	&=\int_{\{x=r\}\cap \partial U_r} F\partial_\nu F\,\dd s
	=\int_0^{r\tan\theta} F(r,y)\partial_x F(r,y)\,\dd y,
\end{align*}
due to the representation~\eqref{FFF3} we obtain
\[
\big|	(n_r+I)(F,F)\big|\le \int_0^{r\tan\theta} (e^{-y}+d^{-1}e^{-dr})d^{-1}e^{-dr}\dd y\le a_2 e^{-dr},\quad a_2>0.
\]
It follows that
\begin{align*}
	(n_r+I)(w,w)&=\underbrace{(n_r+I)(v,v)}_{\le 0 \text{ by }\eqref{nri}}+\xi^2(n_r+I)(F,F)\\
	&\le \xi^2 a_2 e^{-dr}\stackrel{\eqref{wcf}}{\le } \dfrac{2 a_2 e^{-dr}}{a_0 r} \|w\|^2_{L^2(U_r)},
\end{align*}	
and the min-max principle yields
\[
E_{\kappa+1}(N_r)+1\le \dfrac{2 a_2 e^{-dr}}{a_0 r}.\qedhere
\]
\end{proof}

The direction $\Rightarrow$ in Lemma~\ref{thm41} is covered in the following assertion.

\begin{lemma}\label{prop41}
Assume that there exist a sequence $(r_j)\subset(0,\infty)$ with $r_j\xrightarrow{j\to\infty}\infty$
and a constant $b\in(0,\frac{\pi^2}{4})$	 such that
\begin{equation}
	\label{lest2}
\lambda_j:=E_{\kappa+1}(N_{r_j})\le -1+\dfrac{b}{r_j^2} \text{ for }j\to\infty.
\end{equation}
Then there exists a bounded weak solution $F\not\equiv 0$ of \eqref{FFF}.
\end{lemma}	

Note that this actually gives a stronger implication in Lemma~\ref{thm41}: in the absence of threshold resonance, i.e. if the problem \eqref{FFF} has no non-trivial bounded solutions, the non-resonance condition \eqref{eq-nonres3}, or \eqref{eq-nonres2}, holds with any $c\in(0,\frac{\pi^2}{4})$.

The quite involved proof of Lemma~\ref{prop41} occupies the rest of the present Section~\ref{sec4}. We note that the proof does \emph{not} use the explicit form of the bounded solutions and remains independent of the results obtained in~Section~\ref{sec3}.

Recall first that the analysis of \cite[Lem.~3.5--3.6]{kobp} gives the lower bound
\begin{equation}
	E_{\kappa+1}(N_r)\ge -1-\dfrac{a}{r^2}
	\label{lest1}
\end{equation}
for a suitable $a>0$ and all sufficiently large $r$.
Let $u_j$ be eigenfunctions of $N_{r_j}$ for the eigenvalues $\lambda_j$. Due to $L^2$--$L^\infty$-estimates (Proposition \ref{prop-nittka})
one has $u_j\in L^\infty(U_{r_j})$, and without loss of generality we assume the normalization
\begin{equation}
	\label{unorm}
\|u_j\|_{L^\infty(U_{r_j})}= 1 \text{ for all }j.
\end{equation}

Our proof will consist of two main steps. First, we show that a subsequence of $(u_j)$	converges locally uniformly 
to a bounded weak solution $F$ of \eqref{FFF}: this is done in Subsection~\ref{sec-FF} by using compactness arguments
and weak convergence. In the second part of the proof, which is significantly more demanding, we show that $F$ is not identically zero. In fact we show that $F=0$ implies $\|u_j\|_\infty\to0$, which contradicts the normalization \eqref{unorm}.
Namely, by assuming $F=0$ we first show in Subsection~\ref{ssec-oblique} that $u_j$ converge uniformly to zero in any strip of fixed width along the oblique boundary $y=x\tan\theta$. To analyze the behavior of $u_j$ in the rest of $U_{r_j}$ we apply
a suitable truncation (with a transition region contained in a strip of the above type) and extend all functions
by $0$ to the infinite half-strip $\Pi_j:=(0,r_j)\times(0,\infty)$: the truncated functions $v_j$ appear to be solutions
of the problem \eqref{eq-vj1}--\eqref{eq-vj2} involving the Laplacian $H_j$ on $\Pi_j$ with Robin boundary condition $\partial_\nu v=v$ on $\{y=0\}$, the Dirichlet boundary condition on $\{x=0\}$ and Neumann boundary condition on $\{x=r_j\}$.
Note that 
\[
\inf  \spec H_j=-1+\dfrac{\pi^2}{4r_j^2},
\]
so the assumption \eqref{lest2} means that $E_{\kappa+1}(N_{r_j})$ keeps a controllable distance of order $r_j^{-2}$ to $\spec H_j$.
The operator $H_j$ admits a separation of variables and can be written using tensor products
of the one-dimensional operators introduced in \S\S~\ref{ssec-dn}--\ref{ssec-robin}. One then decomposes $v_j$
with respect to the spectral projector on the lowest eigenvalue in the $y$-direction (see \S~\ref{item443})
and shows that each of the components converges uniformly to zero (Subsections \ref{item444} and~\ref{item445}).
The eigenvalue bound~\eqref{lest2} is used in Lemma~\ref{lem-gj} to guarantee a suitable upper bound for the Green function in $x$-direction. All these observations are put together in Subsection~\ref{ssec-prop41}, which completes the proof of Lemma~\ref{prop41}.

\subsection{Constructing a candidate for a threshold resonance}\label{sec-FF}

Note that by \eqref{lest1} and \eqref{lest2} for large $j$ we have
\begin{equation}
	\label{eq-lambdaj}
-1-\dfrac{a}{r_j^2}\le \lambda_j\le -1+\dfrac{b}{r_j^2},\quad -a\le r^2_j(\lambda_j+1)\le b,\qquad
\lambda_j\xrightarrow{j\to\infty}-1,
\end{equation}
so we assume without loss of generality from now on that
\[
-\dfrac{3}{2}<\lambda_j<-\dfrac{1}{2} \text{ for all }j.
\]

\begin{lemma}\label{lem42}
	There exists an infinite index set $J\subset\NN$ such that the subsequence $(u_j)_{j\in J}$
	converges locally uniformly in $U$ to a function $F\in C^\infty(U)$ that satisfies $\|F\|_{L^\infty(U)}\le 1$ and
	$-\Delta F=-F$ in $U$.	
\end{lemma}

\begin{proof}
	For a large $m\in\NN$ and $n\in\NN_{\ge m}$ consider the bounded domains
	\[
	V_n:=\big\{ (x,y)\in U:\  x<n,\ \dist\big( (x,y), \partial U\big)>\frac{1}{n}  \big\}\subset U.
	\]
	Obviously,
	\[
	\overline{V_{n}}\subset V_{n+1},\quad \bigcup_{n\ge m}V_n=U,\quad V_n\subset U_r \text{ for }r\ge n.
	\]
	According to local Calder\'on-Zygmund estimates (Proposition~\ref{calzyg}) there are constants $C_n>0$ with
\begin{multline}
\|u\|_{H^2(V_n)}\le C_n\big(\|u\|_{L^2(V_{n+1})}+\|\Delta u\|_{L^2(V_{n+1})}\big)\\
\text{ for all } u\in H^2_\loc(V_{n+1})\cap L^2(V_{n+1}) \text{ with } \Delta u\in L^2(V_{n+1}).
 \label{locest1}
\end{multline}

Due to elliptic regularity one has $u_j\in H^2_\loc(U_{r_j})\cap L^2(U_{r_j})$ and $\Delta u_j=-\lambda_j u_j\in L^2(U_{r_j})$.
Hence, the preceding estimate \eqref{locest1} is applicable to $u_j$ for $r_j\ge n+1$: using $|\lambda_j|\le 2$ and the normalization \eqref{unorm} we obtain
\begin{equation}
	\label{uj-h2}
\|u_j\|_{H^2(V_n)}\le 3 C_n\|u_j\|_{L^2(V_{n+1})}\le 3C_n\sqrt{|V_{n+1}|} \text{ for }r_j\ge n+1.
\end{equation}

By \eqref{uj-h2}, for each $n$ the sequence $(u_j)$ (with sufficiently large $j$) is bounded in $H^2(V_n)$,
hence, it contains a weakly convergent subsequence, which is then uniformly convergent on $V_n$ due to the compactness of the embedding $H^2(V_n)\hookrightarrow C^0(\overline{V_n})$, while the latter space is endowed with the $L^\infty$-norm.
In order to exploit these observations properly, we first choose an infinite index set $J_m\subset \NN_{\ge m}$ such that 
$r_j\ge m$ for all $j\in J_m$ and $(u_j)_{j\in J_m}$ uniformly converges on $V_m$ to some function
$F_m\in C^0(\overline{V_m})$. Then for each $n\in\NN_{\ge m}$ we choose iteratively an infinite index set $J_{n+1}\subset J_n$ such that $r_j\ge n+1$ for all $j\in J_{n+1}$ and  $(u_j)_{j\in J_{n+1}}$
uniformly converges on $V_{n+1}$ to some function $F_{n+1}\in C^0(\overline{V_{n+1}})$. By construction we have $F_k=F_n$ on $V_n$ for all $k\ge n$, which allows one to define a function $F\in C^0(U)$ by
\[
F(x):=F_n(x) \text{ for }x\in V_n.
\]
In view of the normalization \eqref{unorm} of $u_j$ we have $\|F\|_{L^\infty(V_n)}\le 1$ for each $n$, hence, 
$\|F\|_{L^\infty(U)}\le 1$.

Now for each $n\in\NN_{\ge m}$ let us choose iteratively $j_n\in J_n$ with $j_{n+1}>j_n$, and then we denote $J:=(j_n)_{n\ge m}$. For each $n\in\NN_{\ge m}$ one has $(j_k)_{k\ge n}\subset J_n$, which shows that $(u_j)_{j\in J}$ is (up to finitely many terms) a subsequence of $(u_j)_{j\in J_n}$
and converges to $F$ uniformly on $V_n$. As each compact subset of $U$ is contained in one of $V_n$, the subsequence
$(u_j)_{j\in J}$ converges to $F$ locally uniformly in $U$.

Let $\varphi \in C^\infty_c(U)$. For each $j\in J$ one has $-\Delta u_j =\lambda_j u_j$ in $U_{r_j}$, hence,
by assuming that $j$ is sufficiently large to have $\supp\varphi\subset U_{r_j}$ we obtain
\[
-\iint_U u_j\Delta\varphi\,\dxy= \lambda_j\iint_U u_j\varphi\,\dxy.
\]
As $u_j$ converges to $F$ uniformly on the compact subset $\supp\varphi\subset U$ and $\lambda_j$ converges to $-1$,
one can pass to the limit $j\to \infty$ (with $j\in J$) to obtain
\[
-\iint_U F\Delta\varphi\,\dxy= -\iint_U F\varphi\,\dxy.
\]
As $\varphi \in C^\infty_c(U)$ was arbitrary, one concludes that $-\Delta F=-F$ in $\cD'(U)$, and the elliptic regularity theorem implies $F\in C^\infty(U)$.	
\end{proof}

\emph{We denote the subsequence $(u_j)_{j\in J}$ constructed in Lemma~\ref{lem42} simply by $(u_j)_{j\in\NN}$ to keep simple notation.}

Note that Lemma~\ref{lem42} does not guarantee that the function $F$ is well-behaved near $\partial U$, in particular, the necessary boundary
conditions in \eqref{FFF} are not shown yet: this will be the subject of the next steps.

We will start with a simple estimate:
\begin{lemma}\label{lem-bd0}
For any fixed bounded measurable $V\subset U$ one has
\[
\|u_j-F\|_{L^2(V)}\xrightarrow{j\to\infty}0.
\]
\end{lemma}

\begin{proof}
	Due to the boundedness of $V$ there is $j_0\in\NN$ with $V\subset U_{r_j}$ for all $j\ge j_0$, therefore, the respective $u_j$ are defined on $V$.	Let $\eps>0$ and $K\subset V$ be a compact subset with $|V\setminus K|<\eps$. By construction $u_j$ converge to $F$ uniformly in $K$, and the normalization \eqref{unorm} of $u_j$ together with $|F|\le 1$ gives the integrable upper bound $|u_j-F|^2\le 4$ on $V$. The dominated convergence theorem implies
	\begin{align*}
		\limsup_{j\to\infty} \iint_{V} |u_j-F|^2\,\dxy&=
		\limsup_{j\to\infty} \Big(\iint_{V\setminus K} |u_j-F|^2\,\dxy\\
		&\qquad\qquad\qquad\qquad+\underbrace{\iint_{K} |u_j-F|^2\,\dxy}_{\le \|u_j-F\|^2_{L^\infty(K)}|K|\to 0}\Big)\\
		&\le \limsup_{j\to\infty} \iint_{V\setminus K} \underbrace{|u_j-F|^2}_{\le 4}\,\dxy\le 4\big|V\setminus K\big|<4\eps.
	\end{align*}
	As $\eps>0$ was arbitrary, one obtains $\|u_j\|^2_{L^2(V)}\to 0$.
\end{proof}

\begin{lemma}\label{corol-grad}
	For any $0<r<R$ one can find $j_0\in \NN$ and $c>0$ such that
	\[
	\int_{U\cap B_r(0)} |\nabla u_j|^2\,\dd x\le c\int_{U\cap B_R(0)} u_j^2\,\dd x \text{ for all }j\ge j_0.
	\]
\end{lemma}

\begin{proof}
(a) Let $\eta\in C^\infty_c(\RR^2)$. One picks a large $j_0\in\NN$ with $\supp\eta\cap U\subset U_{r_j}$ for all $j\ge j_0$, which we assume from now on.
	
In view of the definition of $u_j$, for any $\varphi\in H^1(U_{r_j})$ one has
\begin{equation}
	\label{test1}
\iint_{U_{r_j}}\nabla u_j\cdot\nabla\varphi\,\dxy
-\int_{\Sigma_{r_j}} u_j\varphi\,\dd s=
\lambda_j\iint_{U_{r_j}}u_j\varphi\,\dxy.
\end{equation}
In particular, one can set $\varphi:=\eta^2 u_j$: we have
\begin{equation}
	\label{etau}
\begin{aligned}
	\nabla u_j\cdot \nabla(\eta^2 u_j)&=\nabla u_j\cdot(2\eta u_j\nabla\eta+\eta^2\nabla u_j)
	=2\eta u_j \nabla u_j\cdot\nabla\eta+\eta^2|\nabla u_j|^2\\
	&=|\eta \nabla u_j+ u_j\nabla\eta|^2-|\nabla \eta|^2 u_j^2
	=\big|\nabla(\eta u_j)\big|^2-|\nabla \eta|^2 u_j^2.
\end{aligned}
\end{equation}
The substitution into \eqref{test1} gives
\[
	\iint_{U_{r_j}}\big|\nabla(\eta u_j)\big|^2\dxy-\iint_{U_{r_j}}|\nabla \eta|^2 u_j^2\,\dxy
	-\int_{\Sigma_{r_j}} (\eta u_j)^2\,\dd s=
	\lambda_j\iint_{U_{r_j}}(\eta u_j)^2\,\dxy,
\]
and because of $\lambda_j<0$ we obtain
\begin{equation}
	\label{test2}
	\iint_{U}\big|\nabla(\eta u_j)\big|^2\dxy\le
	\iint_{U_{r_j}}|\nabla \eta|^2 u_j^2\,\dxy+
	\int_{\Sigma} (\eta u_j)^2\,\dd s.
\end{equation}

The trace theorem on $U$ shows that one can find $c\ge 1$ such that
\[
\int_{\Sigma} v^2\,\dd s\le
\int_{\partial U} v^2\,\dd s\le \dfrac{1}{2}\iint_U |\nabla v|^2\,\dxy+ c\iint_U v^2\,\dxy \text{ for all }v\in H^1(U).
\]
The substitution into \eqref{test2} yields
\[
	\iint_{U}\big|\nabla(\eta u_j)\big|^2\dxy\le \iint_{U}|\nabla\eta|^2 u_j^2\dxy+
	\dfrac{1}{2}\iint_U \big|\nabla(\eta u_j)\big|^2\,\dxy+ c\iint_U (\eta u_j)^2\,\dxy,
\]
which results in
\begin{equation}
	\label{cuj}
\iint_{U}\big|\nabla(\eta u_j)\big|^2\dxy\le 2c\iint_U \big(|\nabla\eta|^2+\eta^2\big) u_j^2\,\dxy. 
\end{equation}

(b) Now let $0<r<R$. Pick $\eta\in C^\infty_c(\RR^2)$ with $\eta=1$ in $B_r(0)$
	and $\supp\eta\subset B_R(0)$. Choose $j_0$ and $c>0$ as in the part (a), then \eqref{cuj}
	holds for all $j\ge j_0$, and
	\begin{align*}
		\iint_{U\cap B_r(0)} &|\nabla u_j|^2\,\dxy\le \iint_{U} \big|\nabla (\eta u_j)\big|^2\,\dxy\\
		&\le 2c\iint_U \big(|\nabla\eta|^2+\eta^2\big) u_j^2\,\dxy
		=2c\iint_{U\cap B_R(0)} \big(|\nabla\eta|^2+\eta^2\big) u_j^2\,\dxy\\
		&\le 2c\big(\|\nabla\eta\|^2_\infty+\|\eta\|^2_\infty\big)\iint_{U\cap B_R(0)} u_j^2\,\dxy.\qedhere
	\end{align*}
\end{proof}

\begin{lemma}\label{lem48a}
	The function $F$ constructed in Lemma~\ref{lem42} is $H^1$-regular on each bounded subdomain of $U$ and satisfies the boundary value problem \eqref{FFF} in the weak sense.
\end{lemma}

\begin{proof}
	(a) For each $n\in\NN$ denote
\[
V_n:=U\cap B_n(0)
\]	
and choose $j_n\in\NN$ and $C_n>0$ with
\[
	\iint_{V_n} |\nabla u_j|^2\,\dxy\le C_n\iint_{V_{n+1}} u_j^2\,\dxy \text{ for all }j\ge j_n,
\]
which is possible by virtue of Lemma~\ref{corol-grad}.

(b) Let $n$ be fixed. For $j\ge j_n$ the normalization \eqref{unorm} of $u_j$ implies
\[
\iint_{V_n} u_j^2\,\dxy\le |V_n|,\quad
\iint_{V_n} |\nabla u_j|^2\,\dxy\le C_n |V_{n+1}|.
\]	
Hence, the sequence $(u_j)_{j\ge j_n}$ is bounded in $H^1(V_n)$
and contains a subsequence $(u_j)_{j\in J_n}$ which converges weakly to some $G_n\in H^1(V_n)$.
As the embedding $H^1(V_n)\hookrightarrow L^2(V_n)$ is compact, the same subsequence converges to $G_n$
strongly in $L^2(V_n)$. On the other hand, by Lemma~\ref{lem-bd0} the functions $u_j$ converge to $F$ strongly
in $L^2(V_n)$, which yields $F=G_n\in H^1(V_n)$, and $(u_j)_{j\in J_n}$ converges to $F$ weakly in $H^1(V_n)$.

(c) As any bounded subdomain of $U$ is contained in one of $V_n$, by (b) one obtains $F\in H^1(U')$ for any bounded open subdomain $U'\subset U$.

(d) Let $\varphi\in H^1(U)$ have bounded support. One can choose $n\in\NN$ with $U\cap\supp\varphi\subset V_n$. For the sequence
$(u_j)_{j\in J_n}$ in (b) the weak convergence to $F$ in $H^1(V_n)$ implies
\begin{align*}
\iint_{V_n} u_j \varphi\,\dxy&\xrightarrow{j\in J_n,\, j\to\infty}\iint_{V_n} F \varphi\,\dxy\equiv \iint_{U} F \varphi\,\dxy,\\
\iint_{V_n} \nabla u_j\cdot\nabla \varphi\,\dxy&\xrightarrow{j\in J_n,\, j\to\infty}\iint_{V_n} \nabla F\cdot\nabla \varphi\,\dxy\equiv \iint_{U} \nabla F\cdot\nabla \varphi\,\dxy,
\end{align*}
and the continuity of the Sobolev trace on $V_n$ gives, in view of $\Sigma_n\subset\partial V_n$,
\[
\int_{\Sigma_n} u_j \varphi\,\dd s\xrightarrow{j\in J_n,\, j\to\infty}\int_{\Sigma_n} F \varphi\,\dd s\equiv \int_{\Sigma} F \varphi\,\dd s.
\]

Recall that for each $j$ one has
\[
\iint_{V_n}\nabla u_j\cdot\nabla\varphi\,\dxy-\int_{\Sigma_n} u_j\varphi\,\dd s=\lambda_j\iint_{V_n} u_j\varphi\,\dxy,
\]
and for $j\to\infty$ with $j\in J_n$ one obtains
\[
\iint_{U}\nabla F\cdot\nabla\varphi\,\dxy-\int_{\Sigma} F\varphi\,\dd s=-\iint_{U} F\varphi\,\dxy,
\]
which is exactly the weak formulation of the boundary value problem \eqref{FFF} due to the arbitrariness of $\varphi$.
\end{proof}

The rest of this section is to show that the function $F$ is not identically zero. We are going to show that the assumption 
$F=0$ leads to $\|u_j\|_{L^\infty(U_{r_j})}\xrightarrow{j\to\infty}0$, which contradicts the normalization \eqref{unorm}.

\subsection{Analysis near the oblique boundary}\label{ssec-oblique}

We will start with a simple estimate:
\begin{lemma}\label{lem-bd}
	Assume that $F=0$, then for any fixed bounded measurable $V\subset U$ one has $\|u_j\|_{L^2(V)}+\|\nabla u_j\|_{L^2(V)}\xrightarrow{j\to\infty}0$.
\end{lemma}

\begin{proof}
	Due to the boundedness of $V$ there is $j_0\in\NN$ with $V\subset U_{r_j}$ for all $j\ge j_0$, therefore, the respective $u_j$ are defined on $V$. Choose $0<r<R$ with
	\[
	V\subset U\cap B_r(0)=:V_r
	\]
	and  increase the above $j_0$ to have
	\[
	V_R:=U\cap B_R(0)\subset U_{r_j}
	\text{ for all }j\ge j_0.
	\]
	By Lemma~\ref{lem-bd0} one obtains $\|u_j\|^2_{L^2(V_R)}\to 0$, and Lemma~\ref{corol-grad} yields $\|\nabla u_j\|^2_{L^2(V_r)}\to 0$, and the sought claim follows from
	\[
	\|u_j\|_{L^2(V)}+\|\nabla u_j\|_{L^2(V)}\le \|u_j\|_{L^2(V_R)}+\|\nabla u_j\|_{L^2(V_r)}. \qedhere
	\]
\end{proof}

\begin{lemma}\label{lem-bound0}
	If $F=0$, then for any fixed bounded measurable $V\subset U$ one has $\|u_j\|_{L^\infty(V)}\xrightarrow{j\to\infty}0$.
\end{lemma}

\begin{proof}
During the proof, for $r>0$ denote
\[
V_r:=U\cap B_r(0).
\]

Pick $r>0$ with $V\subset V_r$, then pick $\chi\in C^\infty(\RR)$ with $0\le\chi\le 1$, $\chi(t)=1$ for $t\le r$ and $\supp\chi\subset(-\infty,r+1)$. The function
\[
\rho: \ (x,y)\mapsto \chi(\sqrt{x^2+y^2})
\]
satisfies $\partial_\nu\rho=0$ on $\partial V_{r+1}$. It follows that, for all sufficiently large $j$ (such that $r_j>r+1$), the functions $v_j:=\rho u_j$ are weak solutions of
\begin{gather*}
-\Delta v_j=f_j \text{ in } V_{r+1},\quad \partial_\nu v_j-\beta v_j=0 \text{ on }\partial V_{r+1},\\
f_j=(-\Delta\rho+\lambda_j\rho)u_j-2\nabla\rho\cdot\nabla u_j,\quad \beta=\begin{cases}
	1 & \text{ on } \{y=0\},\\
	0 & \text{ on }\partial V_{r+1}\setminus\{y=0\}.
\end{cases}
\end{gather*}
Due to the $L^2$--$L^\infty$ estimate (Proposition~\ref{prop-nittka}) one finds $C_1>0$ such that for any $j$ one has
\[
\|v_j\|_{L^\infty(V_{r+1})}\le C_1\big( \|v_j\|_{L^2(V_{r+1})}+\|f_j\|_{L^2(V_{r+1})}\big).
\]
With a suitable $C_2\ge 0$ one estimates for all sufficiently large $j$:
\[
\|v_j\|_{L^2(V_{r+1})}\le \|u_j\|_{L^2(V_{r+1})},\quad
\|f_j\|_{L^2(V_{r+1})}\le C_2\big( \|u_j\|_{L^2(V_{r+1})}+\|\nabla u_j\|_{L^2(V_{r+1})}\big).
\]
Furthermore, for a suitable $C_3>0$ and all sufficiently large $j$ we have, by Lemma~\ref{corol-grad},
\[
\|\nabla u_j\|_{L^2(V_{r+1})}\le C_3 \|u_j\|_{L^2(V_{r+2})}.
\]
By combining these inequalities, for all sufficiently large $j$ we arrive at
\begin{align*}
\|f_j\|_{L^2(V_{r+1})}&\le C_2\big( \|u_j\|_{L^2(V_{r+1})}+C_3\|u_j\|_{L^2(V_{r+2})}\big)\\
&\le C_2(C_3+1)\|u_j\|_{L^2(V_{r+2})},\\
\|v_j\|_{L^\infty(V_{r+1})}&\le C_1\big( \|u_j\|_{L^2(V_{r+1})}+C_2(C_3+1)\|u_j\|_{L^2(V_{r+2})}\big)\\
&\le C_1\big(1+C_2(C_3+1)\big)\|u_j\|_{L^2(V_{r+2})}\xrightarrow[\text{Lem.~\ref{lem-bd}}]{j\to\infty}0,
\end{align*}
and it remains to note that $\|u_j\|_{L^\infty(V)}\le \|v_j\|_{L^\infty(V_{r+1})}$.
\end{proof}

\begin{lemma}[Agmon estimate in $y$]
	There exist $\gamma>0$, $C>0$ and $j_0\in\NN$ such that for any $j\ge j_0$ one has
	\begin{equation}
		\label{eq-agmon1}
	\iint_{U_{r_j}} \Big(\dfrac{e^{\gamma y}}{1+\gamma x}\Big)^2\Big(\big|u_j(x,y)\big|^2+\big|\nabla u_j(x,y)\big|^2	\Big)\,\dxy\le C.
	\end{equation}
\end{lemma}

\begin{proof}
For $\gamma>0$ to be chosen later consider the function
\[
\eta:\ (x,y)\mapsto \dfrac{e^{\gamma y}}{1+\gamma x},
\]	
which is smooth on $U_{r_j}$. Recall that for any $\varphi\in H^1(U_{r_j})$ one has
	\[
		\iint_{U_{r_j}}\nabla u_j\cdot\nabla\varphi\,\dxy
		-\int_{\Sigma_{r_j}} u_j\varphi\,\dd s=
		\lambda_j\iint_{U_{r_j}}u_j\varphi\,\dxy.
	\]
In particular, for $\varphi:=\eta^2 u_j$ with the help of \eqref{etau} one obtains
\begin{multline}
	\label{eq-star1}
\iint_{U_{r_j}}\big|\nabla (\eta u_j)\big|^2\,\dxy
-\int_{\Sigma_{r_j}} (\eta u_j)^2\,\dd s\\
=
\lambda_j\iint_{U_{r_j}}(\eta u_j)^2\,\dxy
+\iint_{U_{r_j}}|\nabla \eta|^2 u_j^2\,\dxy.
\end{multline}
We have, for any $(x,y)\in U_{r_j}$,
\begin{align}
\nabla\eta(x,y)&=\Big(-\dfrac{\gamma e^{\gamma y}}{(1+\gamma x)^2},\dfrac{\gamma e^{\gamma y}}{1+\gamma x}\Big)=\gamma \eta(x,y) \Big(-\dfrac{1}{1+\gamma x},1\Big), \nonumber\\
\big|\nabla\eta(x,y)\big|^2&= \gamma^2 \eta(x,y)^2 \Big(\dfrac{1}{(1+\gamma x)^2}+1\Big)\le 2\gamma^2\eta(x,y)^2,\label{etanabla}
\end{align}
and the substitution into \eqref{eq-star1} yields
\begin{align*}
\iint_{U_{r_j}}\big|\nabla (\eta u_j)\big|^2\,\dxy
+(-\lambda_j-2\gamma^2)&\iint_{U_{r_j}}(\eta u_j)^2\,\dxy\le\int_{\Sigma_{r_j}} (\eta u_j)^2\,\dd s\\
&\le\int_{\Sigma_{r_j}} \dfrac{1}{(1+\gamma x)^2}\,\dd s\le \int_0^\infty \dfrac{\dd x}{(1+\gamma x)^2}=\dfrac{1}{\gamma}.
\end{align*}

Now we pick $\gamma\in(0,\frac{1}{\sqrt{2}})$ and choose any $\delta\in(0,1-2\gamma^2)$, then due to $\lambda_j\to-1$ one can find $j_0\in\NN$ such that $-\lambda_j-2\gamma^2\ge \delta$ for all $j\ge j_0$, and the preceding inequality gives
\begin{equation}
	\label{etauj}
\iint_{U_{r_j}}\big|\nabla (\eta u_j)\big|^2\,\dxy
+\delta \iint_{U_{r_j}}(\eta u_j)^2\,\dxy\le \dfrac{1}{\gamma} \text{ for } j\ge j_0.
\end{equation}
In particular, 
\begin{equation}
	\label{eq-umu1}
\iint_{U_{r_j}}(\eta u_j)^2\,\dxy\le \dfrac{1}{\delta \gamma} \text{ for }j\ge j_0.
\end{equation}

Furthermore, $\eta \nabla u_j=\nabla(\eta u_j)-u_j \nabla \eta$, hence
\begin{gather*}
|\eta \nabla u_j|^2\le 2\big|\nabla(\eta u_j)\big|^2+2 |\nabla \eta|^2u_j^2\stackrel{\eqref{etanabla}}{\le}
2\big|\nabla(\eta u_j)\big|^2+4 \gamma^2 \eta^2u_j^2,
\end{gather*}
and for $j\ge j_0$ we obtain with the help of \eqref{etauj} and \eqref{eq-umu1}:
\begin{align*}
\iint_{U_{r_j}} \eta^2\big|\nabla u_j(x,y)\big|^2\,\dxy
&\le 2\iint_{U_{r_j}} \big|\nabla(\eta u_j)\big|^2\,\dxy+4\gamma^2\iint_{U_{r_j}} (\eta u_j)^2\,\dxy\\
&
\le \dfrac{2}{\gamma} + 4\gamma^2\dfrac{1}{\delta \gamma}\equiv \dfrac{2}{\gamma} + \dfrac{4\gamma}{\delta}. 
\end{align*}
By combining with \eqref{eq-umu1} we arrive at the sought estimate \eqref{eq-agmon1}
with
\[
C:=\dfrac{1}{\delta\gamma}+\dfrac{2}{\gamma} + \dfrac{4\gamma}{\delta}. \qedhere
\]
\end{proof}	

Recall that the point $A_r$ is the upper right corner of the triangle $U_r$, see \eqref{eq-ar} and Fig.~\ref{fig-usr}.

\begin{lemma}\label{lem-upper0}
	For any fixed $R>0$ one has
	$\lim\limits_{j\to\infty}\|u_j\|_{L^\infty(U_{r_j}\cap B_R(A_{r_j}))}=0$.
\end{lemma}

\begin{proof}
For $R>0$ denote
\[
V_{j,R}:=U_{r_j}\cap B_R(A_{r_j}),
\]
see Fig.~\ref{fig-vjr}.
Pick $R>0$ and a cutoff function $\chi\in C^\infty(\RR)$
with $\chi(t)=1$ for $t\le R$ and $\supp\chi\subset(-\infty,2R)$, and consider the functions
\[
\rho_j:\ (x,y)\mapsto \chi\Big( \big|(x,y)-A_{r_j}\big|\Big).
\]
Assume that $j_0$ is large enough such that $\min\{r_j,r_j\tan\theta\}>2R$ 
and that the Agmon estimate \eqref{eq-agmon1} holds with some $C>0$ and $\gamma>0$
for all $j\ge j_0$.

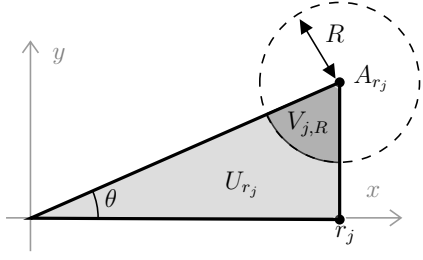
\begin{figure}
	
	\centering
	
	\tikzset{every picture/.style={line width=0.75pt}} 
	
	\begin{minipage}{0.5\textwidth}\centering
	\scalebox{0.8}{\begin{tikzpicture}[x=0.75pt,y=0.75pt,yscale=-1,xscale=1]
			
			\draw  [fill={rgb, 255:red, 155; green, 155; blue, 155 }  ,fill opacity=0.66 ] (313.84,175.81) .. controls (293.37,175.78) and (275.81,163.22) .. (268.25,145.31) -- (313.91,125.62) -- cycle ;
			\draw [color={rgb, 255:red, 155; green, 155; blue, 155 }  ,draw opacity=1 ] (104.7,210.56) -- (352.38,210.56)(120.01,100.16) -- (120.01,231.45) (345.38,205.56) -- (352.38,210.56) -- (345.38,215.56) (115.01,107.16) -- (120.01,100.16) -- (125.01,107.16)  ;
			\draw  [color={rgb, 255:red, 255; green, 255; blue, 255 }  ,draw opacity=1 ][fill={rgb, 255:red, 0; green, 0; blue, 0 }  ,fill opacity=1 ] (310.28,125.62) .. controls (310.28,123.61) and (311.91,121.98) .. (313.91,121.98) .. controls (315.92,121.98) and (317.55,123.61) .. (317.55,125.62) .. controls (317.55,127.63) and (315.92,129.26) .. (313.91,129.26) .. controls (311.91,129.26) and (310.28,127.63) .. (310.28,125.62) -- cycle ;
			\draw  [color={rgb, 255:red, 255; green, 255; blue, 255 }  ,draw opacity=1 ][fill={rgb, 255:red, 0; green, 0; blue, 0 }  ,fill opacity=1 ] (310.28,211.63) .. controls (310.28,209.62) and (311.91,208) .. (313.91,208) .. controls (315.92,208) and (317.55,209.62) .. (317.55,211.63) .. controls (317.55,213.64) and (315.92,215.27) .. (313.91,215.27) .. controls (311.91,215.27) and (310.28,213.64) .. (310.28,211.63) -- cycle ;
			\draw  [color={rgb, 255:red, 0; green, 0; blue, 0 }  ,draw opacity=1 ][fill={rgb, 255:red, 155; green, 155; blue, 155 }  ,fill opacity=0.34 ][line width=1.5]  (313.91,211.63) -- (120.06,210.78) -- (313.91,125.62) -- cycle ;
			\draw  [draw opacity=0] (159.18,194.22) .. controls (161.28,199.25) and (162.44,204.77) .. (162.44,210.56) .. controls (162.44,210.71) and (162.44,210.86) .. (162.44,211.01) -- (120.01,210.56) -- cycle ; \draw   (159.18,194.22) .. controls (161.28,199.25) and (162.44,204.77) .. (162.44,210.56) .. controls (162.44,210.71) and (162.44,210.86) .. (162.44,211.01) ;  
			\draw  [dash pattern={on 4.5pt off 4.5pt}] (263.83,125.62) .. controls (263.83,97.96) and (286.25,75.54) .. (313.91,75.54) .. controls (341.58,75.54) and (364,97.96) .. (364,125.62) .. controls (364,153.28) and (341.58,175.7) .. (313.91,175.7) .. controls (286.25,175.7) and (263.83,153.28) .. (263.83,125.62) -- cycle ;
			\draw [color={rgb, 255:red, 0; green, 0; blue, 0 }  ,draw opacity=1 ]   (290.38,88.57) -- (309.78,120.76) ;
			\draw [shift={(311.33,123.33)}, rotate = 238.92] [fill={rgb, 255:red, 0; green, 0; blue, 0 }  ,fill opacity=1 ][line width=0.08]  [draw opacity=0] (8.93,-4.29) -- (0,0) -- (8.93,4.29) -- cycle    ;
			\draw [shift={(288.83,86)}, rotate = 58.92] [fill={rgb, 255:red, 0; green, 0; blue, 0 }  ,fill opacity=1 ][line width=0.08]  [draw opacity=0] (8.93,-4.29) -- (0,0) -- (8.93,4.29) -- cycle    ;
			
			\draw (329.12,188.95) node [anchor=north west][inner sep=0.75pt]  [color={rgb, 255:red, 128; green, 128; blue, 128 }  ,opacity=1 ]  {$x$};
			\draw (132.29,102.57) node [anchor=north west][inner sep=0.75pt]  [color={rgb, 255:red, 128; green, 128; blue, 128 }  ,opacity=1 ]  {$y$};
			\draw (240.5,179.8) node [anchor=north west][inner sep=0.75pt]    {$U_{r_{j}}$};
			\draw (320.99,113.42) node [anchor=north west][inner sep=0.75pt]    {$A_{r_{j}}$};
			\draw (309.63,214.09) node [anchor=north west][inner sep=0.75pt]    {$r_{j}$};
			\draw (165.13,190.9) node [anchor=north west][inner sep=0.75pt]    {$\theta $};
			\draw (303.33,88.07) node [anchor=north west][inner sep=0.75pt]    {$R$};
			\draw (279.33,142.4) node [anchor=north west][inner sep=0.75pt]    {$V_{j,R}$};

		\end{tikzpicture}}	
	\end{minipage}\begin{minipage}{0.45\textwidth}\caption{The sets $V_{j,R}$.}\label{fig-vjr}\end{minipage}
	
\end{figure}

By construction, $\rho_j=1$ in $V_{j,R}$. Furthermore, as $\rho_j$ is radial with respect to $A_{r_j}$ and vanishes near $\big|(x,y)-A_{r_j}\big|=2R$,
one has $\partial_\nu \rho_j=0$ on $\partial V_{j,2R}$. 
It follows that the functions $v_j:=\rho_j u_j$
are weak solutions of
\begin{gather*}
-\Delta v_j =f_j \text{ in } V_{j,2R},\quad \partial_\nu v_j=0 \text{ on }\partial V_{j,2R},\\
f_j:=(-\Delta \rho_j+\lambda_j\rho_j)u_j-2\nabla\rho_j\cdot\nabla u_j.
\end{gather*}
Note that $V_{j,2R}$ with various $j$ are translations of each other and that the Laplacian is translation-invariant as well.
By the $L^2$-$L^\infty$ estimate (Proposition~\ref{prop-nittka}) one can find $C_1>0$ such that for all $j\ge j_0$
one has
\[
\|v_j\|^2_{L^\infty(V_{j,2R})}\le C_1\big(\|v_j\|^2_{L^2(V_{j,2R})}+\|f_j\|^2_{L^2(V_{j,2R})}\big).
\]
As the $L^\infty$-norms of $\rho_j,\nabla\rho_j,\Delta\rho_j$ on $V_{j,2R}$ do not depend on $j$ and $\lambda_j$
are uniformly bounded, the summands on the right hand-side can be controlled by $\|u_j\|_{L^2(V_{j,2R})}^2$ and $\|\nabla u_j\|_{L^2(V_{j,2R})}^2$ uniformly in $j\ge j_0$. 
We further note that $u_j=v_j$ in $V_{j,R}$, hence, $\|u_j\|_{L^\infty(V_{j,R})}\le\|v_j\|_{L^\infty(V_{j,2R})}$.
Therefore, with a suitable fixed $C_2>0$ and for any $j\ge j_0$ one obtains
\begin{equation}
	\label{ujinf00}
\|u_j\|^2_{L^\infty(V_{j,R})}\le C_2\big(\|u_j\|_{L^2(V_{j,2R})}^2+\|\nabla u_j\|_{L^2(V_{j,2R})}^2\big).
\end{equation}
For $(x,y)\in V_{j,2R}$ one has $x<r_j$ and $y>r_j\tan\theta-2R$, and the Agmon estimate \eqref{eq-agmon1}
implies
\begin{align*}
	\|u_j\|_{L^2(V_{j,2R})}^2&+\|\nabla u_j\|_{L^2(V_{j,2R})}^2=\,\iint_{V_{j,2R}} \big(|u_j|^2+|\nabla u_j|^2\big)\dxy\\
	&\le \Big(\dfrac{1+\gamma r_j}{e^{\gamma (r_j\tan\theta-2R)}}\Big)^2\iint_{V_{j,2R}} \Big(\dfrac{e^{\gamma y}}{1+\gamma x}\Big)^2\big(|u_j|^2+|\nabla u_j|^2\big)\dxy\\
	&\le \Big(\dfrac{1+\gamma r_j}{e^{\gamma (r_j\tan\theta-2R)}}\Big)^2\iint_{U_{r_j}} \Big(\dfrac{e^{\gamma y}}{1+\gamma x}\Big)^2\big(|u_j|^2+|\nabla u_j|^2\big)\dxy\\
	&\le Ce^{4R\gamma}(1+\gamma r_j)^2 e^{-2\gamma r_j\tan\theta},
\end{align*}
and the substitution into \eqref{ujinf00} yields
\[
\|u_j\|_{L^\infty(V_{j,R})}^2\le C_2Ce^{4R\gamma}(1+\gamma r_j)^2 e^{-2\gamma r_j\tan\theta}\xrightarrow{j\to\infty}0. \qedhere
\]
\end{proof}	

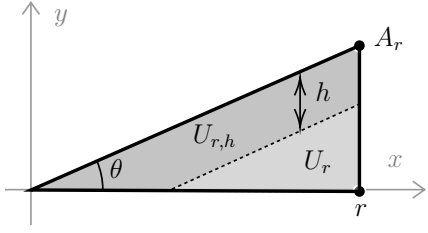
\begin{figure}
	
	\centering
	
	\tikzset{every picture/.style={line width=0.75pt}} 
	
	\begin{minipage}{0.45\textwidth}\centering
		\scalebox{0.85}{\begin{tikzpicture}[x=0.75pt,y=0.75pt,yscale=-1,xscale=1]
				
				\draw [color={rgb, 255:red, 155; green, 155; blue, 155 }  ,draw opacity=1 ] (86.7,185.56) -- (334.38,185.56)(102.01,75.16) -- (102.01,206.45) (327.38,180.56) -- (334.38,185.56) -- (327.38,190.56) (97.01,82.16) -- (102.01,75.16) -- (107.01,82.16)  ;
				\draw  [color={rgb, 255:red, 0; green, 0; blue, 0 }  ,draw opacity=1 ][fill={rgb, 255:red, 155; green, 155; blue, 155 }  ,fill opacity=0.59 ][line width=0.75]  (295.91,186.63) -- (102.06,185.78) -- (295.91,100.62) -- cycle ;
				\draw  [color={rgb, 255:red, 255; green, 255; blue, 255 }  ,draw opacity=1 ][fill={rgb, 255:red, 0; green, 0; blue, 0 }  ,fill opacity=1 ] (292.28,100.62) .. controls (292.28,98.61) and (293.91,96.98) .. (295.91,96.98) .. controls (297.92,96.98) and (299.55,98.61) .. (299.55,100.62) .. controls (299.55,102.63) and (297.92,104.26) .. (295.91,104.26) .. controls (293.91,104.26) and (292.28,102.63) .. (292.28,100.62) -- cycle ;
				\draw  [color={rgb, 255:red, 255; green, 255; blue, 255 }  ,draw opacity=1 ][fill={rgb, 255:red, 0; green, 0; blue, 0 }  ,fill opacity=1 ] (292.28,186.63) .. controls (292.28,184.62) and (293.91,183) .. (295.91,183) .. controls (297.92,183) and (299.55,184.62) .. (299.55,186.63) .. controls (299.55,188.64) and (297.92,190.27) .. (295.91,190.27) .. controls (293.91,190.27) and (292.28,188.64) .. (292.28,186.63) -- cycle ;
				\draw  [color={rgb, 255:red, 0; green, 0; blue, 0 }  ,draw opacity=1 ][fill={rgb, 255:red, 255; green, 255; blue, 255 }  ,fill opacity=0.34 ][dash pattern={on 1.5pt off 1.5pt}][line width=0.75]  (295.91,186.63) -- (182.73,186.42) -- (295.86,135.17) -- cycle ;
				\draw  [color={rgb, 255:red, 0; green, 0; blue, 0 }  ,draw opacity=1 ][line width=1.5]  (295.91,186.63) -- (102.06,185.78) -- (295.91,100.62) -- cycle ;
				\draw    (260.88,121) -- (260.88,147) ;
				\draw [shift={(260.88,149)}, rotate = 270] [color={rgb, 255:red, 0; green, 0; blue, 0 }  ][line width=0.75]    (10.93,-3.29) .. controls (6.95,-1.4) and (3.31,-0.3) .. (0,0) .. controls (3.31,0.3) and (6.95,1.4) .. (10.93,3.29)   ;
				\draw [shift={(260.88,119)}, rotate = 90] [color={rgb, 255:red, 0; green, 0; blue, 0 }  ][line width=0.75]    (10.93,-3.29) .. controls (6.95,-1.4) and (3.31,-0.3) .. (0,0) .. controls (3.31,0.3) and (6.95,1.4) .. (10.93,3.29)   ;
				\draw  [draw opacity=0] (141.18,169.22) .. controls (143.28,174.25) and (144.44,179.77) .. (144.44,185.56) .. controls (144.44,185.71) and (144.44,185.86) .. (144.44,186.01) -- (102.01,185.56) -- cycle ; \draw   (141.18,169.22) .. controls (143.28,174.25) and (144.44,179.77) .. (144.44,185.56) .. controls (144.44,185.71) and (144.44,185.86) .. (144.44,186.01) ;  
				
				\draw (311.12,163.95) node [anchor=north west][inner sep=0.75pt]  [color={rgb, 255:red, 128; green, 128; blue, 128 }  ,opacity=1 ]  {$x$};
				\draw (114.29,77.57) node [anchor=north west][inner sep=0.75pt]  [color={rgb, 255:red, 128; green, 128; blue, 128 }  ,opacity=1 ]  {$y$};
				\draw (261.5,160.8) node [anchor=north west][inner sep=0.75pt]    {$U_{r}$};
				\draw (302.99,88.42) node [anchor=north west][inner sep=0.75pt]    {$A_{r}$};
				\draw (291.63,193) node [anchor=north west][inner sep=0.75pt]    {$r$};
				\draw (196,145.8) node [anchor=north west][inner sep=0.75pt]    {$U_{r,h}$};
				\draw (268.38,120.4) node [anchor=north west][inner sep=0.75pt]    {$h$};
				\draw (147.13,165.9) node [anchor=north west][inner sep=0.75pt]    {$\theta $};

			\end{tikzpicture}
		}	
	\end{minipage}
	\begin{minipage}{0.5\textwidth}
		\caption{\small The set $U_{r,h}\subset U_r$.}\label{fig-uh}
	\end{minipage}
\end{figure}

\begin{lemma}\label{lem410}
Assume that $F=0$. For $r,h>0$ denote
	\[
	U_{r,h}:=U_r\cap\{(x,y):\ y> x\tan\theta-h\},
	\]
	see Fig.~\ref{fig-uh}. Then for any fixed $h>0$ one has
	$\lim\limits_{j\to\infty}\|u_j\|_{L^\infty(U_{r_j,h})}=0$.
\end{lemma}

\begin{proof}
Recall that the points $A_t\equiv(t,t\tan\theta)$ with $t\in[0,r]$ lie on the oblique boundary of $U_r$.
	
Let $h>0$ be fixed. For $t\in\RR$ consider the parallelograms
\[
Q_t:=\big\{(x,y): t-h<x<t+h,\, x\tan\theta-h<y<x\tan\theta\big\},
\]
and for $R>0$ consider the half-disks $H_t^R$ of radius $R$ centered at $A_t$,
\[
H^R_t:=B_R(A_t)\cap\big\{(x,y): y<x\tan\theta\big\}.
\]
As all $Q_t$ are translations of each other, we may pick $R>0$ with $Q_t\subset H^R_t$ for all $t$. See Fig.~\ref{fig-qh}.

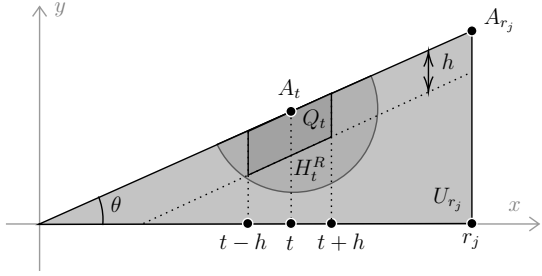
\begin{figure}
	
	\centering
	
	\tikzset{every picture/.style={line width=0.75pt}} 
	
	\begin{minipage}{0.5\textwidth}	
		\scalebox{0.7}{\begin{tikzpicture}[x=0.75pt,y=0.75pt,yscale=-1,xscale=1]
				
				\draw [color={rgb, 255:red, 155; green, 155; blue, 155 }  ,draw opacity=1 ] (106.7,240) -- (488.83,240)(130.83,84.33) -- (130.83,273.67) (481.83,235) -- (488.83,240) -- (481.83,245) (125.83,91.33) -- (130.83,84.33) -- (135.83,91.33)  ;
				\draw  [fill={rgb, 255:red, 155; green, 155; blue, 155 }  ,fill opacity=0.42 ] (368.33,133.88) .. controls (371.44,141.1) and (373.17,149.06) .. (373.17,157.41) .. controls (373.17,190.5) and (346.11,217.33) .. (312.74,217.33) .. controls (288.61,217.33) and (267.78,203.31) .. (258.1,183.03) -- cycle ;
				\draw  [color={rgb, 255:red, 0; green, 0; blue, 0 }  ,draw opacity=1 ][fill={rgb, 255:red, 155; green, 155; blue, 155 }  ,fill opacity=0.59 ][line width=0.75]  (440.67,239.67) -- (130,240.33) -- (440.67,101.67) -- cycle ;
				\draw  [color={rgb, 255:red, 0; green, 0; blue, 0 }  ,draw opacity=1 ][fill={rgb, 255:red, 155; green, 155; blue, 155 }  ,fill opacity=0.69 ] (280.33,173.67) -- (280.33,205) -- (340,177.67) -- (340,146.33) -- cycle ;
				\draw  [color={rgb, 255:red, 255; green, 255; blue, 255 }  ,draw opacity=1 ][fill={rgb, 255:red, 0; green, 0; blue, 0 }  ,fill opacity=1 ] (437.03,101.67) .. controls (437.03,99.66) and (438.66,98.03) .. (440.67,98.03) .. controls (442.68,98.03) and (444.3,99.66) .. (444.3,101.67) .. controls (444.3,103.68) and (442.68,105.3) .. (440.67,105.3) .. controls (438.66,105.3) and (437.03,103.68) .. (437.03,101.67) -- cycle ;
				\draw  [color={rgb, 255:red, 255; green, 255; blue, 255 }  ,draw opacity=1 ][fill={rgb, 255:red, 0; green, 0; blue, 0 }  ,fill opacity=1 ] (437.03,239.67) .. controls (437.03,237.66) and (438.66,236.03) .. (440.67,236.03) .. controls (442.68,236.03) and (444.3,237.66) .. (444.3,239.67) .. controls (444.3,241.68) and (442.68,243.3) .. (440.67,243.3) .. controls (438.66,243.3) and (437.03,241.68) .. (437.03,239.67) -- cycle ;
				\draw  [draw opacity=0] (173.01,220.32) .. controls (175.11,225.35) and (176.27,230.87) .. (176.27,236.66) .. controls (176.27,238.06) and (176.2,239.44) .. (176.06,240.81) -- (133.83,236.66) -- cycle ; \draw   (173.01,220.32) .. controls (175.11,225.35) and (176.27,230.87) .. (176.27,236.66) .. controls (176.27,238.06) and (176.2,239.44) .. (176.06,240.81) ;  
				\draw  [color={rgb, 255:red, 255; green, 255; blue, 255 }  ,draw opacity=1 ][fill={rgb, 255:red, 0; green, 0; blue, 0 }  ,fill opacity=1 ] (307.36,159.33) .. controls (307.36,157.32) and (308.99,155.7) .. (311,155.7) .. controls (313.01,155.7) and (314.64,157.32) .. (314.64,159.33) .. controls (314.64,161.34) and (313.01,162.97) .. (311,162.97) .. controls (308.99,162.97) and (307.36,161.34) .. (307.36,159.33) -- cycle ;
				\draw  [dash pattern={on 0.84pt off 2.51pt}]  (311,159.33) -- (311.17,240.33) ;
				\draw  [color={rgb, 255:red, 255; green, 255; blue, 255 }  ,draw opacity=1 ][fill={rgb, 255:red, 0; green, 0; blue, 0 }  ,fill opacity=1 ] (307.36,239.67) .. controls (307.36,237.66) and (308.99,236.03) .. (311,236.03) .. controls (313.01,236.03) and (314.64,237.66) .. (314.64,239.67) .. controls (314.64,241.68) and (313.01,243.3) .. (311,243.3) .. controls (308.99,243.3) and (307.36,241.68) .. (307.36,239.67) -- cycle ;
				\draw  [color={rgb, 255:red, 255; green, 255; blue, 255 }  ,draw opacity=1 ][fill={rgb, 255:red, 0; green, 0; blue, 0 }  ,fill opacity=1 ] (336.36,239.67) .. controls (336.36,237.66) and (337.99,236.03) .. (340,236.03) .. controls (342.01,236.03) and (343.64,237.66) .. (343.64,239.67) .. controls (343.64,241.68) and (342.01,243.3) .. (340,243.3) .. controls (337.99,243.3) and (336.36,241.68) .. (336.36,239.67) -- cycle ;
				\draw  [color={rgb, 255:red, 255; green, 255; blue, 255 }  ,draw opacity=1 ][fill={rgb, 255:red, 0; green, 0; blue, 0 }  ,fill opacity=1 ] (276.36,239.67) .. controls (276.36,237.66) and (277.99,236.03) .. (280,236.03) .. controls (282.01,236.03) and (283.64,237.66) .. (283.64,239.67) .. controls (283.64,241.68) and (282.01,243.3) .. (280,243.3) .. controls (277.99,243.3) and (276.36,241.68) .. (276.36,239.67) -- cycle ;
				\draw  [dash pattern={on 0.84pt off 2.51pt}]  (340,146.33) -- (340,239.67) ;
				\draw  [dash pattern={on 0.84pt off 2.51pt}]  (280.33,174) -- (280.33,236) ;
				\draw  [dash pattern={on 0.84pt off 2.51pt}]  (205.17,239.67) -- (440.33,131.33) ;
				\draw    (409.54,118.33) -- (409.5,143.33) ;
				\draw [shift={(409.5,145.33)}, rotate = 270.08] [color={rgb, 255:red, 0; green, 0; blue, 0 }  ][line width=0.75]    (10.93,-3.29) .. controls (6.95,-1.4) and (3.31,-0.3) .. (0,0) .. controls (3.31,0.3) and (6.95,1.4) .. (10.93,3.29)   ;
				\draw [shift={(409.54,116.33)}, rotate = 90.08] [color={rgb, 255:red, 0; green, 0; blue, 0 }  ][line width=0.75]    (10.93,-3.29) .. controls (6.95,-1.4) and (3.31,-0.3) .. (0,0) .. controls (3.31,0.3) and (6.95,1.4) .. (10.93,3.29)   ;
				
				\draw (465.45,222) node [anchor=north west][inner sep=0.75pt]  [color={rgb, 255:red, 128; green, 128; blue, 128 }  ,opacity=1 ]  {$x$};
				\draw (139.96,79.73) node [anchor=north west][inner sep=0.75pt]  [color={rgb, 255:red, 128; green, 128; blue, 128 }  ,opacity=1 ]  {$y$};
				\draw (411.5,212.63) node [anchor=north west][inner sep=0.75pt]    {$U_{r_j}$};
				\draw (447.66,84.25) node [anchor=north west][inner sep=0.75pt]    {$A_{r_{j}}$};
				\draw (432.29,245) node [anchor=north west][inner sep=0.75pt]    {$r_{j}$};
				\draw (180.46,219.06) node [anchor=north west][inner sep=0.75pt]    {$\theta $};
				\draw (300.17,135.25) node [anchor=north west][inner sep=0.75pt]    {$A_{t}$};
				\draw (305.96,247) node [anchor=north west][inner sep=0.75pt]    {$t$};
				\draw (328.63,245) node [anchor=north west][inner sep=0.75pt]    {$t+h$};
				\draw (258.29,245) node [anchor=north west][inner sep=0.75pt]    {$t-h$};
				\draw (317.67,156.73) node [anchor=north west][inner sep=0.75pt]    {$Q_{t}$};
				\draw (311,190.4) node [anchor=north west][inner sep=0.75pt]    {$H_{t}^{R}$};
				\draw (417.71,116.73) node [anchor=north west][inner sep=0.75pt]    {$h$};

			\end{tikzpicture}
		}\end{minipage}\begin{minipage}{0.45\textwidth}	
			\caption{\small The parallelograms $Q_t$ and the half-disks $H^R_t$.}\label{fig-qh}
			\end{minipage}
	
\end{figure}

We now choose $C,\gamma>0$ and $j_0\in\NN$ such that the Agmon estimate \eqref{eq-agmon1} is valid.
Further pick $T_0>\max\{2R,2R\cot\theta\}$ and assume additionally that $j_0$ is sufficiently large
to have $H^{2R}_t\subset U_{r_j}$ for all $j\ge j_0$ and $t\in[T_0,r_j-T_0]$.

Pick a cutoff function $\chi\in C^\infty(\RR)$
with $\chi(t)=1$ for $t\le R$ and $\supp\chi\subset(-\infty,2R)$, and for $t\in\RR$ consider the functions
\[
\rho_t:\ (x,y)\mapsto \chi\Big( \big|(x,y)-A_t\big|\Big).
\]
By construction, $\rho_t=1$ in $H_t^R$. Furthermore, as $\rho_t$ is radial with respect to $A_t$ and vanishes near $\big|(x,y)-A_t\big|=2R$, one has $\partial_\nu \rho_t=0$ on $\partial H^{2R}_t$. 
It follows that the functions $v_{j,t}:=\rho_t u_j$, with $t\in[T_0,r_j-T_0]$ and $j\ge j_0$, are weak solutions of
\begin{gather*}
	-\Delta v_{j,t} =f_{j,t} \text{ in } H_t^{2R},\quad \partial_\nu v_{j,t}=0 \text{ on }\partial H^{2R}_t,\\
	f_{j,t}:=(-\Delta \rho_t+\lambda_j\rho_t)u_j-2\nabla\rho_t\cdot\nabla u_j.
\end{gather*}
As already noted, the sets $H^{2R}_t$ with various $t$ are translations of each other, and the Laplacian is translation-invariant as well. By the $L^2$-$L^\infty$ estimate (Proposition~\ref{prop-nittka}) one can find $C_1>0$ such that for all $t\in[T_0,r_j-T_0]$ and $j\ge j_0$ one has
\[
\|v_{j,t}\|^2_{L^\infty(H^{2R}_t)}\le C_1\big(\|v_{j,t}\|^2_{L^2(H^{2R}_t)}+\|f_{j,t}\|^2_{L^2(H^{2R}_t)}\big).
\]
As the $L^\infty$-norms of $\rho_t,\nabla\rho_t,\Delta\rho_t$ on $H^{2R}_t$ do not depend on $t$ and $j$ and $\lambda_j$ are uniformly bounded, the summands on the right hand-side can be controlled by $\|u_j\|_{L^2(H^{2R}_t)}^2$ and $\|\nabla u_j\|_{L^2(H^{2R}_t)}^2$ uniformly in $j\ge j_0$. 
We further note that $u_j=v_{j,t}$ in $Q_t$, hence, $\|u_j\|_{L^\infty(Q_t)}\le\|v_{j,t}\|_{L^\infty(H^{2R}_t)}$.
Therefore, with a suitable  $C_2>0$ (which does not depend on $t$) and for any $t\in[T_0,r_j-T_0]$ and $j\ge j_0$ one obtains
\begin{equation*}
	\|u_j\|_{L^\infty(Q_t)}^2\le C_2\big(\|u_j\|_{L^2(H^{2R}_t)}^2+\|\nabla u_j\|_{L^2(H^{2R}_t)}^2\big).
\end{equation*}
For $(x,y)\in H^{2R}_t$ one has $x<t+2R$ and $y>t\tan\theta-2R$, and due to the Agmon estimate \eqref{eq-agmon1} for all $t\in[T_0,r_j-T_0]$
and $j\ge j_0$ we have
\begin{align*}
	\|u_j\|_{L^2(H^{2R}_t)}^2+&\|\nabla u_j\|_{L^2(H^{2R}_t)}^2=\iint_{H^{2R}_t} \big(|u_j|^2+|\nabla u_j|^2\big)\dxy\\
	&\le \Big(\dfrac{1+\gamma (t+2R)}{e^{\gamma (t\tan\theta-2R)}}\Big)^2 \iint_{H^{2R}_t} \Big(\dfrac{e^{\gamma y}}{1+\gamma x}\Big)^2\big(|u_j|^2+|\nabla u_j|^2\big)\dxy\\
	&\le \Big(\dfrac{1+\gamma (t+2R)}{e^{\gamma (t\tan\theta-2R)}}\Big)^2 \iint_{U_{r_j}} \Big(\dfrac{e^{\gamma y}}{1+\gamma x}\Big)^2\big(|u_j|^2+|\nabla u_j|^2\big)\dxy\\
	&\le C \Big(\dfrac{1+\gamma (t+2R)}{e^{\gamma (t\tan\theta-2R)}}\Big)^2=C e^{4\gamma R}\big(1+\gamma (t+2R)\big)^2e^{-2\gamma (\tan\theta)t}
\end{align*}
and, subsequently,
\begin{equation}
	\label{ujest3}
\|u_j\|_{L^\infty(Q_t)}^2\le C_2C e^{4\gamma R}\big(1+\gamma (t+2R)\big)^2e^{-2\gamma (\tan\theta)t},
\quad t\in[T_0,r_j-T_0],\quad j\ge j_0.
\end{equation}

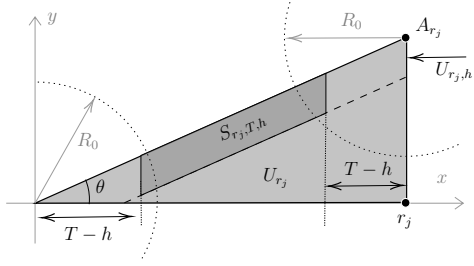
\begin{figure}
	\centering
	
	\begin{minipage}{0.45\textwidth}\centering
	\scalebox{0.6}{\begin{tikzpicture}[x=0.75pt,y=0.75pt,yscale=-1,xscale=1]
			
			\draw [color={rgb, 255:red, 155; green, 155; blue, 155 }  ,draw opacity=1 ]   (460.67,112.58) -- (365,112.99) ;
			\draw [shift={(363,113)}, rotate = 359.75] [color={rgb, 255:red, 155; green, 155; blue, 155 }  ,draw opacity=1 ][line width=0.75]    (10.93,-3.29) .. controls (6.95,-1.4) and (3.31,-0.3) .. (0,0) .. controls (3.31,0.3) and (6.95,1.4) .. (10.93,3.29)   ;
			\draw [color={rgb, 255:red, 155; green, 155; blue, 155 }  ,draw opacity=1 ] (126.7,250.91) -- (508.83,250.91)(150.83,95.24) -- (150.83,284.58) (501.83,245.91) -- (508.83,250.91) -- (501.83,255.91) (145.83,102.24) -- (150.83,95.24) -- (155.83,102.24)  ;
			\draw  [color={rgb, 255:red, 0; green, 0; blue, 0 }  ,draw opacity=1 ][fill={rgb, 255:red, 155; green, 155; blue, 155 }  ,fill opacity=0.59 ][line width=0.75]  (461.5,250.24) -- (150.83,250.91) -- (461.5,112.24) -- cycle ;
			\draw  [color={rgb, 255:red, 255; green, 255; blue, 255 }  ,draw opacity=1 ][fill={rgb, 255:red, 0; green, 0; blue, 0 }  ,fill opacity=1 ] (457.03,112.58) .. controls (457.03,110.57) and (458.66,108.94) .. (460.67,108.94) .. controls (462.68,108.94) and (464.3,110.57) .. (464.3,112.58) .. controls (464.3,114.59) and (462.68,116.21) .. (460.67,116.21) .. controls (458.66,116.21) and (457.03,114.59) .. (457.03,112.58) -- cycle ;
			\draw  [color={rgb, 255:red, 255; green, 255; blue, 255 }  ,draw opacity=1 ][fill={rgb, 255:red, 0; green, 0; blue, 0 }  ,fill opacity=1 ] (457.03,250.58) .. controls (457.03,248.57) and (458.66,246.94) .. (460.67,246.94) .. controls (462.68,246.94) and (464.3,248.57) .. (464.3,250.58) .. controls (464.3,252.59) and (462.68,254.21) .. (460.67,254.21) .. controls (458.66,254.21) and (457.03,252.59) .. (457.03,250.58) -- cycle ;
			\draw  [draw opacity=0] (193.01,231.23) .. controls (195.11,236.26) and (196.27,241.78) .. (196.27,247.57) .. controls (196.27,248.97) and (196.2,250.36) .. (196.06,251.72) -- (153.83,247.57) -- cycle ; \draw   (193.01,231.23) .. controls (195.11,236.26) and (196.27,241.78) .. (196.27,247.57) .. controls (196.27,248.97) and (196.2,250.36) .. (196.06,251.72) ;  
			\draw  [dash pattern={on 0.75pt off 0.75pt}]  (394,174.85) -- (393,270) ;
			\draw [color={rgb, 255:red, 0; green, 0; blue, 0 }  ,draw opacity=1 ] [dash pattern={on 4.5pt off 4.5pt}]  (225.5,250.58) -- (416.25,165.48) -- (460.67,145.67) ;
			\draw    (398,236.65) -- (456.67,236.02) ;
			\draw [shift={(458.67,236)}, rotate = 179.39] [color={rgb, 255:red, 0; green, 0; blue, 0 }  ][line width=0.75]    (10.93,-3.29) .. controls (6.95,-1.4) and (3.31,-0.3) .. (0,0) .. controls (3.31,0.3) and (6.95,1.4) .. (10.93,3.29)   ;
			\draw [shift={(396,236.67)}, rotate = 359.39] [color={rgb, 255:red, 0; green, 0; blue, 0 }  ][line width=0.75]    (10.93,-3.29) .. controls (6.95,-1.4) and (3.31,-0.3) .. (0,0) .. controls (3.31,0.3) and (6.95,1.4) .. (10.93,3.29)   ;
			\draw  [dash pattern={on 0.75pt off 0.75pt}]  (239.33,232.33) -- (239.67,269.33) ;
			\draw    (158.33,262.97) -- (229.67,262) ;
			\draw [shift={(231.67,261.97)}, rotate = 179.22] [color={rgb, 255:red, 0; green, 0; blue, 0 }  ][line width=0.75]    (10.93,-3.29) .. controls (6.95,-1.4) and (3.31,-0.3) .. (0,0) .. controls (3.31,0.3) and (6.95,1.4) .. (10.93,3.29)   ;
			\draw [shift={(156.33,263)}, rotate = 359.22] [color={rgb, 255:red, 0; green, 0; blue, 0 }  ][line width=0.75]    (10.93,-3.29) .. controls (6.95,-1.4) and (3.31,-0.3) .. (0,0) .. controls (3.31,0.3) and (6.95,1.4) .. (10.93,3.29)   ;
			\draw  [color={rgb, 255:red, 0; green, 0; blue, 0 }  ,draw opacity=1 ][fill={rgb, 255:red, 155; green, 155; blue, 155 }  ,fill opacity=0.69 ] (239.33,211) -- (239.33,244.33) -- (394,175.67) -- (394,142.33) -- cycle ;
			\draw    (511.33,128.67) -- (466.67,129.3) ;
			\draw [shift={(464.67,129.33)}, rotate = 359.18] [color={rgb, 255:red, 0; green, 0; blue, 0 }  ][line width=0.75]    (10.93,-3.29) .. controls (6.95,-1.4) and (3.31,-0.3) .. (0,0) .. controls (3.31,0.3) and (6.95,1.4) .. (10.93,3.29)   ;
			\draw  [draw opacity=0][dash pattern={on 0.84pt off 2.51pt}] (153.05,149.68) .. controls (208.38,150.4) and (253,194.52) .. (253,248.83) .. controls (253,266.3) and (248.38,282.71) .. (240.28,296.97) -- (151.67,248.83) -- cycle ; \draw  [dash pattern={on 0.84pt off 2.51pt}] (153.05,149.68) .. controls (208.38,150.4) and (253,194.52) .. (253,248.83) .. controls (253,266.3) and (248.38,282.71) .. (240.28,296.97) ;  
			\draw [color={rgb, 255:red, 155; green, 155; blue, 155 }  ,draw opacity=1 ]   (150.83,250.91) -- (199.02,165.74) ;
			\draw [shift={(200,164)}, rotate = 119.5] [color={rgb, 255:red, 155; green, 155; blue, 155 }  ,draw opacity=1 ][line width=0.75]    (10.93,-3.29) .. controls (6.95,-1.4) and (3.31,-0.3) .. (0,0) .. controls (3.31,0.3) and (6.95,1.4) .. (10.93,3.29)   ;
			\draw  [draw opacity=0][dash pattern={on 0.84pt off 2.51pt}] (504.89,201.83) .. controls (491.53,208.18) and (476.52,211.74) .. (460.67,211.74) .. controls (404.7,211.74) and (359.33,167.35) .. (359.33,112.58) .. controls (359.33,101.56) and (361.17,90.96) .. (364.56,81.05) -- (460.67,112.58) -- cycle ; \draw  [dash pattern={on 0.84pt off 2.51pt}] (504.89,201.83) .. controls (491.53,208.18) and (476.52,211.74) .. (460.67,211.74) .. controls (404.7,211.74) and (359.33,167.35) .. (359.33,112.58) .. controls (359.33,101.56) and (361.17,90.96) .. (364.56,81.05) ;  
			
			\draw (485.45,226.02) node [anchor=north west][inner sep=0.75pt]  [color={rgb, 255:red, 128; green, 128; blue, 128 }  ,opacity=1 ]  {$x$};
			\draw (159.96,90.64) node [anchor=north west][inner sep=0.75pt]  [color={rgb, 255:red, 128; green, 128; blue, 128 }  ,opacity=1 ]  {$y$};
			\draw (340.17,220.54) node [anchor=north west][inner sep=0.75pt]    {$U_{r_{j}}$};
			\draw (467.66,95.16) node [anchor=north west][inner sep=0.75pt]    {$A_{r_{j}}$};
			\draw (452.29,259) node [anchor=north west][inner sep=0.75pt]    {$r_{j}$};
			\draw (200.46,229.97) node [anchor=north west][inner sep=0.75pt]    {$\theta $};
			\draw (408.29,215.16) node [anchor=north west][inner sep=0.75pt]    {$T-h$};
			\draw (174.12,269) node [anchor=north west][inner sep=0.75pt]    {$T-h$};
			\draw (300.92,187.99) node [anchor=north west][inner sep=0.75pt]  [rotate=-337.25]  {$S_{r_{j} ,T,h}$};
			\draw (484.5,131.87) node [anchor=north west][inner sep=0.75pt]    {$U_{r_{j} ,h}$};
			\draw (184.79,190.69) node [anchor=north west][inner sep=0.75pt]  [color={rgb, 255:red, 128; green, 128; blue, 128 }  ,opacity=1 ]  {$R_{0}$};
			\draw (405.12,90.36) node [anchor=north west][inner sep=0.75pt]  [color={rgb, 255:red, 128; green, 128; blue, 128 }  ,opacity=1 ]  {$R_{0}$};

		\end{tikzpicture}
	}\end{minipage}
	\begin{minipage}{0.5\textwidth}
	\caption{\small Covering $U_{r_j,h}$ by $S_{r_j,T,h}$ and two large disks.}\label{fig-sth}
	\end{minipage}
\end{figure}

Let $\eps>0$. By \eqref{ujest3} we can choose $T>T_0$ and $j_0$ sufficiently large
to have 
\[
\|u_j\|_{L^\infty(Q_t)}\le \eps \text{ for all } t\in[T,r_j-T]\subset[T_0,r_j-T_0],\ j\ge j_0.
\]
Hence, if one denotes
\begin{align*}
	S_{r_j,T,h}&:=\bigcup_{t\in [T,r_j-T]} Q_t\equiv U_{r_j,h}\cap\big\{T-h<x<r_j-(T-h)\big\},
\end{align*}
then
\begin{equation}
	\label{ujs}
\|u_j\|_{L^\infty(S_{r_j,T,h})}\le \eps,\quad j\ge j_0.
\end{equation}
 In addition,
\[
U_{r_j,h}\setminus S_{r_j,T,h}=\big(U_{r_j,h}\cap\{x\le T-h\}\big)\,\cup\,\Big(U_{r_j,h}\cap\big\{x\ge r_j-(T-h)\big\}\Big),
\]
and for a large $R_0>0$ and all $j\ge j_0$ we have the inclusions
\begin{align*}
U_{r_j,h}\cap\{x\le T-h\}&\subset U_{r_j}\cap B_{R_0}(0),\\
U_{r_j,h}\cap\big\{x\ge r_j-(T-h)\big\}&\subset U_{r_j}\cap B_{R_0}(A_{r_j}),
\end{align*}
which gives
\begin{equation}
	\label{incl11}
U_{r_j,h}\subset \big(U_{r_j}\cap B_{R_0}(0)\big)\cup S_{r_j,T,h}\cup \big(U_{r_j}\cap B_{R_0}(A_{r_j})\big) \text{ for all }j\ge j_0,
\end{equation}
see Fig.~\ref{fig-sth}.
By Lemma \ref{lem-bound0} and Lemma~\ref{lem-upper0} we can find $j_1\ge j_0$ to have additionally
\[
\|u_j\|_{L^\infty(U_{r_j}\cap B_{R_0}(0))}\le\eps,\quad \|u_j\|_{L^\infty(U_{r_j}\cap B_{R_0}(A_{r_j}))}\le\eps \quad\text{ for all }j\ge j_1,
\]
and in view of the bound~\eqref{ujs} and the inclusion \eqref{incl11} we arrive at $\|u_j\|_{L^\infty(U_{r_j,h})}\le\eps$ for all $j\ge j_1$.
\end{proof}

\subsection{Reduction to a problem on an infinite half-strip}

The aim of this subsection is to reduce the initial Robin-Neumann problem on the triangles $U_{r_j}$
to a Robin-Neumann problem on the infinite half-strips
\[
\Pi_j:=(0,r_j)\times(0,\infty).
\]
This will allow one to use the separation of variables.

For the subsequent constructions let  us pick any $h>0$
and a function $\chi\in C^\infty(\RR)$ with $0\le\chi\le 1$ and such that
\[
\chi(t)=1 \text{ for }t\le -h,\quad \supp \chi\subset(-\infty, 0).
\]
For all sufficiently large $j$ the smooth function
\[
\rho_j:\ (x,y)\mapsto \chi(y-x\tan\theta)\chi\Big(\dfrac{h\cot\theta-x}{\cot \theta}\Big)\chi(y-r_j\tan\theta+h)
\]
satisfies $\partial_\nu \rho_j=0$ on $\partial U_{r_j}$, and for $(x,y)\in U_{r_j}$ one has
\begin{align*}
\rho_j(x,y)&=1 \text{ for }y<x\tan\theta-2h,\\
\rho_j(x,y)&=0 \text{ for $(x,y)$ in a neighborhood of }\partial U_{r_j}\cap\{y=x\tan\theta\}.
\end{align*}

Consider the functions
\begin{equation}
	\label{eq-vj}
	v_j:=\rho_j u_j \text{ extended by zero to }\Pi_j.
\end{equation}
By construction we have
\[
v_j=u_j \text{ in }U_{r_j}\setminus U_{r_j,2h}.
\]
Moreover, the functions $v_j$
are weak solutions of
\begin{gather}
\label{eq-vj1}
	-\Delta v_j=\lambda_j v_j+f_j\text{ in }\Pi_j, \quad f_j:=-(\Delta \rho_j)u_j-2\nabla\rho_j\cdot\nabla u_j,\\
\label{eq-vj2}	
	\underbrace{-\partial_y v_j}_{=\partial_\nu v_j}=v_j \text{ on }\{y=0\},\quad v_j=0 \text{ on }\{x=0\},\quad \underbrace{\partial_x v_j}_{=\partial_\nu v_j}=0 \text{ on }\{x=r_j\}.
\end{gather}	
(with $f_j$ extended by zero to $\Pi_j$ as well), meaning that
\begin{equation}
	\label{vj-weak}
\begin{aligned}
	v_j&\in \big\{v\in H^1(\Pi_j):\ v(0,\cdot)=0\big\}=:\Tilde H^1(\Pi_j),\\
	\iint_{\Pi_j} \nabla v_j\cdot \nabla\varphi\,\dxy&- \int_0^{r_j} v_j(x,0)\varphi(x,0)\dd x=\lambda_j\iint_{\Pi_j} v_j\varphi\,\dxy\\
	&\qquad +\iint_{\Pi_j} f_j\varphi\,\dxy \text{ for all } \varphi\in \Tilde H^1(\Pi_j).
\end{aligned}	
\end{equation}
The symbols $\Pi_j$, $\rho_j$, $v_j$ and $f_j$ will keep their meaning through the rest of Section~\ref{sec4}.

In order to explore the tensor structure of the problem let us introduce two important one-dimensional operators.

\subsubsection{Dirichlet-Neumann-Laplacian on $(0,r_j)$}\label{ssec-dn}
Denote by $D_j$ the self-adjoint operator in $L^2(0,r_j)$
acting as $g\mapsto-g''$ on $g\in H^2(0,r_j)$ with the Dirichlet-Neumann boundary condition $g(0)=g'(r_j)=0$.
Its bilinear form $d_j$ is
\begin{equation}
	\label{eq-star2}
\begin{aligned}d_j(g,g)&=\big\|\sqrt{D_j}g\big\|^2_{L^2(0,r_j)}=\int_0^{r_j} g'(x)^2\dd x,\\
\dom d_j=\dom\sqrt{D_j}&=\big\{g\in H^1(0,r_j):\ g(0)=0\big\}=:\Tilde H^1(0,r_j),
\end{aligned}
\end{equation}
and the spectrum is
\[
\spec D_j=\Big\{\dfrac{\pi^2}{r_j^2}\Big(n-\dfrac{1}{2}\Big)^2,\quad n\in\NN\Big\},
\]
while each element of $\spec D_j$ is a simple eigenvalue of $D_j$.

Further note that for any  $\lambda\in\CC\setminus\spec D_j$ the resolvent $(D_j-\lambda)^{-1}$
is an integral operator,
\[
(D_j-\lambda)^{-1} g(x)=\int_0^{r_j} G_{j,\lambda}(x,t)g(t)\,\dd t,\quad x\in (0,r_j),
\]
with $G_{j,\lambda}$ being the associated Green function,
\begin{equation}
	\label{eq-glj}
G_{j,\lambda}(x,t)=
\dfrac{ \sin\big(\sqrt{\lambda}\min\{x,t\}\big)\cos\Big(\sqrt{\lambda}\big(r_j-\max\{x,t\}\big)\Big)}{\sqrt{\lambda}\cos(\sqrt{\lambda}r_j)},
\end{equation}
 where the square root branch is chosen arbitrarily and the expression is extended by continuity for $\lambda=0$.

\subsubsection{Robin Laplacian on the half-line}\label{ssec-robin}
Consider the following self-adjoint operator $Q$ in $L^2(0,\infty)$:
\[
Q:\ f\mapsto -f'',\quad \dom Q=\big\{f\in H^2(0,\infty):\ -f'(0)=f(0)\big\}.
\]
One easily checks that
$\spec Q=\{-1\}\cup[0,\infty)$,
while $-1$ is a simple eigenvalue with normalized eigenfunction
\[
\psi:\  y\mapsto \sqrt{2}e^{-y}.
\]
We denote by $p$ the orthogonal projector on $\RR\psi$ in $L^2(0,\infty)$
and set
\[
p_\perp:=I-p, \quad L^2_\perp(0,\infty):= p_\perp \big( L^2(0,\infty)\big).
\]
Then the operator 
\[
Q_\perp:= Q \text{ viewed as an operator in }L^2_\perp(0,\infty),\quad
\dom Q_\perp:=\dom Q\cap L^2_\perp(0,\infty),
\]
is non-negative due to the spectral theorem.

Further recall that the bilinear form $q$ for $Q$ is given by
\[
q(f,f)=\int_0^\infty f'(y)^2\,\dd y-f(0)^2,\quad \dom q= \dom\sqrt{|Q|}=H^1(0,\infty).
\]
So if we denote 
\[
H^1_\perp(0,\infty):=H^1(0,\infty)\cap L^2_\perp(0,\infty),
\]
then $\dom \sqrt{Q_\perp}=H^1_\perp(0,\infty)$ and
\begin{gather*}
	\big\|\sqrt{Q_\perp}f\big\|^2_{L^2_\perp(0,\infty)}=\int_0^\infty f'(y)^2\,\dd y-f(0)^2\ge0,\quad f\in 	\dom \sqrt{Q_\perp}.
\end{gather*}

For each $f\in H^1(0,\infty)$ we have, using $|2ab|\le 2t^{-1}a^2+2tb^2$ with $t=4$,
\[
f(0)^2=\Big| \int_0^\infty 2f(y)f'(y)\,\dd y\Big|\le \dfrac{1}{2}\int_0^\infty f'(y)^2\dd y+8\int_0^\infty f(y)^2\dd y,
\]
hence, 
\begin{equation}
	\label{q-ineq1}
	2\big\|\sqrt{Q_\perp} f\big\|^2_{L^2(0,\infty)}+16\|f\|^2_{L^2(0,\infty)}
	\ge \int_0^\infty f'(y)^2\dd y,
	\quad f\in \dom \sqrt{Q_\perp}.
\end{equation}

\subsubsection{Tensor representation of the problem for $v_j$}\label{item443}
We will use the identification
\[
L^2(\Pi_j)\simeq L^2(0,r_j)\otimes L^2(0,\infty).
\]
Noting that the space $\Tilde H^1(\Pi_j)$ in \eqref{vj-weak} is the form domain of $D_j\otimes I+I\otimes Q$, one rewrites
the weak formulation \eqref{vj-weak} as
\[
(d_j\otimes I + I\otimes q)(v_j,\varphi)=\langle \lambda_j v_j +f_j,\varphi\rangle_{L^2(\Pi_j)},\quad \varphi\in \dom (d_j\otimes I + I\otimes q),
\]
which yields
\begin{equation}
	\label{vj-tensor1}
v_j\in\dom H_j,\quad (H_j-\lambda_j)v_j =f_j,\quad H_j:=D_j\otimes I+I\otimes Q.
\end{equation}

Note that
\[
\inf \spec H_j=\inf\spec D_j+\inf\spec Q=-1+\dfrac{\pi^2}{4r_j^2},
\]
and the assumption on $\lambda_j$ is
\[
\lambda_j\le -1+\dfrac{b}{r_j^2},\quad b<\dfrac{\pi^2}{4}.
\]
Hence, in particular, $\lambda_j<\inf \spec H_j$.

In order to exploit the tensor structure \eqref{vj-tensor1} we use the orthogonal projectors
\begin{gather*}
P_j:=I\otimes p,\quad P_j^\perp :=I\otimes p_\perp,\\
P_j f:\ (x,y)\mapsto  \big\langle f(x,\cdot),\psi\big\rangle_{L^2(0,\infty)}\psi(y), \quad f\in L^2(\Pi_j).
\end{gather*}
and note that considering the problem \eqref{vj-tensor1} is equivalent to considering two separate problems
\begin{align}
	P_j (H_j-\lambda_j ) v_j&=P_j f_j, \label{prob-pp}\\
	P^\perp_j (H_j-\lambda_j) v_j&=P^\perp_j f_j. \label{prob-perp}
\end{align}

\subsection{Analysis of the right-hand side}
In order to analyze the equations~\eqref{prob-pp} and \eqref{prob-perp}, let us first establish some properties of the function $f_j$ appearing on their right-hand sides.

\begin{lemma}[Agmon estimate in $x$]\label{lem-fj}
	One can find some $\beta>0$, $C_\beta>0$ and $j_0\in\NN$ such that
	\begin{equation}
		\label{eq-agmon2}
		\iint_{\Pi_j} e^{2\beta x}\big|f_j(x,y)\big|^2\dxy\le C_\beta \text{ for all }j\ge j_0.
	\end{equation}		
	If $F=0$, then
	\begin{equation}
		\label{fj-conv0}
		\|f_j\|_{L^2(\Pi_j)}\xrightarrow{j\to\infty}0.
	\end{equation}
\end{lemma}

\begin{proof}
	First note that $(0,R)\times(0,\infty)\subset\Pi_j$ for large $j$, so $f_j$ is defined on $(0,R)\times(0,\infty)$.	From uniform bounds for $\|\Delta \rho_j\|_\infty$
	and $\|\nabla \rho_j\|_\infty$ in \eqref{eq-vj1} one sees that the functions $f_j$ satisfy pointwise
	\begin{equation}
		\label{temp00}
		|f_j|^2\le C_0\big(|u_j|^2+|\nabla u_j|^2\big)
	\end{equation}
	with $C_0>0$ independent of $j$. Further note that $f_j$ is supported in the oblique strip
	\[
	S_h:=\big\{(x,y): x\tan\theta-2h<y<x\tan\theta \big\}.
	\]
	
	By choosing $C>0$, $\gamma>0$ and $j_0\ge0$ as in the Agmon estimate \eqref{eq-agmon1} we obtain for all $j\ge j_0$:
	\begin{align*}
		C&\ge \iint_{U_{r_j}} \Big(\dfrac{e^{\gamma y}}{1+\gamma x}\Big)^2\Big(\big|u_j(x,y)\big|^2+\big|\nabla u_j(x,y)\big|^2	\Big)\,\dxy\\
		&\ge \dfrac{1}{C_0}\iint_{U_{r_j}\cap S_h} \Big(\dfrac{e^{\gamma y}}{1+\gamma x}\Big)^2 \big| f_j(x,y)\big|^2\,\dxy\\
		&\ge \dfrac{1}{C_0}\iint_{U_{r_j}\cap S_h} \Big(\dfrac{e^{\gamma (x\tan\theta -2h)}}{1+\gamma x}\Big)^2 \big| f_j(x,y)\big|^2\,\dxy\\
		&=\dfrac{e^{-4\gamma h}}{C_0} \iint_{\Pi_j} \Big(\dfrac{e^{\gamma x\tan\theta}}{1+\gamma x}\Big)^2 \big| f_j(x,y)\big|^2\,\dxy.
	\end{align*}
	We can further choose a small $\beta>0$ and a large $c>0$ to have
	\[
	\dfrac{e^{\gamma x\tan\theta}}{1+\gamma x}\ge \dfrac{1}{c}e^{\beta x} \text{ for all }x>0
	\]
	which results in
	\[
	C_\beta:=CC_0c^2e^{4\gamma h}\ge \iint_{\Pi_j} e^{2\beta x} \big| f_j(x,y)\big|^2\,\dxy,
	\]
	showing~\eqref{eq-agmon2}.

	Assume now that $F=0$. For any $R>0$ the set
	\[
	V:=S_h\cap \big( (0,R)\times(0,\infty)\big)
	\]
	is bounded and contained in the sector $U$. We have then
	\begin{equation}
		\label{ff-bound}
		\|f_j\|^2_{L^2 ((0,R)\times(0,\infty))}=\|f_j\|^2_{L^2 (V)}\stackrel{\eqref{temp00}}{\le}C_0\big(\|u_j\|^2_{L^2(V)}+\|\nabla u_j\|^2_{L^2(V)}\big)\xrightarrow[\text{Lem.~\ref{lem-bd}}]{j\to\infty}0.
	\end{equation}
	Furthermore, for $j\ge j_0$ we obtain
	\begin{equation}
		\label{fj00}
		\begin{aligned}
			\|f_j\|^2_{L^2(\Pi_j)}&=\|f_j\|^2_{L^2 ((0,R)\times(0,\infty))}+\|f_j\|^2_{L^2 ((R,r_j)\times(0,\infty))}\\
			&\le \|f_j\|^2_{L^2 ((0,R)\times(0,\infty))}+e^{-2\beta R}\iint_{(R,r_j)\times(0,\infty)}e^{2\beta x}\big|f_j(x,y)\big|^2\dxy\\
			&\le \|f_j\|^2_{L^2 ((0,R)\times(0,\infty))}+e^{-2\beta R}\iint_{\Pi_j}e^{2\beta x}\big|f_j(x,y)\big|^2\dxy\\
			&
			\stackrel{\eqref{eq-agmon2}}{\le}\|f_j\|^2_{L^2 ((0,R)\times(0,\infty))}
			+C_\beta e^{-2\beta R}.
		\end{aligned}
	\end{equation}
	Let $\eps>0$. We first choose $R>0$ sufficiently large to have $C_\beta e^{-2\beta R}<\eps$. We then use \eqref{ff-bound} and increase $j_0$ to obtain $\|f_j\|^2_{L^2 ((0,R)\times(0,\infty))}<\eps$ for all $j\ge j_0$.
	One arrives at $\|f_j\|^2_{L^2(\Pi_j)}<2\eps$ for all $j\ge j_0$ by \eqref{fj00}, and \eqref{fj-conv0} follows due to the arbitrariness of $\eps$.
\end{proof}

\subsection{Analysis of the projection on the ground state}\label{item444}

Let us consider first the problem \eqref{prob-pp} in $\ran P_j$. We have
\begin{equation*}
\begin{aligned}
	P_j v_j&=g_j\otimes\psi, &P_j f_j&=h_j\otimes\psi,\\
g_j&:\  x\mapsto \big\langle v_j(x,\cdot),\psi\big\rangle_{L^2(0,\infty)}, &
h_j&:\  x\mapsto \big\langle f_j(x,\cdot),\psi\big\rangle_{L^2(0,\infty)}.
\end{aligned}	
\end{equation*}

As $p$ is the orthogonal projector on $\ker(Q+1)$, we have
\begin{align*}
	P_j (H_j-\lambda_j)v_j&= (H_j-\lambda_j) P_j v_j\\
	&=\big((D_j-(\lambda_j+1)I\big)\otimes I\big)P_j v_j=\big[\big(D_j -(\lambda_j+1)\big)g_j\big]\otimes\psi,
\end{align*}
and the problem \eqref{prob-pp} together with the assumption \eqref{lest2}
takes the simple form
\begin{equation}
	\label{eq-tjmj}
	(D_j-\mu_j)g_j=h_j,\quad \mu_j:=\lambda_j+1\le \dfrac{b}{r_j^2},\quad b<\dfrac{\pi^2}{4}.
\end{equation}

\begin{lemma}\label{lem-gj}
	Assume that  $\lambda_j$ satisfy \eqref{lest2} and $F=0$, then
	$\|g_j\|_{L^\infty(0,r_j)}\xrightarrow{j\to\infty}0$.
\end{lemma}

\begin{proof}
As discussed in \S~\ref{ssec-dn}, the solution $g_j$ of \eqref{eq-tjmj} is given by
\begin{equation}
	\label{eq-green1}
g_j(x)=\int_{0}^{r_j} G_{j,\mu_j}(x,t)h_j(t)\,\dd t,
\end{equation}
with the Green function $G_{j,\mu_j}$ from~\eqref{eq-glj}, which we rewrite 
as
\begin{gather*}
	G_{j,\mu_j}(x,t)=\dfrac{1}{\cos(\sqrt{\mu_j}r_j)}\dfrac{\sin (z_1)}{z_1}\cos (z_2)\min\{x,t\},\\
	z_1:=\sqrt{\mu_j}\min\{x,t\},\quad
	z_2:=\sqrt{\mu_j}\big(r_j-\max\{x,t\}\big).
\end{gather*}
Recall that we have the uniform bound
\[
-a\le r_j^2\mu_j\le b<\dfrac{\pi^2}{4},
\]
see \eqref{eq-lambdaj}. On the one hand, this gives the uniform bounds
\[
|z_j|\le \max\{\sqrt{|a|},\sqrt{b}\},\quad j\in\{1,2\}.
\]
At the same time, the value $\sqrt{\mu_j}r_j$ remains at a fixed positive distance from the zeros of $\lambda\mapsto \cos \lambda$, and one has
\[
\sup_j \dfrac{1}{\cos(\sqrt{\mu_j}r_j)}<\infty.
\]
From these observations we conclude that for some $C>0$ and all $j$ we have
\[
\big|G_{j,\mu_j}(x,t)\big|\le C\min\{x,t\}\le Ct,\quad x,t\in(0,r_j),
\]
and the substitution into \eqref{eq-green1} yields
\begin{equation}
	\label{eq-green2}	
	\|g_j\|_{L^\infty(0,r_j)}\le C \int_0^{r_j} t \big|h_j(t)\big|\,\dd t.
\end{equation}

We have
\begin{equation}
	\label{hjnorm}
\begin{aligned}
\big|h_j(t)\big|^2&=\Big|\big\langle f_j(t,\cdot),\psi\big\rangle_{L^2(0,\infty)}\Big|^2\\
&\le \big\|f_j(t,\cdot)\big\|^2_{L^2(0,\infty)}\underbrace{\|\psi\|^2_{L^2(0,\infty)}}_{=1}=\int_0^\infty \big|f_j(t,y)\big|^2\dd y.
\end{aligned}
\end{equation}
Choose $\beta>0$, $C_\beta>0$ and $j_0\in\NN$ as in the Agmon estimate \eqref{eq-agmon2} for $f_j$. For $j\ge j_0$
we have
\begin{equation}
	\label{hj00}
\begin{aligned}
	\bigg(\int_0^{r_j} t \big|h_j(t)\big|\,\dd t\bigg)^2&=\bigg(\int_0^{r_j} xe^{-\frac{\beta x}{2}} e^{\frac{\beta x}{2}}\big|h_j(x)\big|\,\dd x\bigg)^2\\
	\text{(use Cauchy-Schwarz) }&\le \int_0^{r_j} x^2e^{-\beta x}\dd x \cdot \int_0^{r_j} e^{\beta x}\big|h_j(x)\big|^2\dd x.
\end{aligned}
\end{equation}
We estimate
\begin{align*}
	\int_0^{r_j} x^2e^{-\beta x}\dd x&\le \int_0^{\infty} x^2e^{-\beta x}\dd x=:B<\infty,\\
	\bigg(\int_0^{r_j} e^{\beta x}\big|h_j(x)\big|^2\dd x\bigg)^2&=\bigg(\int_0^{r_j} e^{\beta x}\big|h_j(x)\big| \cdot \big|h_j(x)\big|\,\dd x\bigg)^2\\
	\text{(use Cauchy-Schwarz) }&\le \int_0^{r_j} e^{2\beta x}\big|h_j(x)\big|^2\dd x \cdot \int_0^{r_j} \big|h_j(x)\big|^2\,\dd x\\
	&\stackrel{\eqref{hjnorm}}{\le} \iint_{\Pi_j}e^{2\beta x}\big|f_j(x,y)\big|^2\dxy \cdot\iint_{\Pi_j}\big|f_j(x,y)\big|^2\dxy\\
	&\le C_\beta \|f_j\|^2_{L^2(\Pi_j)}.
\end{align*}
The substitution into \eqref{hj00} and the assumption $F=0$ yield
\[
\bigg(\int_0^{r_j} t \big|h_j(t)\big|\,\dd t\bigg)^2\le B\sqrt{C_\beta} \|f_j\|_{L^2(\Pi_j)}\xrightarrow[\text{by \eqref{fj-conv0}}]{j\to\infty}0,
\]
and the sought claim follows from \eqref{eq-green2}.
\end{proof}

\subsection{Analysis in the orthogonal complement of the ground state}\label{item445}

Here we study the problem \eqref{prob-perp} in $\ran P_j^\perp$. Consider the functions
\[
w_j:= P_j^\perp v_j\in L^2_\perp(\Pi_j)\equiv P_j^\perp L^2(\Pi_j)\equiv L^2(0,r_j)\otimes L_\perp^2(0,\infty).
\]
In view of the definition of $P_j$ and $P_j^\perp$ we have
\begin{equation}
	\label{eq-wj}
w_j\in \dom H_j^\perp,\quad (H_j^\perp-\lambda_j)w_j=P_j^\perp f_j=:f_j^\perp,\quad
H_j^\perp:=D_j\otimes I+I\otimes Q_\perp,
\end{equation}
with the operators $D_j\ge 0$ and $Q_\perp\ge 0$ introduced in \S\S~\ref{ssec-dn}--\ref{ssec-robin}.
We are going to show first in Lemma~\ref{lemh21} that the $H^2$-norm of $w_j$ is controlled by the $L^2$-norm of $f_j^\perp$. In view of the embedding $H^2(\Pi_j)\hookrightarrow L^\infty(\Pi_j)$ with a uniform constant (Lemma~\ref{lemh22}), the assumption $F=0$ results in the uniform convergence of $w_j$ to $0$ (Corollary~\ref{corol-wj}).

Recall that the spaces $\Tilde H^1$ were introduced in \eqref{vj-weak} and \eqref{eq-star2}.

\begin{lemma}\label{lemh21}
One has $\dom H_j^\perp\subset H^2(\Pi_j)$. Furthermore, there exists $C>0$
such that for any $j$, any $\psi_j\in L^2_\perp(\Pi_j)$
and $\varphi_j:=(H_j^\perp-\lambda_j)^{-1}\psi_j$ one has
\[
\|\varphi_j\|_{H^2(\Pi_j)}\le C\|\psi_j\|_{L^2_\perp(\Pi_j)}.
\]
	
\end{lemma}

\begin{proof}
First remark that $\varphi_j\in \Tilde H^1(\Pi_j)$, as already noted above. We have
\begin{equation}
\begin{aligned}
\Big|\big\langle (H_j^\perp-\lambda_j)\varphi_j,\varphi_j \big\rangle_{L^2_\perp(\Pi_j)}\Big|&=	\Big| \langle \psi_j,\varphi_j\rangle_{L^2_\perp(\Pi_j)}\Big|\le 
\|\psi_j\|_{L^2_\perp(\Pi_j)}\|\varphi_j\|_{L^2_\perp(\Pi_j)}.
\end{aligned}
\label{fhj}	
\end{equation}
Recall that $D_j\ge 0$ and $Q_\perp\ge 0$, hence, $H_j^\perp \ge 0$, and
\begin{align*}
	\big\langle H_j^\perp \varphi_j,\varphi_j \big\rangle_{L^2_\perp(\Pi_j)}&=\big\langle (D_j\otimes I) \varphi_j,\varphi_j \big\rangle_{L^2_\perp(\Pi_j)}+\big\langle (I\otimes Q_\perp) \varphi_j,\varphi_j \big\rangle_{L^2_\perp(\Pi_j)}\\
	&=\big\|(\sqrt{D_j}\otimes I)\varphi_j\big\|^2_{L^2_\perp(\Pi_j)}+\big\|(I\otimes\sqrt{Q_\perp})\varphi_j\big\|^2_{L^2_\perp(\Pi_j)}.
\end{align*}

Taking into account that $\lambda_j\le-\frac12$, we deduce from \eqref{fhj} that
\begin{multline}
	\big\|(\sqrt{D_j}\otimes I)\varphi_j\big\|^2_{L^2_\perp(\Pi_j)}+\big\|(I\otimes\sqrt{Q_\perp})\varphi_j\big\|^2_{L^2_\perp(\Pi_j)}\\
	+\dfrac{1}{2}\|\varphi_j\|^2_{L^2_\perp(\Pi_j)}\le \|\psi_j\|_{L^2(\Pi_j)}\|\varphi_j\|_{L^2_\perp(\Pi_j)}.
	\label{long00}
\end{multline}
As all three terms on the left-hand side are non-negative, each of them does not exceed the right-hand side. By considering the last summand we obtain
\begin{equation}
	\label{wjup}
	\|\varphi_j\|_{L^2_\perp(\Pi_j)}\le 2\|\psi_j\|_{L^2(\Pi_j)},
\end{equation}
while the first summand gives
\begin{equation}
	\label{wx1}
\begin{aligned}
	\iint_{\Pi_j} (\partial_x \varphi_j)^2\dxy&\equiv \big\|(\sqrt{D_j}\otimes I)\varphi_j\big\|^2_{L^2_\perp(\Pi_j)}\\
	&\le \|\psi_j\|_{L^2(\Pi_j)}\|\varphi_j\|_{L^2_\perp(\Pi_j)}\stackrel{\eqref{wjup}}{\le}2\|\psi_j\|^2_{L^2(\Pi_j)}.
\end{aligned}
\end{equation}	
Finally, by using the one-dimensional  inequality \eqref{q-ineq1} we obtain
\begin{align*}
	\|\partial_y\varphi_j\|^2_{L^2(\Pi_j)}&\le 2\big\|(I\otimes\sqrt{Q_\perp})\varphi_j\big\|^2_{L^2_\perp(\Pi_j)}
	+16 \|\varphi_j\|^2_{L^2_\perp(\Pi_j)}\\
\text{use \eqref{long00} and \eqref{wjup}: }	&\le 2\|\psi_j\|_{L^2(\Pi_j)}\|\varphi_j\|_{L^2_\perp(\Pi_j)}+64 \|\psi_j\|^2_{L^2_\perp(\Pi_j)}\le 68 \|\psi_j\|_{L^2(\Pi_j)}^2.
\end{align*}
By combining the last inequality with \eqref{wjup} and \eqref{wx1} we conclude that 
\begin{equation}
	\label{phih1}
\|\varphi_j\|^2_{H^1(\Pi_j)}\le C' \|\psi_j\|^2_{L^2(\Pi_j)}
\end{equation}
for all $j$ and some $C'>0$ independent of $j$.

Now note that due to the spectral theorem the operators
\begin{gather*}
S_{1,j}:=(D_j\otimes I) (H_j^\perp-\lambda_j)^{-1},\quad
S_{2,j}:=(I\otimes Q_\perp) (H_j^\perp-\lambda_j)^{-1},\\
S_{3,j}:=2(\sqrt{D_j}\otimes \sqrt{Q_\perp})(H_j^\perp-\lambda_j)^{-1}
\end{gather*}
are bounded in $L^2_\perp(\Pi_j)$, and their norms can be controlled by a $j$-independent constant.
Namely, if $D_j$ is unitarily equivalent to the operator of multiplication by some function $\tau_j\ge 0$ in $L^2(M_j,\nu_j)$
and $Q_\perp$ is unitarily equivalent to the operator of multiplication by a function $\rho\ge 0$ in $L^2(N,\varrho)$,
then the above operators $S_{\ell,j}$ are unitarily equivalent to the operators of multiplication by the functions $s_{\ell,j}$
in $L^2(M_j\times N,\nu_j\otimes\varrho)$,
\begin{gather*}
	s_{1,j}:\ (m,n)\mapsto \dfrac{\tau_j(m)}{\tau_j(m)+\rho(n)-\lambda_j},\quad
	s_{2,j}:\ (m,n)\mapsto \dfrac{\rho(n)}{\tau_j(m)+\rho(n)-\lambda_j},\\
	s_{3,j}:\ (m,n)\mapsto \dfrac{2\sqrt{\tau_j(m)\rho(n)}}{\tau_j(m)+\rho(n)-\lambda_j},
\end{gather*}
and then
\begin{equation}
	\label{sjnorm}
\|S_{\ell,j}\|_{L^2_\perp(\Pi_j)\to L^2_\perp(\Pi_j)}=\|s_{\ell,j}\|_{L^\infty(M_j\times N,\nu_j\otimes\varrho)}\le 1,\quad \ell\in\{1,2,3\}.
\end{equation}
Using similar arguments we have the inclusion
\[
\dom(D_j\otimes I+I\otimes Q_\perp)\subset \dom(D_j\otimes I)\cap\dom (I\otimes Q_\perp).
\]

Consider the subspace
\[
\cD_j:=\vspan\{  a\otimes b: a\in \dom D_j,\ b\in \dom Q_\perp\}\subset \dom(D_j\otimes I+I\otimes Q_\perp).
\]
Due to $\dom D_j\subset H^2(0,r_j)$ and $\dom Q_\perp\subset H^2(0,\infty)$ we have $\cD_j\subset H^2(\Pi_j)$.
Assume first $\psi_j\in (H_j^\perp-\lambda_j)(\cD_j)$, then $\varphi_j\in \cD_j$, and
\begin{equation}
	\label{pxx}
\begin{aligned}
\|\partial^2_{xx} \varphi_j\|^2_{L^2(\Pi_j)}&=\big\|(D_j\otimes I)\varphi_j\big\|_{L^2(\Pi_j)}^2=\big\|S_{1,j}\psi_j\big\|_{L^2(\Pi_j)}^2\stackrel{\eqref{sjnorm}}{\le} \|\psi_j\big\|_{L^2(\Pi_j)}^2,\\
\|\partial^2_{yy} \varphi_j\|^2_{L^2(\Pi_j)}&=\big\|(I\otimes Q_\perp)\varphi_j\big\|_{L^2(\Pi_j)}^2=\big\|S_{2,j}\psi_j\big\|_{L^2(\Pi_j)}^2\stackrel{\eqref{sjnorm}}{\le} \|\psi_j\big\|_{L^2(\Pi_j)}^2.
\end{aligned}
\end{equation}
Due to $\varphi_j\in H^2(\Pi_j)\cap L^2_\perp(\Pi_j)$ we have $\partial_x\varphi_j\in  H^1(\Pi_j)\cap L^2_\perp(\Pi_j)\subset\dom(I\otimes\sqrt{Q_\perp})$,
and with the help of the one-dimensional  inequality \eqref{q-ineq1} we obtain
\begin{equation}
	\label{pxy}
	\|\partial^2_{xy}\varphi_j\|^2_{L^2(\Pi_j)}\le 2\big\|(I\otimes\sqrt{Q_\perp})\partial_x\varphi_j\big\|^2_{L^2_\perp(\Pi_j)}
	+16 \|\partial_x\varphi_j\|^2_{L^2_\perp(\Pi_j)}.
\end{equation}
Due to $\varphi_j\in \cD_j\subset\dom (D_j\otimes Q_\perp)$ we may represent
\begin{gather*}
\varphi_j=\sum_{k=1}^N c_k a_k\otimes b_k,\quad c_k\in\RR, \quad a_k\in \dom D_j,\quad b_k\in \dom Q_\perp,\\
(I\otimes\sqrt{Q_\perp})\partial_x\varphi_j=\sum_{k=1}^N c_k (\partial_x a_k)\otimes \sqrt{Q_\perp}b_k=\partial_x (I\otimes\sqrt{Q_\perp})\varphi_j,
\end{gather*}
while $(I\otimes\sqrt{Q_\perp})\varphi_j\in \dom(\sqrt{D_j}\otimes I)$.
This allows one to transform
\begin{align*}
\big\|(I\otimes\sqrt{Q_\perp})\partial_x\varphi_j\big\|^2_{L^2_\perp(\Pi_j)}&=\big\|\partial_x(I\otimes\sqrt{Q_\perp})\varphi_j\big\|^2_{L^2_\perp(\Pi_j)}\\
&=\big\|(\sqrt{D_j}\otimes I)(I\otimes\sqrt{Q_\perp})\varphi_j\big\|^2_{L^2_\perp(\Pi_j)}\\
&=\big\|(\sqrt{D_j}\otimes\sqrt{Q_\perp})\varphi_j\big\|^2_{L^2_\perp(\Pi_j)}\\
&=\dfrac{1}{4}\big\|S_{3,j}\psi_j\big\|^2_{L^2_\perp(\Pi_j)}\stackrel{\eqref{sjnorm}}{\le}\dfrac{1}{4}\big\|\psi_j\big\|^2_{L^2_\perp(\Pi_j)}.
\end{align*}
The substitution into \eqref{pxy} yields
\[
\|\partial^2_{xy}\varphi_j\|^2_{L^2(\Pi_j)}\le \dfrac{1}{2}\|\psi_j\|^2_{L^2_\perp(\Pi_j)}+16 \|\partial_x\varphi_j\|^2_{L^2_\perp(\Pi_j)}\stackrel{\eqref{phih1}}{\le} \Big(\dfrac{1}{2}+16C'\Big)\|\psi_j\|^2_{L^2_\perp(\Pi_j)},
\]
and by combining with \eqref{pxx} we arrive at
\[
\|\partial^2_{xx} \varphi_j\|^2_{L^2(\Pi_j)}+\|\partial^2_{yy} \varphi_j\|^2_{L^2(\Pi_j)}+\|\partial^2_{xy} \varphi_j\|^2_{L^2(\Pi_j)}\le C''\|\psi_j\|^2_{L^2_\perp(\Pi_j)}
\]
for all $j$ and some fixed $C''>0$. Together with the estimate \eqref{phih1} for the $H^1$-norm we arrive at
\[
\|\varphi_j\|_{H^2(\Pi_j)}\le C\|\psi_j\|_{L^2_\perp(\Pi_j)},\quad \psi_j\in(H_j^\perp-\lambda_j)(\cD_j),
\]
for all $j$ with some fixed $C>0$. As $\lambda_j<0\le\inf\spec H_j^\perp$
and $H_j^\perp$  is by construction essentially self-adjoint on $\cD_j$, the inequality extends by density to all $\psi_j\in L^2_\perp(\Pi_j)$.
\end{proof}

\begin{lemma}\label{lemh22}
	There exist $C>0$ and $j_0\in\NN$ such that $\|\varphi\|_{L^\infty(\Pi_j)}\le C\|\varphi\|_{H^2(\Pi_j)}$
	for all $\varphi\in H^2(\Pi_j)$ and all $j\ge j_0$.
\end{lemma}

\begin{proof}
Denote $\Omega:=(0,\infty)\times(0,\infty)$, then there exists a constant $c_0>0$ such that 	$\|\varphi\|_{L^\infty(\Omega)}\le c_0\|\varphi\|_{H^2(\Omega)}$ due to Sobolev embedding theorem. 
Pick a large $j_0\in\NN$ to have $r_j>2$ for all $j\ge j_0$, which is assumed in all subsequent inequalities.

Further pick a function $\chi\in C^\infty(\RR)$ with $\chi(t)=1$ for $t\le 1$ and $\chi(t)=0$ for $t\ge 2$, and set $\rho(x,y):=\chi(x)$. For any $j\ge j_0$ and any $\varphi\in H^2(\Pi_j)$ one has
$\rho \varphi\in H^2(\Omega)$ and $(1-\rho)\varphi\in H^2(\Omega_j)$ for $\Omega_j:=(-\infty,r_j)\times(0,\infty)$ (with both $\rho\varphi$ and $(1-\rho)\varphi$ extended by zero),
and for some $c_1>0$ which only depends on $\chi$ one has 
\begin{equation}
	\label{rhorho1}
\|\rho \varphi\|_{H^2(\Omega)}\le c_1\|\varphi\|_{H^2(\Pi_j)},
\quad
\|(1-\rho) \varphi\|_{H^2(\Omega_j)}\le c_1\|\varphi\|_{H^2(\Pi_j)}.
\end{equation}
As each $\Omega_j$ is an isometric copy of $\Omega$, we also have for all $\varphi\in H^2(\Pi_j)$:
\begin{equation}
	\label{rhorho2}
\|\rho \varphi\|_{L^\infty(\Omega)}\le c_0\|\rho \varphi\|_{H^2(\Omega)},\quad
\|(1-\rho) \varphi\|_{L^\infty(\Omega_j)}\le c_0\|(1-\rho) \varphi\|_{H^2(\Omega_j)}.
\end{equation}
Hence, for all $\varphi\in H^2(\Pi_j)$ we obtain
\begin{align*}
	\|\varphi\|_{L^\infty(\Pi_j)}&\le\|\rho \varphi\|_{L^\infty(\Pi_j)}+\|(1-\rho)\varphi\|_{L^\infty(\Pi_j)}
	\le \|\rho \varphi\|_{L^\infty(\Omega)}+\|(1-\rho)\varphi\|_{L^\infty(\Omega_j)}\\
	&\stackrel{\eqref{rhorho2}}{\le} c_0\big(\|\rho \varphi\|_{H^2(\Omega)}+\|(1-\rho) \varphi\|_{H^2(\Omega_j)}\big)\stackrel{\eqref{rhorho1}}{\le} 2c_0c_1\|\varphi\|_{H^2(\Pi_j)}. \qedhere
\end{align*}
\end{proof}
	
\begin{corollary}\label{corol-wj}
	Assume that $F=0$, then $\|w_j\|_{L^\infty(\Pi_j)}\xrightarrow{j\to\infty}0$.
\end{corollary}	
	
\begin{proof}
We have, with a suitable fixed constant $C>0$ and all sufficiently large $j$:
\begin{align*}
	\|w_j\|_{L^\infty(\Pi_j)}&\le C\|w_j\|_{H^2(\Pi_j)} && \text{by Lemma~\ref{lemh22}}\\
	&\le C^2\big\|(H_j^\perp-\lambda_j)w_j\big\|_{L^2_\perp(\Pi_j)}&& \text{by Lemma~\ref{lemh21}}\\
	&=C^2\|P_j^\perp f_j\|_{L^2_\perp(\Pi_j)}&& \text{by \eqref{eq-wj}}\\
	&\le C^2\|f_j\|_{L^2(\Pi_j)}\xrightarrow{j\to \infty} 0 && \text{by Lemma~\ref{lem-fj}}.\qedhere
\end{align*}
\end{proof}	
	
\subsection{Finalizing the proof}\label{ssec-prop41}	

Let us collect all the preceding statements to complete the proof of Lemma~\ref{prop41}.
Assume that the limiting bounded solution $F$ of the threshold resonance problem~\eqref{FFF} constructed in Lemma~\ref{lem42}
and Lemma~\ref{lem48a} is identically zero. Pick any $h>0$, then by Lemma~\ref{lem410} one can find $j_0\in\NN$ such that
\begin{equation}
	\label{uj-bd1}
\big|u_j(x,y)\big|<\frac12 \text{ for all } j\ge j_0,\quad (x,y)\in U_{r_j}\cap\big\{(x,y):\ y> x\tan\theta-3h\big\}.
\end{equation}

Furthermore, by construction, for all $(x,y)\in U_{r_j}\cap\big\{(x,y):\ y< x\tan\theta-2h\big\}\subset \Pi_j$ and $j\ge j_0$ we have,
for $v_j$ defined in \eqref{eq-vj},
\begin{align*}
\big|u_j(x,y)\big| &= |v_j(x,y)|=\big|g_j(x)\psi(y)+w_j(x,y)\big|\\
&\le \big|g_j(x)\psi(y)\big|+\big|w_j(x,y)\big|\le \sqrt{2}\big|g_j(x)\big|+\big|w_j(x,y)\big|,
\end{align*}
with $g_j$ and $w_j$ as in Subsections~\ref{item444} and~\ref{item445} respectively.
We estimate $g_j$ by Lemma~\ref{lem-gj} and $w_j$ by Corollary~\ref{corol-wj}, which allows one to
increase  $j_0$ to have $\sqrt{2}\big|g_j(x)\big|<\frac{1}{4}$ and $|w_j(x,y)|<\frac{1}{4}$ in the preceding inequality,
which shows
\begin{equation}
	\label{uj-bd2}
\big|u_j(x,y)\big|<\frac12 \text{ for all } j\ge j_0,\quad (x,y)\in U_{r_j}\cap\big\{(x,y):\ y< x\tan\theta-2h\big\}.
\end{equation}

By combining \eqref{uj-bd1} and \eqref{uj-bd2} we obtain $\|u_j\|_{L^\infty(U_{r_j})}\le \frac{1}{2}$ for all $j\ge j_0$, which contradicts the initial normalization \eqref{unorm} of $u_j$.
Hence, $F\not\equiv 0$. Lemma~\ref{prop41} is proved.

\begin{corollary}\label{cor-last2}
	For each $\theta\in(0,\frac{\pi}{2})$ exactly one of the following statements holds true:
	\begin{itemize}
		\item[(a)] For each $c\in(0,\frac{\pi^2}{4})$ one has
		\begin{equation}
			\label{ek1}
			E_{\kappa+1}(N_r)\ge -1+\dfrac{c}{r^2} \text{ for all sufficiently large $r$;}
		\end{equation}
		in this case $\theta$ is non-critical.
		\item[(b)] There is $d>0$ such that
		\begin{equation}
			\label{ek2}
	E_{\kappa+1}(N_r)\le -1+d^{-1}e^{-dr} \text{ for all sufficiently large $r$;}
\end{equation}
in this case $\theta$ is critical.
\end{itemize}
\end{corollary}	

\begin{proof}
Assume that the inequality \eqref{ek1} in (a) is satisfied. If $\theta$ were critical, then one would
obtain \eqref{ek2} by Lemma~\ref{prop41a}, and one easily sees that \eqref{ek1} and \eqref{ek2} cannot hold simultaneously. Hence, $\theta$ is non-critical.

Now assume that the inequality \eqref{ek1} in (a) is \emph{not} satisfied. Then for some subsequence $(r_j)$ with $r_j\to\infty$ and some $b\in\big(0,\frac{\pi^2}{4}\big)$
	one has
	\[
	E_{\kappa+1}(N_{r_j})\le -1+\dfrac{b}{r_j^2},
	\]
	and one has threshold resonance by Lemma~\ref{prop41}, which shows that $\theta$ is critical by the results of Section~\ref{sec3}, and then Lemma~\ref{prop41a} shows \eqref{ek2}.
\end{proof}

\begin{remark}
	In view of the decomposition \eqref{decomp1} and the lower bound \eqref{eq-dr}, the assertions
	of Corollary~\ref{cor-last2} also hold if one replaces $N_r$ by the original Robin-Neumann Laplacian $\Tilde N^r_\theta$
	in the $r$-dependent quadrangle $OA^+_rB_rA^-_r$.
\end{remark}
	
	\section*{Statements and declarations}
\subsection*{Acknowledgments}

\addtocontents{toc}{\protect\setcounter{tocdepth}{1}}

The author thanks Michael Levitin who brought the reference \cite{kmz} to his attention and provided him with the expressions~\eqref{urs00} that correct a typo in \cite{lyal}. The figures were prepared with the online editor \url{http://mathcha.io}.

\subsection*{Use of artificial intelligence (AI)} Some of the proofs were found during a discussion with an AI system (ChatGPT).
In particular, the AI system suggested the overall proof scheme by contradiction for
Lemma~\ref{prop41} while searching for an adaptation of the paper~\cite{bakh} to the sector geometry,
and it also found the necessary elliptic estimates in~Subsection~\ref{ssec22}.
The whole text is written by the author from scratch and does not contain any automatically generated part.
The author takes
full responsibility for the mathematical arguments contained in the article.

\subsection*{Data availability}

Data sharing is not applicable to this article as no datasets were generated or analyzed during the current study.

\subsection*{Competing interests}

The author has no competing interests to declare that are relevant to the contents of this article.

\end{document}